\documentclass[11pt,letterpaper]{article}

\usepackage[T1]{fontenc}
\usepackage[utf8]{inputenc}
\usepackage{lmodern}
\usepackage[margin=1in]{geometry}
\usepackage{microtype}
\usepackage{enumitem}
\usepackage{amsmath,amssymb,amsthm,mathtools}
\usepackage{graphicx}
\usepackage{xcolor}
\usepackage{booktabs}
\usepackage{float}
\usepackage{algorithm}
\usepackage{algpseudocode}
\usepackage[round]{natbib}
\usepackage{xurl}
\usepackage{hyperref}
\hypersetup{
  colorlinks=true,
  linkcolor=blue,
  citecolor=blue,
  urlcolor=blue,
  pdfauthor={Anonymous Authors},
  pdftitle={Sliced Wasserstein Barycenters: The Analysis Approach within the Barycentric Coding Model in the Wasserstein Space}
}

\newtheorem{theorem}{Theorem}
\newtheorem{lemma}{Lemma}
\newtheorem{remark}{Remark}
\newtheorem{claim}{Claim}
\newtheorem{proposition}{Proposition}
\newtheorem{example}{Example}
\newtheorem{corollary}{Corollary}

\newcommand{\R}{\mathbb{R}}
\newcommand{\Ptwo}{\mathcal{P}_2}
\newcommand{\SW}{SW_2}
\newcommand{\W}{W_2}
\newcommand{\sphere}{\mathbb{S}^{d-1}}
\newcommand{\Id}{\mathrm{Id}}

\newcommand{\ip}[2]{\left\langle #1,#2\right\rangle}

\title{Sliced Wasserstein Barycenters: The Analysis Approach within the Barycentric Coding Model in the Wasserstein Space}
\author{
  Rocío Díaz Martín\\[0.3em]
  \small Department of Mathematics\\
  \small Florida State University
  \and
  James M. Murphy\\[0.3em]
  \small Department of Mathematics\\
  \small Tufts University
}
\date{}

\begin{document}
\maketitle

\begin{abstract}

We study sliced Wasserstein barycenters from a variational perspective in Wasserstein space, with emphasis on the \emph{analysis problem} in the barycentric coding model: given a query measure and a finite dictionary of probability measures, recover simplex-constrained barycentric coordinates. We derive an explicit first-variation formula for the sliced Wasserstein barycenter functional with respect to the classical $2$-Wasserstein geometry. The resulting gradient is expressed as an average of 1D monotone transport displacements over the sphere. The stationarity equation yields a Gram-matrix criterion for estimating barycentric coordinates through a quadratic program, while its fixed-point form leads to a \emph{synthesis} iteration to new barycentric measures. For Gaussian templates, we show that every global sliced Wasserstein barycenter is Gaussian; however, the functional may admit non-Gaussian critical points. We also provide a certificate for global optimality built from the 1D transport potentials. Numerical experiments demonstrate accurate coordinate recovery on synthesized queries and illustrate the use of these coordinates for data representation and stationarity residuals for reliability assessment.
\end{abstract}

\section{Introduction}
This work is part of a bigger program on \emph{variational characterizations} of {optimal transport} (OT) barycenters.
Let $\mathcal{P}(\mathbb R^d)$ denote the space of probability measures on $\mathbb R^d$,  $\mathcal{P}_p(\mathbb R^d)$ be the subset of probability measures with finite $p$-th moment ($1<p<\infty$). Given $N\in \mathbb N$, let $[N]:=\{1,2,\dots,N\}$, fix measures $\nu_1,\dots,\nu_N\in\mathcal{P}_p(\mathbb R^d)$ and weights $\lambda=(\lambda_1,\dots,\lambda_N)\in \Delta_{N-1}:=\{\lambda\in \mathbb R^N:\lambda_j\geq0\ \forall j\in[N],\ \sum_{j=1}^N\lambda_j=1\}
$
(the $(N-1)$-dimensional probability simplex).
We consider OT barycenter functionals of the form 
\begin{equation}
    \mathcal{F}^\lambda:\mathcal{P}(\R^d)\to \R, \qquad  
    \mathcal{F}^{\lambda}(\mu)=\frac{1}{2}\sum_{j=1}^N \lambda_j \, d(\nu_j,\mu)^2, \label{eq: objective}
\end{equation}
where $d(\cdot,\cdot)$ is a transport-based discrepancy on the space of probabilities. Typical choices include the classical OT distance, namely the $2$-Wasserstein distance $W_2$ when restricting the domain to $\mathcal P_2(\mathbb R^d)$ \cite{agueh2011barycenters,alvarez2016fixed}; more generally,  $p$-Wasserstein barycenter functionals on $\mathcal P_p(\mathbb R^d)$ have also been studied by replacing the quadratic power in \eqref{eq: objective} with the corresponding $p$-power \cite{brizzi2025p}. Other variants include entropically regularized OT costs \cite{cuturi2013sinkhorn, janati2020debiased, mallery2025synthesis}, linearized OT distances \cite{wang2013linear, werenski2025linearized}, and several other transport-inspired dissimilarities between probability measures \cite{cuturi2014fast,janati2020debiased,huang2021projection}. More general settings include unbalanced OT costs on the space of Radon measures \cite{friesecke2021barycenters}, as well as Gromov-Wasserstein (GW) distances \cite{memoli2011gromov} for measures supported on different metric spaces \cite{peyre2016gromov,beier2023multimarginal,diazmartin2026gromov}.

When existence is guaranteed, a barycenter associated with the weights $\lambda\in \Delta_{N-1}$ will be denoted by
\begin{equation}\label{eq: sw barycenter point}
    \nu_\lambda\in\arg\min_{\mu}\mathcal{F}^\lambda(\mu)
\end{equation}    
i.e., $\nu_\lambda$ is a weighted center of mass of the measures $\nu_1,\ldots,\nu_N$ in the chosen transport geometry.
For the classical $2$-Wasserstein problem, the foundational work \cite{agueh2011barycenters} proves existence and, under an appropriate regularity assumption on at least one input measure, uniqueness of the barycenter. Its methods are variational and use the dual formulation of OT. Subsequent work develops a fixed-point approach \cite{alvarez2016fixed} or uses the %
Riemannian-like structure of Wasserstein space \cite{zemel2019frechet,panaretos2020invitation}.

As described in \cite{chewi2020gradient}, expression \eqref{eq: objective} can be understood in the statistical setting as observing $N$ independent, weighted realizations $\nu_1,\dots,\nu_N$ of an unknown random probability measure on $\mathbb R^d$ with law $P$. 
The goal is then to construct a single probability measure that serves as an appropriate notion of the population average under $P$. 

In the present work, we focus on the case where $d(\cdot,\cdot)$ in \eqref{eq: objective} is the sliced Wasserstein (SW) distance, $SW_2$, which replaces the full high-dimensional OT problem by an average of one-dimensional $2$-Wasserstein distances over projection directions \cite{boneel, rabin1, Bonnotte}. Thus, we study %
\begin{equation}\label{eq: SW bary functional}
  \mathcal F^\lambda(\mu)=\frac{1}{2}\sum_{j=1}^N  \lambda_j\, SW_2(\mu,\nu_j)^2, \quad \mu\in \mathcal{P}_2(\mathbb R^d). 
\end{equation}
We consider differentiation in Wasserstein space $(\mathcal{P}_2(\R^d),W_2)$, and set $\nabla_{W_2}\mathcal{F}^\lambda\equiv 0$ with the aim of characterizing critical points.

Two problems arise in the context of the %
\emph{barycentric coding model} (BCM) \cite{bonneel2016wasserstein,werenski2022measure,werenski2025linearized}, where a set of essential components, $\{\nu_i\}_{i=1}^N$, referred to as a \emph{dictionary} or \emph{templates},  is used to  encode a family of probability measures through weights $\lambda\in\Delta_{N-1}$, which play the role of barycentric coordinates in a nonlinear space of measures. Borrowing the language of harmonic analysis, the two complementary tasks within the BCM  are:
\begin{enumerate}[leftmargin=*]
    \item \emph{The Synthesis Problem:} The goal is to identify the weighted center of mass $\nu_\lambda$, given $\lambda\in\Delta_{N-1}$ and $\{\nu_i\}_{i=1}^N$. That is, to solve problem \eqref{eq: sw barycenter point}.
    \item \emph{The Analysis Problem:} The goal is to determine weights $\lambda$ for a given measure $\mu$ assumed to lie in set of barycenters generated by fixed probability measures $\{\nu_i\}_{i=1}^N$. That is, we assume  $\mu=\nu_\lambda$ as in \eqref{eq: sw barycenter point} for some unknown vector $\lambda\in\Delta_{N-1}$, and we aim to determine such coordinates.
\end{enumerate}
During \emph{synthesis}, we map coordinates to a measure ($\lambda\mapsto \nu_\lambda$), whereas \emph{analysis} seeks an inverse coordinate representation. The latter may be non-unique when different weight vectors generate the same barycenter.

\paragraph{Related Work.}
\begin{itemize}[leftmargin=*]
    \item \emph{Barycentric Coding Model.} The BCM belongs to a broader class of representation  methods in which data are encoded by coordinates relative to a finite set of   reference elements, beyond linear bases in vector spaces. Related   paradigms include dictionary learning \cite{tovsic2011dictionary} and     matrix factorization \cite{lee2000algorithms}, while         applications range from data compression, to generative modeling, clustering, and embedding in machine learning \cite{DeRoseMeyer2006, zhang2014local, hormann2017generalized, warren2007barycentric}.
    In OT, 
    Wasserstein barycentric coordinates were proposed in
    \cite{bonneel2016wasserstein} for   histogram regression and related inverse
    problems.
            The analysis problem was subsequently studied using the
    geometry of the $2$-Wasserstein space in \cite{werenski2022measure},                 and later developed for  linearized Wasserstein barycenters         \cite{werenski2025linearized}. For GW, \cite{diazmartin2026gromov} formulate the same synthesis--analysis dichotomy and develop fixed-point and differentiation-based approaches, building  on \cite{peyre2016gromov} and \cite{chowdhury2020gromov}. %
    These works distinguish global optimality from the Karcher condition. The latter requires the first-variation of the 
functional \eqref{eq: objective} to vanish at the candidate point, and is therefore a necessary
stationarity condition. Without
additional convexity or structural hypotheses, stationarity need not
imply that the candidate is a global barycenter; it may nevertheless provide a
tractable surrogate for the analysis problem.

    \item \emph{SW Barycenters.} 
    Foundational references on sliced OT and 
projection-based variants include
\cite{rabin1,Bonnotte,boneel,kolouri2019generalized,
deshpande2019maxsliced,paty2019subspace}; see also
\cite{nguyen2025introduction} for a recent survey.
Recent work on the
metric, topological, and geometric properties of spaces of
probability measures endowed with sliced transport distances, and
on their comparison with classical Wasserstein spaces, includes
\cite{kitagawa2026sliced,carlier2025sharp,park2025geometry}. 
Recent work has also investigated the
critical-point structure and gradient dynamics of
SW objectives \cite{vauthier2025towards}.
Related projection and matching procedures appear in
\cite{pitie2007automated,li2024approximation,li2023measure}. 
Sliced and Radon Wasserstein barycenters were introduced and
studied computationally in \cite{boneel}. The
connection between sliced multi-marginal transport and barycentric
averaging under SW is developed in
\cite{cohen2021sliced}. 
Broader frameworks include \cite{kitagawa2024disintegrated, kitagawa2026barycenters}.

    \item \emph{Strategies and Techniques.} To compute Wasserstein gradients, we use Otto's formal Riemannian calculus \cite{otto2001geometry} together with 
    \cite{ambrosio2005gradient}; see also the first-order optimization framework in \cite{lanzetti2025firstorder}. 
    The overall approach to the analysis problem follows \cite{werenski2022measure}.     The fixed-point procedure and the treatment of the Gaussian case are inspired by \cite{alvarez2016fixed}, developed for classical $2$-Wasserstein barycenters. For global minimizers, we follow the  strategy of \cite{agueh2011barycenters}. 
\end{itemize}

\paragraph{Contributions and Organization.} Section \ref{sec: prelim} provides preliminaries on the SW distance. Section \ref{sec: critical points} computes the Wasserstein first-variation of the barycentric functional \eqref{eq: SW bary functional}: Proposition \ref{thm: gradient} gives the averaged 1D displacement formula, and Corollary \ref{cor:direct-necessary-stationarity} gives the resulting necessary stationarity condition. By using these results, Section \ref{sec: analysis} formulates the analysis problem as a simplex-constrained quadratic program, and Section \ref{sec: fix-point} provides a fixed-point reformulation used in a descent scheme for the synthesis update. 
Section \ref{sec: Gaussian} treats the particular case of Gaussian probability measures, showing that global barycenters of Gaussian templates are Gaussian, although critical points might not be Gaussian, and specializes the analysis problem to a linear system in the weights. Section \ref{sec: global} gives a criterion for global minimality and an exact analysis theorem under additional compactness and regularity hypotheses.
Section~\ref{sec: alg and experiments} presents our algorithms, with particular emphasis on the analysis problem, together with numerical experiments and applications illustrating how the proposed approach enables fast estimation of barycentric coordinates.
The appendices contain the proofs of our results, additional experiments, and a link to the anonymized code, and recalls some well-known properties for completeness.

\section{Preliminaries}
\label{sec: prelim}

For $d\geq2$, the $2$-SW distance \cite{boneel, rabin1, Bonnotte} is defined by
\begin{equation}\label{eq: SW}
    \SW^2(\nu,\mu)
    :=
    \int_{\sphere} \W\bigl((P_\theta)_\#\nu,(P_\theta)_\#\mu\bigr)^2 \, d\sigma(\theta),    
\end{equation}
where $\sigma$ denotes the normalized uniform surface measure on $\sphere$, and $P_\theta:\R^d\to \R$ is the scalar projector with respect to the $\theta\in \sphere$ direction, that is, $P_\theta(x)=x\cdot \theta$, and therefore $W_2$ in \eqref{eq: SW} is the 1D Wasserstein distance in $\mathcal{P}_2(\R)$.  For simplicity in the notation, we use $\mu^\theta:=(P_\theta)_\#\mu$.
We note that $SW_2^2$ is convex in the following sense (see \cite{nguyen2025introduction}):  
Let $\nu_1,\nu_2,\mu \in \mathcal{P}_2(\mathbb{R}^d)$, and define the convex combination
$
\nu_t=t\nu_1+(1-t)\nu_2$, $t\in[0,1]$. 
Then,
$SW_2(\nu_t,\mu)^2
\leq t\,SW_2(\nu_1,\mu)^2+(1-t)\,SW_2(\nu_2,\mu)^2$.
Notice that this is a notion of \emph{flat convexity}, which is not the same as \emph{geodesic convexity} in $(\Ptwo(\R^d),W_2)$.

Throughout this work, we will repeatedly use that
\begin{equation}\label{eq: Id/d}
\int_{\sphere}\theta\theta^\top\,d\sigma(\theta)
=\tfrac1d\Id,
\end{equation}    
where $\Id:\mathbb R^d\to\mathbb R^d$ denotes the identity map. 

\section{SW Barycenters as Critical Points in Wasserstein Space}\label{sec: critical points}

In this section we study critical points of the SW barycenter functional \eqref{eq: SW bary functional} in Wasserstein space.
That is, the metric used for differentiation is the ambient Wasserstein distance $W_2$, although the functional itself is defined using $SW_2$. 
Our primary goal is to calculate the Wasserstein gradient of \eqref{eq: SW bary functional}.
In ordinary calculus, if $F:\R^d\to\R$ is differentiable, the gradient is the vector satisfying
$\left.\frac{d}{dt}\right|_{t=0}F(x+tv)=\nabla F(x)\cdot v$ for every $v\in\mathbb R^d$.
Computing Wasserstein gradients follows the same basic idea, but a ``point'' is now a probability measure $\mu$, and a direction is represented by a velocity field $v:\mathbb R^d\to\mathbb R^d$ that moves each particle $x$. This is the viewpoint of \cite{otto2001geometry,ambrosio2005gradient}:
Given a functional $\mathcal{F}$ over $\mathcal{P}_2(\mathbb R^d)$, such as \eqref{eq: objective},
if we consider the curve of probability measures $\mu_t=(\Id+t v)_\#\mu$,
then a vector field $G\in L^2(\mu;\R^d):=\{G:\mathbb R^d\to\mathbb R^d:\, \int_{\mathbb R^d} \|G(x)\|^2d\mu(x)<\infty\}$
represents the \emph{transport first-variation} of $\mathcal F$ at $\mu$ if, %

\begin{equation*}
   \frac{d}{dt}\Big|_{t=0}\mathcal F(\mu_t)
   =\int_{\R^d}\ip{G(x)}{v(x)}d\mu(x) \quad \forall v\in C_c^1(\mathbb R^d;\mathbb R^d).
\end{equation*}   

In general, for a time-dependent curve $\mu_t$, the velocity field $v_t$ is described by the continuity equation
\begin{equation}\label{eq: cont eq in Rd}
\partial_t \mu_t + \nabla\cdot(\mu_tv_t)=0,\quad
        \mu_0=\mu
\end{equation}
(where the equation is understood in the sense of distributions).
This equation is the mass-conservation law: mass moves with velocity $v_t$ and is neither created nor destroyed. \cite[Thm~8.3.1]{ambrosio2005gradient} shows that absolutely continuous curves in Wasserstein space are exactly the curves that can be represented this way, with an appropriate velocity field.  The Wasserstein tangent space is the closure in $L^2(\mu;\mathbb R^d)$ of $\{\nabla\psi:\psi\in C_c^\infty(\mathbb R^d)\}$
\cite[Def.~8.4.1]{ambrosio2005gradient}.
When $G$ belongs to this space and satisfies the preceding identity, we call it the \emph{Wasserstein gradient in the transport (Otto) sense} and write $G=\nabla_{W_2}\mathcal F(\mu)$.
It is a vector field on $\mathbb R^d$, defined $\mu$-almost everywhere, rather than a single vector in $\mathbb R^d$.

The transport first-variation should be distinguished from the \emph{scalar first-variation}, or functional derivative, denoted by
$\frac{\delta\mathcal F}{\delta\mu}(\mu)$.
When it exists, this scalar function represents differentiation along additive perturbations:
$
   \left.\frac{d}{dt}\right|_{t=0}\mathcal F((1-t)\mu+t\rho)
   =\int_{\mathbb R^d}\frac{\delta\mathcal F}{\delta\mu}(\mu)(x)\,d(\rho-\mu)(x).
$   
Under regularity and integrability assumptions that justify the chain rule and integration by parts in \eqref{eq: cont eq in Rd},
$
   \frac{d}{dt}\mathcal F(\mu_t)
   =\int_{\mathbb R^d}
     \nabla_x\frac{\delta\mathcal F}{\delta\mu}(\mu_t)(x)
     \cdot v_t(x)\,d\mu_t(x).
$ Here $\nabla_x$ is the ordinary Euclidean gradient with respect to $x$.
Consequently, when the spatial gradient belongs to the tangent space,
   $\nabla_{W_2}\mathcal F(\mu)(x)
   =\nabla_x\frac{\delta\mathcal F}{\delta\mu}(\mu)(x)$
   $\mu\text{-a.e.}$ \cite[Chs.~7--8]{Santambrogio-OTAM}.
This provides a useful interpretation of the Wasserstein gradient, but Proposition \ref{thm: gradient} computes the transport derivative directly.%

Let us fix measures $\nu_1,\ldots,\nu_N\in\Ptwo(\R^d)$ and weights
$\lambda=(\lambda_1,\ldots,\lambda_N)\in\Delta_{N-1}$.
We state that the Wasserstein gradient at $\mu$ of the functional $\mathcal{F}^\lambda$ defined by \eqref{eq: SW bary functional}
is obtained by averaging 1D transport displacements over all directions $\theta\in\mathbb S^{d-1}$.

\paragraph{Standing Assumptions:} Given $\mu\in\mathcal P_2(\mathbb R^d)$, we assume that, for $\sigma$-a.e. $\theta\in\sphere$, the projected measure $\mu^\theta$ is non-atomic so that, for each $j\in[N]$, there exists a unique ($\mu^\theta$-a.e.) nondecreasing OT map from $\mu^\theta$ to $\nu_j^\theta$ \cite[Thms.~2.5, 2.9]{Santambrogio-OTAM}, denoted by
\begin{equation}\label{eq: T_theta}
    T_{j,\theta}^\mu:=T_{\mu^\theta\to\nu_j^\theta}.
\end{equation}
For example, notice that, if $\mu$ is absolutely continuous with respect to Lebesgue measure on $\R^d$, then every affine hyperplane has $\mu$-mass zero. Hence, for every $t\in\R,\theta\in\sphere$,
  $\mu^\theta(\{t\})
  =\mu\bigl(\{x\in\R^d:x\cdot\theta=t\}\bigr)=0$.
Thus, $\mu^\theta$ is non-atomic. %
\begin{proposition}
\label{thm: gradient}
Let $\mu\in\Ptwo(\R^d)$. Assume that, for $\sigma$-a.e. $\theta\in\sphere$, $\mu^\theta$ is non-atomic, and for each such $\theta$ and each $j\in[N]$, let $T_{j,\theta}^\mu$ be as in \eqref{eq: T_theta}. For each $j\in[N]$, we denote by $U_j^\mu:\R^d\to\R^d$ the vector field
\begin{equation}\label{eq: Uj}
    U_j^\mu(x):=\int_{\sphere}
\bigl(x\cdot\theta - T_{j,\theta}^\mu(x\cdot\theta)\bigr)\,\theta
\,d\sigma(\theta).
\end{equation}
Then $
   \frac{d}{dt}\Big|_{t=0}\mathcal F^\lambda(\mu_t)=\int_{\R^d}\langle\sum_{j=1}^N\lambda_j\,U_j^\mu(x),{v(x)}\rangle d\mu(x)$,   for every perturbation
$\mu_t=(\Id+t v)_\#\mu$ with $v\in C_c^1(\mathbb R^d;\mathbb R^d)$,
and the unique tangent-space vector field representing the transport first-variation is
\begin{equation}\label{eq: grad of sw bary}
   \nabla_{W_2}\mathcal F^\lambda(\mu)
   =\sum_{j=1}^N\lambda_j\,U_j^\mu.
\end{equation}    
\end{proposition}

Formula \eqref{eq: grad of sw bary} is consistent with the sliced first-variation formula
in 
\cite[Prop.~5.1.7]{Bonnotte}, obtained for measures supported on the same compact set $K\subset\R^d$ and with $\mu$ absolutely continuous; see also \cite[Prop. 4.7(c)]{vauthier2025towards}. We provide an alternative derivation in a broader setting.
\\

\textbf{Proof sketch.}
We combine two facts: (1) Projecting an absolutely continuous curve of measures $\mu_t$ as in \eqref{eq: cont eq in Rd} onto each direction and identifying the projected velocity  by conditional expectation $v_t^\theta(s):=
    \mathbb{E}\bigl[v_t(X)\cdot\theta \,\big|\, X\cdot\theta=s\bigr]$,
     $X\sim \mu_t$, results in mass-conservation $\partial_t \mu_t^\theta + \partial_s(\mu_t^\theta v_t^\theta)=0$ on each slice, with $\mu_0^\theta=\mu^\theta$. (2) The gradient of $\mathcal{P}_2(\R)\ni\rho\mapsto \frac12 W_2^2(\nu_j^\theta,\rho)$ is given by the optimal displacement in 1D. The main technical point is justifying differentiation under the integral sign over the sphere.
\\

Our first-order viewpoint is consistent with the general optimality
theory for Wasserstein space optimization in \cite[Thm.~3.2]{lanzetti2025firstorder}: In the Wasserstein geometry, at every point of Wasserstein differentiability, a local minimizer (or maximizer) must have vanishing Wasserstein gradient; see also the subdifferential framework of \cite[Ch.~10]{ambrosio2005gradient}.
The stationarity condition below is proved directly from
the sliced {transport} first-variation formula.

\begin{corollary}
\label{cor:direct-necessary-stationarity}
Assume the hypotheses of Proposition~\ref{thm: gradient}. Let
$\bar\mu$ be a local minimizer of $\mathcal F^\lambda$ in
$(\mathcal P_2(\mathbb R^d),W_2)$. Then $\sum_{j=1}^N\lambda_jU_j^{\bar\mu}(x)=0$
$\bar\mu\text{-a.e. }x.$
\end{corollary}

\textbf{Proof sketch.}
Local minimality along $(\Id+tv)_\#\bar\mu$ and
Proposition~\ref{thm: gradient} imply that
$\sum_j\lambda_jU_j^{\bar\mu}$ is orthogonal to every
$v\in C_c^1(\mathbb R^d;\mathbb R^d)$.
Density of these fields in $L^2(\bar\mu;\mathbb R^d)$
then gives the conclusion.

\section{The Analysis Problem}\label{sec: analysis}
In this section we study the analysis problem. Precisely, assuming we are given $\mu$ in the SW barycenter space generated by templates $\{\nu_j\}_{j=1}^N$, i.e.,  $\mu=\nu_\lambda$ as in \eqref{eq: sw barycenter point} for an unknown $\lambda\in \Delta_{N-1}$ (not necessarily unique), and
our goal is to recover such a $\lambda\in \Delta_{N-1}$. 

Following the strategy in \cite{werenski2022measure} for classical Wasserstein barycenters, we characterize the norm of the Wasserstein gradient of $\mathcal F^\lambda$ (Lemma \ref{lem:norm-of-grad}), leading to a quadratic program in $\lambda$ (Proposition \ref{prop:matrix-characterization}).  
\begin{lemma}\label{lem:norm-of-grad} Under the hypotheses of Proposition \ref{thm: gradient}, $\|\nabla_{W_2}\mathcal{F}^\lambda(\mu)\|_{L^2(\mu;\mathbb R^d)}^2=\lambda^\top\mathcal{A}(\mu)\lambda$,
                where, for $U_j^\mu$ as in \eqref{eq: Uj}, $\mathcal{A}(\mu)\in\R^{N\times N}$ is the positive semi-definite matrix 
\begin{equation}\label{eq: A}
        \mathcal{A}_{i,j}(\mu):=\langle U_i^\mu,U_j^\mu\rangle_{L^2(\mu;\mathbb R^d)}.
    \end{equation}
\end{lemma}

\textbf{Proof sketch.}
Insert $\nabla_{W_2}\mathcal F^\lambda(\mu)=\sum_j\lambda_jU_j^\mu$ and expand its squared $L^2(\mu;\mathbb R^d)$ norm. The resulting coefficients form the Gram matrix $\mathcal A(\mu)$, which is automatically positive semi-definite.

We now  provide a characterization of  critical points of the SW barycenter functional \eqref{eq: SW bary functional}, which will serve in the analysis problem, i.e., when the task is to retrieve weights $\lambda\in \Delta_{N-1}$ of a measure $\mu$ assumed in the barycenter space generated by fixed measures $\{\nu_j\}_{j=1}^N$.

\begin{proposition}\label{prop:matrix-characterization}
Let $\mu\in\mathcal P_2(\mathbb R^d)$ be such that $\mu^\theta$ is non-atomic
for $\sigma$-a.e. $\theta\in\mathbb S^{d-1}$. Let $\mathcal{A}(\mu)$ be defined by \eqref{eq: A}.
Then there exists $\lambda^*\in\Delta_{N-1} $ such that $\nabla_{W_2}\mathcal F^{\lambda^*}(\mu)=0$
in $L^2(\mu;\mathbb R^d)$ if and only if
\begin{equation}\label{eq: mat cond}
 \min_{\lambda\in\Delta_{N-1} }\lambda^\top\mathcal A(\mu)\lambda=0.   
\end{equation}
\end{proposition}

\textbf{Proof sketch.}
Using Lemma \ref{lem:norm-of-grad}, $\lambda^\top\mathcal A(\mu)\lambda=0$ exactly when the Wasserstein gradient vanishes. Compactness of $\Delta_{N-1}$ guarantees that the minimum is attained. 
\\

The matrix condition \eqref{eq: mat cond} characterizes the existence of a
simplex weight vector $\lambda^*$ for which $\mu$ is a stationary point of the SW
barycentric functional. %
In general, it does not imply that $\mu$ is a (global) SW barycenter. However, 
it provides a diagnostic test: if $\mu$ is a SW barycenter for some $\lambda$, then necessarily
$\lambda^\top\mathcal A(\mu)\lambda=0$. 

\section{Fixed-Point  Synthesis}\label{sec: fix-point}

Proposition \ref{thm: gradient} and its Corollary \ref{cor:direct-necessary-stationarity} show that a critical point $\mu$ of $\mathcal{F}^\lambda$ for the weight vector $\lambda\in \Delta_{N-1}$ satisfies $\sum_{j=1}^N\lambda_j \, U_j^\mu(x)=0$
for $\mu\text{-a.e. }x$.
Using the identity \eqref{eq: Id/d},
the critical-point condition can be written as
\begin{equation}\label{eq: c.p. condition as f.p.}
    {
 x
   =\underbrace{d\sum_{j=1}^N\lambda_j
     \int_{\sphere}T_{j,\theta}^\mu(x\cdot\theta)\theta\,d\sigma(\theta)}_{B_\mu(x)},
   \quad \mu\text{-a.e. }x.}
\end{equation} 
This is the sliced analogue of the familiar Wasserstein barycenter condition
$x=\sum_{j=1}^N\lambda_jT_j(x)$,
where $T_j$ transports the barycenter to $\nu_j$ \cite{agueh2011barycenters,alvarez2016fixed}. Inspired by the fixed-point strategy in
\cite[Sec. 3]{alvarez2016fixed},
Proposition \ref{prop:sliced-fixed-point-descent} provides a descent estimate, which is the key for a fixed-point descent scheme aimed at approximating
Wasserstein critical points of $\mathcal F^\lambda$. 

\begin{proposition}
\label{prop:sliced-fixed-point-descent}
Let $\mu^+:=(B_\mu)_\#\mu$. Under the hypotheses of Proposition~\ref{thm: gradient}, we have the descent estimate
\begin{align}\label{eq:sliced-descent-strong}
\mathcal F^\lambda(\mu)
&\geq \mathcal F^\lambda(\mu^+)
+\tfrac{1}{2d}\int_{\R^d}|x-B_\mu(x)|^2\,d\mu(x) \geq \mathcal F^\lambda(\mu^+)
+\tfrac{1}{2d}W_2(\mu,\mu^+)^2.
\end{align}    
In particular, $\mathcal F^\lambda(\mu^+)
\le
\mathcal F^\lambda(\mu)$.
\end{proposition}

\textbf{Proof sketch.}
For each $\theta\in \mathbb S^{d-1}$, the transport plan
$(P_\theta\circ B_\mu,\,
T_{j,\theta}^\mu\circ P_\theta)_\#\mu$ 
couples $(\mu^+)^\theta$ and $\nu_j^\theta$.
Bounding the transport costs by these couplings 
gives the first inequality in \eqref{eq:sliced-descent-strong}.
The fact that $B_\mu$ is a transport map from $\mu$ to $\mu^+$ (see \eqref{eq: ite}),
gives the second inequality.
\\

Proposition \ref{prop:sliced-fixed-point-descent} motivates the fixed-point
iteration
\begin{equation}\label{eq: ite}
\mu_{n+1}:=(B_{\mu_n})_\#\mu_n.    
\end{equation}
In Section \ref{sec: alg and experiments} and Appendix \ref{sec: synt}, we describe its implementation for
approximating Wasserstein critical points of
$\mathcal F^\lambda$ and discuss its relation with the 
methods in \cite{boneel}.

\section{The Gaussian Case}\label{sec: Gaussian}
Inspired by \cite{alvarez2016fixed}, we study the case where the templates are Gaussians. 
We use the notation $\mathbb S_{++}^d$ for the set of $d\times d$ symmetric  positive definite matrices. 
Our first result shows that every SW barycenter of Gaussian measures is Gaussian.

\begin{theorem}\label{thm: gaussian-sliced-barycenters}
Let $\nu_1,\ldots,\nu_N\in\mathcal P_2(\mathbb R^d)$ be Gaussian
measures $\nu_i=\mathcal N(m_i,S_i)$, $i\in[N]$, and let $\lambda\in \Delta_{N-1}$. Assume that for at least one $i_0\in[N]$, the
covariance matrix $S_{i_0}\in \mathbb S_{++}^d$ with $\lambda_{i_0}>0$. 
Then,
\begin{equation}\label{eq: G}
    \inf_{\mu\in\mathcal P_2(\mathbb R^d)}
\mathcal F^\lambda(\mu)
=
\inf_{\substack{\mu\in\mathcal P_2(\mathbb R^d)\\
\mu\text{ Gaussian}}}
\mathcal F^\lambda(\mu),
\end{equation}
and every minimizer $\overline \mu$ of $\mathcal F^\lambda$ over
$\mathcal P_2(\mathbb R^d)$ is Gaussian.
\end{theorem}

\textbf{Proof sketch.}
Replace an arbitrary candidate $\mu\in \mathcal{P}_2(\mathbb R^d)$ by the Gaussian $\gamma_\mu:=\mathcal N(m_\mu,C_\mu)$ with the same mean $m_\mu$ and covariance $C_\mu$ as $\mu$, and apply the arguments in \cite[Thms 2.3 and 2.4]{alvarez2016fixed}, using the 1D Gelbrich inequality. Then $\mathcal F^\lambda(\gamma_\mu)
\leq
\mathcal F^\lambda(\mu)$, from which one derives \eqref{eq: G}.
\\

Recent work has shown that the critical-point structure of SW objectives can be substantially richer than their set of minimizers \cite{vauthier2025towards}.
Example \ref{example: non Gaussian crit point} shows that %
Wasserstein critical points of the SW barycenter functional \eqref{eq: SW bary functional} need not
be Gaussian.

\begin{example}\label{example: non Gaussian crit point}
Let $\nu=\mathcal N(0,\Id)$, and consider
    $\mathcal F(\mu):=\frac12 SW_2^2\left(\mu,\mathcal N(0,\Id)\right)$.
There exists $r_*>0$ such that
$\mu_{r_*}=\operatorname{Unif}(r_*\mathbb S^{d-1})$ is a non-Gaussian Wasserstein critical point of $\mathcal F$ but it is not a global minimizer.
\end{example}

\textbf{Proof sketch.}
Although $\mu_r=\operatorname{Unif}(r\mathbb S^{d-1})$, $r>0$, is singular,
its 1D projections are non-atomic, so the Wasserstein
gradient formula applies. Rotational symmetry of $\mu_r$ and the
Gaussian target reduces stationarity to an affine scalar equation
in $r$, whose positive solution $r_*$ makes the gradient vanish
on the entire sphere.
Finally, since $SW_2$ is a metric and
$\mu_{r_*}\neq\mathcal N(0,\Id)$,
we have
$\mathcal F(\mu_{r_*})>0
=
\mathcal F(\mathcal N(0,\Id))$.
Hence $\mu_{r_*}$ is a critical point but not a global minimizer.\\

By Theorem~\ref{thm: gaussian-sliced-barycenters}, the search for global
SW barycenters of Gaussian targets may be restricted, without loss of
generality, to Gaussian candidates: Given Gaussian templates $\{\nu_j=\mathcal N(m_j,S_j)\}_{j=1}^N$,  $m_j\in\mathbb R^d$, $S_j\in\mathbb S_{++}^d$,  and
$\lambda\in \Delta_{N-1}$,
for Gaussian candidates $\mu=\mathcal N(m,S)$, $m\in\mathbb R^d$, $S\in\mathbb S_{++}^d$,
we define the restricted objective
\begin{align}\label{eq: SW obj for Gaussians}
   2 \, \mathcal F_G^\lambda(m,S):&=
\sum_{j=1}^N\lambda_j
SW_2\bigl(\mathcal N(m,S),\mathcal N(m_j,S_j)\bigr)^2\\
&=\sum_{j=1}^N\lambda_j\left[\tfrac{1}{d}|m-m_j|^2\notag+\int_{\sphere}
(\sqrt{\theta^\top S\theta}-\sqrt{\theta^\top S_j\theta})^2
d\sigma(\theta)\right]\notag
\end{align} 
For the second equivalent identity, we used the closed-form of the $W_2$ distance between 1D Gaussians and integrated over the sphere applying \eqref{eq: Id/d} to get the mean term.
Restricted to the space of Gaussian measures in $\mathbb R^d$ (characterized by their mean vector and covariance matrix), the functional $\mathcal F_G^\lambda$ in \eqref{eq: SW obj for Gaussians} replaces the general functional $\mathcal{F^\lambda}$ defined on the full space $\mathcal P_2(\mathbb R^d)$ by \eqref{eq: SW bary functional}.

\begin{proposition}\label{prop: gaussian bary}
Let $\nu_j=\mathcal N(m_j,S_j)$, $j\in[N]$,
where $m_j\in\mathbb R^d$, $S_j\in\mathbb S_{++}^d$,  let
$\lambda\in\Delta_{N-1}$, and 
consider the Gaussian-restricted SW barycenter functional $\mathcal{F}_G^\lambda$  in \eqref{eq: SW obj for Gaussians}.   
If $\overline{m}:=\sum_{j=1}^N\lambda_j m_j$ and 
$S\in \mathbb S_{++}^d$ is such that
   \begin{equation}\label{eq: Gaussian crit point}
   \int_{\mathbb S^{d-1}}
\left[
1-
\frac{\sum_{j=1}^N\lambda_j\sqrt{\theta^\top S_j\theta}}
{\sqrt{\theta^\top S\theta}}
\right]
\theta\theta^\top
\,d\sigma(\theta)
=0, 
\end{equation}
then $\mathcal N(\overline m,S)$ minimizes $\mathcal F_G^\lambda$ over Gaussian candidates.
Moreover, 
$\mathcal N(\overline m,S)$ is a global SW barycenter of
$\nu_1,\ldots,\nu_N$ over $\mathcal P_2(\mathbb R^d)$.
\end{proposition}

\textbf{Proof sketch.} From \eqref{eq: SW obj for Gaussians}, we see that the mean and covariance variables decouple, giving rise to two convex objectives: It is easy to see that the optimal mean is $ \overline m$, while differentiating the matrix part gives the condition \eqref{eq: Gaussian crit point}. 
   Finally, Theorem~\ref{thm: gaussian-sliced-barycenters} lifts Gaussian optimality to the full Wasserstein space. \\
 
For the isotropic centered Gaussian case, condition \eqref{eq: Gaussian crit point} collapses completely to a scalar equation:

\begin{corollary}
\label{cor:isotropic-gaussian-sliced-condition}
Let
$\nu_j=\mathcal N(0,\sigma_j^2\Id)$ with
$\sigma_j>0$ for
$j\in[N]$,
and let $\lambda\in\Delta_{N-1}$. Define
$\overline{\sigma}
:=
\sum_{j=1}^N\lambda_j\sigma_j$.
Then
$\mathcal N(0,\overline{\sigma}^{\,2}\Id)
$
minimizes the Gaussian-restricted SW barycenter
functional and, moreover, it is a global SW barycenter of
$\nu_1,\ldots,\nu_N$ over $\mathcal P_2(\mathbb R^d)$.
\end{corollary}
See Appendix \ref
{app:gauss-analysis}: the analysis problem in the Gaussian case is a 
linear problem.

\section{Global Minimizers}\label{sec: global}
Inspired by  \cite{agueh2011barycenters}, we close our theoretical analysis by studying global minimizers for our SW barycenter functional $\mathcal{F}^\lambda$ given by \eqref{eq: SW bary functional}. 

\paragraph{Standing Assumptions:}
In this section, $K\subset\mathbb R^d$ is a compact convex set, the templates $\nu_1,\ldots,\nu_N$ belong to $\mathcal P(K)$, the space of probability measures supported on $K$. For $\lambda\in\Delta_{N-1}$, we restrict $\mathcal F^\lambda$ to $\mathcal P(K)$. 
Given $\mu\in\mathcal P(K)$, assume that $\mu^\theta$ is non-atomic for $\sigma$-a.e. $\theta\in\mathbb S^{d-1}$, and let $T_{j,\theta}^\mu$ be as in \eqref{eq: T_theta} with primitive $
    u_{j,\theta}^\mu(s):=\int_0^s T_{j,\theta}^\mu(r)\,dr$, and
a Kantorovich potential for the  cost $c(s,t)=\frac12|s-t|^2$ between $\mu^\theta$ and $\nu_j^\theta$
\begin{equation}\label{eq: Kant pot}
      \varphi_{j,\theta}^\mu(s)
    :=\tfrac12s^2-u_{j,\theta}^\mu(s).
\end{equation}    

\begin{theorem}
\label{thm:SW-supporting-potential}
In addition to the standing assumptions considered in this section, assume that the map
\begin{equation}\label{eq: functional in terms of kant pot}
    \Phi_\mu(x)
    :=
    \sum_{j=1}^N\lambda_j
    \int_{\mathbb S^{d-1}}
        \varphi_{j,\theta}^\mu(x\cdot\theta)
    \,d\sigma(\theta)
\end{equation}        
is well defined and continuous on $K$.
Suppose that there exists a constant $c_\mu\in\mathbb R$ such that
\begin{align}\label{eq:supporting-potential-condition}
\Phi_\mu(x)&\geq c_\mu
&\forall x\in K,
\quad \Phi_\mu(x)&=c_\mu
&\mu\text{-a.e. }x. 
\end{align}
Then $\mu$ is a global minimizer of the functional \eqref{eq: SW bary functional} on
$\mathcal P(K)$, that is,
$
    \mathcal F^\lambda(\mu)
    =
    \min_{\rho\in\mathcal P(K)}
    \mathcal F^\lambda(\rho).
$
\end{theorem}

\textbf{Proof sketch.}
Adapting the duality argument of
\cite[Prop.~3.8]{agueh2011barycenters}, we use the optimal
dual potentials at $\mu^\theta$ to bound the transport costs
of each projected competitor $\rho^\theta$.
Summing over templates and averaging over directions  yields
$\mathcal F^\lambda(\rho)-\mathcal F^\lambda(\mu)
\geq \int_K\Phi_\mu\,d(\rho-\mu)\geq0$ by using \eqref{eq:supporting-potential-condition}.

\begin{corollary}
\label{cor:SW-Karcher-to-barycenter}
In addition to the standing assumptions considered in this section, 
suppose that $T_{j,\theta}^{\mu}$ is continuous on $P_\theta(O)$ for some open neighborhood $O$ of $K$,
for each $j\in[N]$, $\sigma$-a.e.
$\theta\in\mathbb S^{d-1}$. 
If the sliced Karcher equation $\sum_{j=1}^N\lambda_j U_j^\mu(x)=0$ holds for every $x\in K$, then $\mu$ is a global minimizer of
$\mathcal F^\lambda$ on $\mathcal P(K)$.
\end{corollary}

\textbf{Proof sketch.}
The additional hypothesis leads to $\nabla\Phi_\mu
    =
    \sum_{j=1}^N\lambda_j U_j^\mu$
    on $O$. If it vanishes on the convex set $K$, integration along line segments shows that $\Phi_\mu$ is constant on $K$, and Theorem \ref{thm:SW-supporting-potential} gives global minimality.\\

We note that Corollary~\ref{cor:SW-Karcher-to-barycenter} does not conflict with the existence of
non-minimizing critical points, as in Example~\ref{example: non Gaussian crit point}, where, trivially, the Gaussian target is not supported in any compact set $K$, and therefore does not fall under the standing assumption of this section.

\begin{remark}
\label{rem:pointwise-vs-ae-stationarity}
The hypothesis of
Corollary~\ref{cor:SW-Karcher-to-barycenter} is strictly stronger
than ordinary Wasserstein stationarity. Indeed, the latter only gives \eqref{eq: c.p. condition as f.p.}, equivalently,
\begin{equation}\label{eq: Karcher pp}
        \sum_{j=1}^N\lambda_jU_j^\mu(x)=0
    \qquad
    \mu\text{-a.e. }x,
\end{equation}    
because the Wasserstein gradient is an element of
$L^2(\mu;\mathbb R^d)$. In contrast,  Corollary~\ref{cor:SW-Karcher-to-barycenter} assumes 
\begin{equation}\label{eq: Karcher}
        \sum_{j=1}^N\lambda_jU_j^\mu(x)=0
    \qquad
    \text{for every }x\in K.
\end{equation}    
Requiring the identity on all of $K$, rather than only
$\mu$-a.e., allows us to integrate $\nabla\Phi_\mu$ along
every line segment in $K$. Thus $\Phi_\mu$ is constant on
$K$, and the supporting-potential condition follows.
\end{remark}

\begin{remark}
The hypothesis \eqref{eq:supporting-potential-condition} is the sliced analogue of the dual optimality condition for classical Wasserstein barycenters in \cite[Prop. 3.8]{agueh2011barycenters}; see also \cite[Sec. 3.1.5]{panaretos2020invitation}. A detailed comparison is given in Remark \ref{eq: analogous to AC} in the appendix.
\end{remark}

The regularity assumption in Corollary~\ref{cor:SW-Karcher-to-barycenter}
can be replaced by the stronger, directly checkable uniform Lipschitz requirement 
for all $j\in [N]$ and $\sigma$-a.e.\ $\theta$:
\begin{equation}\label{eq: unif lip cond}
    |T_{j,\theta}^\mu(s)-T_{j,\theta}^\mu(t)|
    \leq \mathrm{L}|s-t|
    \qquad
    \forall s,t\in\mathbb R.
\end{equation}    

\begin{corollary}
\label{cor:exact-SW-analysis}
In addition to the standing assumptions considered in this section, suppose that $\mu$ is absolutely continuous with respect to the Lebesgue measure in $\mathbb R^d$, $\operatorname{supp}\mu=K$, and that the uniform Lipschitz condition \eqref{eq: unif lip cond} holds. Then, for every $\lambda\in\Delta_{N-1}$,
$$
    \mu\in\arg\min_{\rho\in\mathcal P(K)}\mathcal F^\lambda(\rho)
    \quad\Longleftrightarrow\quad
    \mathcal A(\mu)\lambda=0.
$$
Equivalently, the full set of barycentric coordinates of $\mu$ for the problem restricted to
$\mathcal P(K)$ is $\Delta_{N-1}\cap\ker\mathcal A(\mu)$.
\end{corollary}

\textbf{Proof sketch.}
Feasible perturbations $(\Id+tv)_\#\mu\in \mathcal{P}(K)$, for small $|t|$,
and Proposition~\ref{thm: gradient} 
give necessity.
Conversely, $\mathcal A(\mu)\lambda=0$ gives stationarity
$\mu$-a.e., and continuity and full support extend it to all
of $K$, so Corollary~\ref{cor:SW-Karcher-to-barycenter}
gives sufficiency.\\

We close this section by noticing that our supporting-potential condition \eqref{eq:supporting-potential-condition} controls $\Phi_\mu$ on $K$, and this is sufficient for competitors supported there. However,
the preceding results certify global optimality over
$\mathcal P(K)$, not over all of $\mathcal P_2(\mathbb R^d)$.
For classical $W_2$ barycenters, this distinction disappears
when all templates are supported on a common compact convex
set $K$ 
\cite[Sec.~4]{agueh2011barycenters};
see also \cite[Sec.~2]{alvarez2016fixed}.
The following %
example shows that this property fails for SW barycenters.

\begin{example}\label{prop:SW-convex-support-counterexample}
Let
$a=(\frac{\sqrt3}{2},\frac12)$,
$b=(\frac{\sqrt3}{2},-\frac12)$,
and consider the templates
$
\nu_1=\frac12(\delta_a+\delta_{-a})$,
$
\nu_2=\frac12(\delta_b+\delta_{-b})$,
with equal weights $\lambda=(1/2,1/2)$.
Their supports have convex hull
$
K=
[-\frac{\sqrt3}{2},\frac{\sqrt3}{2}]
\times[-\frac12,\frac12]$.
A SW barycenter is
$\mu_*=\frac12(\delta_{Me_1}+\delta_{-Me_1})$,
$
M=\frac1\pi+\frac{\sqrt3}{3}>\frac{\sqrt3}{2}$.
Thus both points in $\operatorname{supp}\mu_*$ lie outside $K$.
Moreover, no unrestricted SW barycenter belongs to
$\mathcal P(K)$.
\end{example}

Because SW compares measures only through their 1D projections, the independently optimal behavior of all slices need not be geometrically compatible with staying inside the convex hull of the original supports.

\section{Algorithms \& Experiments}\label{sec: alg and experiments}

Through Wasserstein first-variation, one can reinterpret the SW barycenter computation in \cite{boneel}, as a fixed-point descent scheme. 
We derive the dynamics intrinsically from the Wasserstein first-variation of $\mathcal F^\lambda(\mu)$, i.e., 
\eqref{eq: grad of sw bary} in Proposition \ref{thm: gradient}, and
interpret the \emph{synthesis} procedure as the $W_2$-gradient flow
$$
    \partial_t\mu_t+\nabla\cdot(\mu_t v_t)=0,
    \qquad
    v_t=-\nabla_{W_2}\mathcal F^\lambda(\mu_t).
$$
Thus, the explicit Euler step in Wasserstein space is
$$
    \mu^{k+1}
    =
    \left(
        \operatorname{Id}
        -
        \tau\nabla_{W_2}\mathcal F^\lambda(\mu^k)
    \right)_\#\mu^k.
$$
Using \eqref{eq: Id/d},
the choice $\tau=d$ gives the fixed-point iteration %
$\mu^{k+1}=(B_{\mu^k})_\#\mu^k$ as in \eqref{eq: ite}.
Thus, the synthesis procedure for  seeking stationary points of the SW
barycenter functional can be viewed as a measure-level $W_2$-gradient
Euler scheme for the functional \eqref{eq: SW bary functional}, whose particle
implementation is closely related to the descent schemes in
\cite{boneel}. The additional information obtained from the first-variation formulation is that Proposition \ref{prop:sliced-fixed-point-descent} justifies a descent result for the iterations: $\mathcal{F}^\lambda(\mu^{k})\geq \mathcal{F}^\lambda(\mu^{k+1})$. A detailed discussion is provided in  Appendix \ref{app:synth}.

In addition, we have a simple  \emph{diagnostic test} for assessing whether a probability measure satisfies a necessary condition for being a SW barycenter. For a fixed candidate measure $\mu$, Lemma \ref{lem:norm-of-grad}
motivates  the following test: One needs to solve problem \eqref{eq: mat cond}.
If $\mu$ is a barycenter associated with a weight vector $\lambda$ (i.e., $\mu=\nu_\lambda$ as in \eqref{eq: sw barycenter point}), and the first-order condition applies (i.e., Proposition \ref{thm: gradient} and Corollary \ref{cor:direct-necessary-stationarity} are applicable to $\mu$), then $\lambda^\top \mathcal A(\mu)\lambda = 0$. Therefore, a small value of \eqref{eq: mat cond} is a useful consistency check.
Adding this diagnostic to the discrete synthesis algorithm provides a post-hoc check of how well the synthesized measure satisfies the first-order optimality condition.

At the same time,
the  differential formulation yields the  Gram matrix criterion \eqref{eq: mat cond} in Proposition \ref{prop:matrix-characterization} for tackling the analysis problem,
which is specific to the
BCM in the present work and is not the focus of the 
barycenter construction in \cite{boneel}.

Both synthesis and analysis  rely on the vector fields 
$U_j^\mu$ in \eqref{eq: Uj}. 
From a computational viewpoint, this is tractable because it reduces to averaging 1D transport displacements over directions $\theta\in \sphere$. That is, after discretizing the sphere by considering
a finite set of directions $\Theta=\{\theta_1,\dots,\theta_L\}\subset \sphere$, 
we replace the integral on $\sphere$ by  averaging
$U_j^\mu(x)
\approx U_j^\Theta(x)
:=
\frac1L \sum_{\theta\in\Theta}
\bigl(x\cdot \theta - T_{j,\theta}^\mu(x\cdot\theta)\bigr)\theta,$
where the 1D transport maps $ T_{j,\theta}^\mu$ are obtained by solving a 1D OT problem, which amounts to sorting projected samples (monotone quantile matching $\text{MQM}$). Indeed, 
we represent a  measure $\mu$ by an empirical measure $\mu_X = \frac1m \sum_{k=1}^m \delta_{X_k}$,
$X=(X_1,\dots,X_m)\in (\R^d)^m$.
Likewise, each target $\nu_j$ is approximated by a point cloud
$Y^{(j)}=(Y^{(j)}_1,\dots,Y^{(j)}_{m_j})$, $j\in[N]$.
 Evaluating \eqref{eq: grad of sw bary} at the particles $X_k$ gives the discrete gradient field $\sum_{j=1}^N  \lambda_j U_j^\Theta(X_k)$, $k=1,\dots,m$.
For \emph{analysis}, given a query point cloud $X$,  one needs to build the matrix \eqref{eq: A}, i.e., $\mathcal A_{ij}(\mu)
\approx
\frac1m \sum_{k=1}^m
\langle U_i^\Theta(X_k), U_j^\Theta(X_k)\rangle$,
in order to  retrieve estimated coordinates $\widehat\lambda$. 
One also can retrieve a residual %
representing the magnitude of the empirical gradient $({\widehat{\lambda}^\top\mathcal{A}(\mu)\widehat{\lambda}})^{\frac12}$. See Algorithm \ref{alg: analysis}.
Assuming
comparable point-cloud sizes ($m$), 
analysis costs $O(LNm\log m+LNmd+N^2md)$ per query, plus an
$N$-variable simplex-constrained quadratic program. The Gram matrix requires $O(Nmd+N^2)$ storage.

\paragraph{Analysis experiments.}
 For small pre-defined dictionaries, we test Algorithm~\ref{alg: analysis} in two settings: recovering weights of synthetically generated barycenters and representing non-synthetic data. (See Appendix \ref{sec: synt}.)

We first synthesize measures from dictionaries inspired by the examples in \cite{boneel}, using 3 templates and 12 weight vectors (Figure \ref{fig:syn-analysis}). For queries produced by our SW
fixed-point synthesis procedure and the similar SW synthesis scheme of \citet{boneel}, mean
$\ell^1$ coordinate errors through Algorithm \ref{alg: analysis} range from $6.22\times10^{-4}$ to
$3.49\times10^{-3}$, while we observe larger discrepancies for the other barycentric
geometries ($2.17\times10^{-2}$--$1.05\times10^{-1}$).
This comparison verifies coordinate-recovery within
the SW model, illustrating its dependence on the chosen
geometry. See Figure \ref{fig:syn-analysis} and Appendix \ref{app:f-b} for detailed results. %

We next use the ModelNet-C component of PointCloud-C
\citep{ren2022benchmark}, derived from ModelNet40 \citep{wu20153d},
and the five categories used by \citet{mallery2025synthesis}:
airplane, bed, guitar, monitor, and vase. Each of ten randomized trials
uses 3 or 5 clean references per class, $L=32$ directions, and 20 test
objects per class, disjoint from the references. In particular, we ask whether
the estimated coordinates contain class information, can serve for classification, and whether they remain stable under point-cloud corruptions.
We use classification to assess the class
information carried by the coordinates. 
For comparison, we also utilize the Point Cloud MNIST dataset
\cite{garcia2021pointcloudmnist} (10 classes).

For a query point cloud $X$, summing within a class $c$ gives its
\emph{class mass},
$\widehat p_c(X)=\sum_{j\in\mathcal I_c}\widehat\lambda_j(X)$,
where $\mathcal I_c$ indexes that class's references.
We predict the class with largest mass and reject queries with the
largest class-conditional residuals 
$r_{\mathrm{class}}(X)
=\min_c\min_{\lambda\in\Delta_{N_c-1}}
(\lambda^\top{\mathcal A}_c(X)\lambda)^{\frac{1}{2}}$,
where ${\mathcal A}_c(X)$ is the Gram block for the $N_c$
templates of class $c$.
Across both datasets,
Figure~\ref{fig:pointcloud-analysis-diagnostics} shows that
classification accuracy on the retained queries improves as
high-residual queries are removed. Encoding and diagnostics average about $50$\,ms per query (detailed timings in Appendix \ref{app:classif}).
For comparison, we evaluate nearest-reference classification
using SW distance (SW 1-NN), with the same templates and retained queries.
The results suggest that the residual helps identify real-data
subsets for which SW BCM coordinates provide reliable class information,
even for data not generated by barycentric synthesis.
Appendix~\ref{app:stab} provides
an additional experiment on the stability of SW barycentric
coordinates under point-cloud corruptions.

\section{Conclusions \& Discussion}
Our main contribution is a fast and interpretable algorithm for the
SW barycentric analysis problem, complementing existing work on
SW barycenter synthesis. The first-variation formula leads to
coordinate estimation through 1D transport computations and a convex
quadratic program. The resulting simplex-constrained weights make
each reference's estimated contribution explicit, providing a practical
way to encode measures relative to a dictionary. The accompanying
residual measures stationarity mismatch and provides a diagnostic
of the resulting representation. Experiments show accurate coordinate
recovery on synthesized SW queries and useful representations of
non-synthetic point clouds at low computational cost. For the latter,
class-conditional residuals help flag potentially unreliable predictions.

\begin{algorithm}[ht!]
\caption{Analysis - SW Barycentric Coordinates}
\label{alg: analysis}
\small
\begin{algorithmic}[1]
\Require Query point cloud $X=(X_k)_{k=1}^m$, template point clouds
$\{Y^{(j)}\}_{j=1}^N$, and number of directions $L$.
\Statex \State Sample $\Theta=\{\theta_q\}_{q=1}^L\subset\sphere$.
\For{$j=1,\ldots,N$}
    \For{$k=1,\ldots,m$}
        \State $U_{j,k}\gets 0\in\R^d$
    \EndFor
    \For{$q=1,\ldots,L$}
        \State $a_k\gets X_k\cdot\theta_q$ for $k=1,\ldots,m$
        \State $b_r\gets Y_r^{(j)}\cdot\theta_q$ for
        $r=1,\ldots,m_j$
        \State $(t_k)_{k=1}^m\gets
        \operatorname{MQM}((a_k)_{k=1}^m,(b_r)_{r=1}^{m_j})$
        \For{$k=1,\ldots,m$}
            \State $U_{j,k}\gets U_{j,k}+L^{-1}(a_k-t_k)\theta_q$
        \EndFor
    \EndFor
\EndFor
\For{$i=1,\ldots,N$}
    \For{$j=i,\ldots,N$}
        \State $\widehat{\mathcal A}_{ij}\gets
        m^{-1}\sum_{k=1}^m
        \langle U_{i,k},U_{j,k}\rangle$
        \State $\widehat{\mathcal A}_{ji}\gets
        \widehat{\mathcal A}_{ij}$
    \EndFor
\EndFor
\State $\widehat\lambda\in
\displaystyle\operatorname*{arg\,min}_{\lambda\in\Delta_{N-1}}
\lambda^\top\widehat{\mathcal A} \, \lambda$
\State \Return $\left(\widehat\lambda,\,
\sqrt{\widehat\lambda^\top\widehat{\mathcal A}\, \widehat\lambda}\right) \qquad \text{ (coordinates, residual)}$
\end{algorithmic}
\end{algorithm}

\begin{figure}[ht!]
    \centering
    \includegraphics[width=1\linewidth]{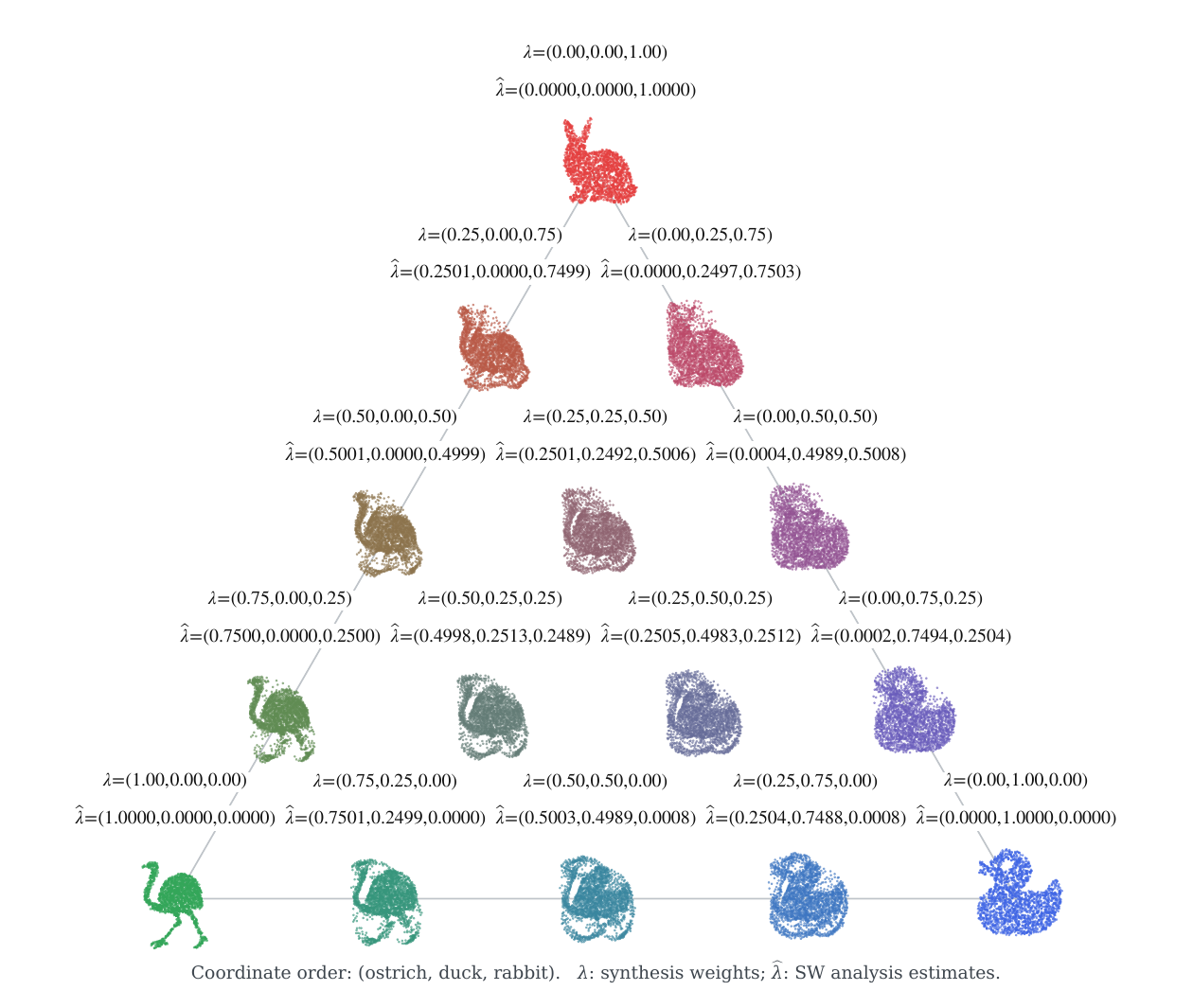}
    \caption{{\small SW BCM: vertices represent reference measures, while other grid points show SW barycenters synthesized (via fixed-point scheme) with weights $\lambda$ and coordinates $\widehat{\lambda}$ recovered by analysis (Algorithm \ref{alg: analysis}).}    
    }
    \label{fig:syn-analysis}
\end{figure}

\begin{figure}[ht!]
  \centering
  \includegraphics[width=\linewidth]{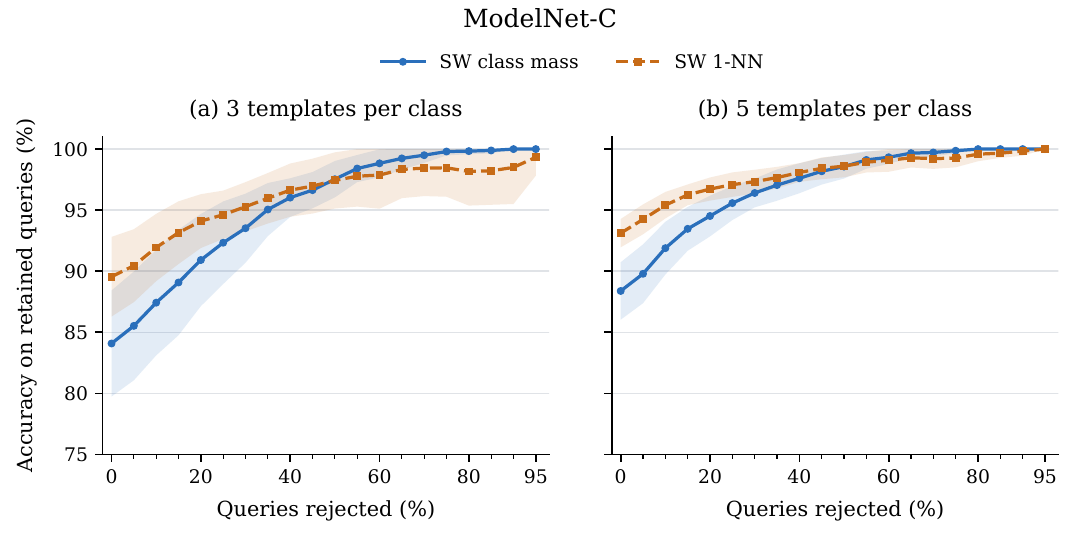}
  \includegraphics[width=\linewidth]{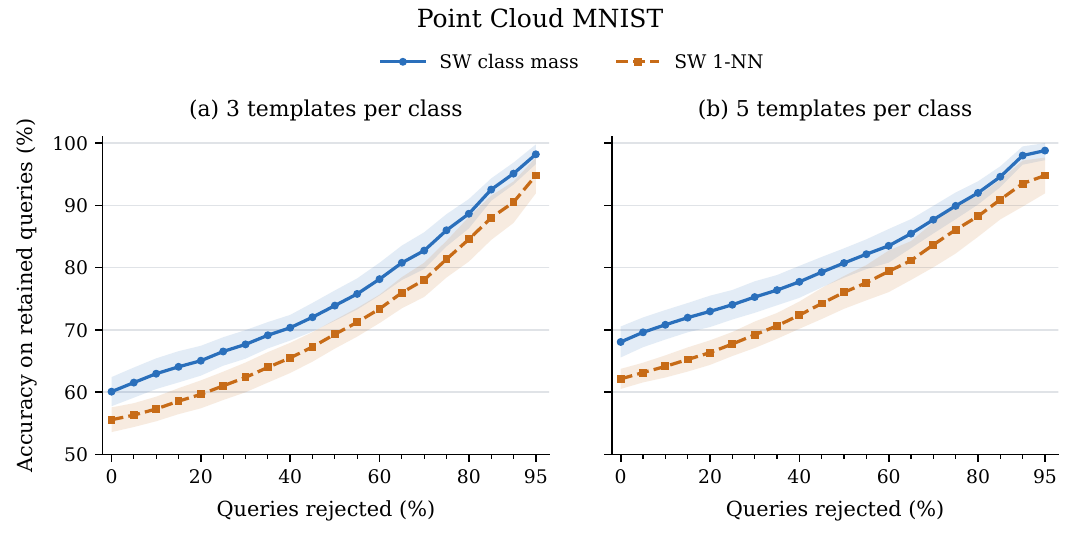}
  \caption{{\small Selective classification on ModelNet-C (top) and Point Cloud
MNIST (bottom), using nested dictionaries with three (a) or five (b)
references per class. Classifiers retain the
same queries after rejecting those with largest $r_{\mathrm{class}}$.
Curves show mean retained accuracy; shading indicates pointwise
$95\%$ Student-$t$ confidence intervals over ten ModelNet-C trials
or ten MNIST dictionary draws. More specifications in Appendix \ref{app:classif}.}}
  \label{fig:pointcloud-analysis-diagnostics}
\end{figure}

\section*{AI Use Statement}
The authors used ChatGPT to improve the clarity of the exposition, assist with literature searches, and discuss and refine mathematical arguments and proofs. In particular, substantial AI assistance was used in developing Example 2. The algorithms were implemented with significant 
assistance of Codex. The authors independently checked the mathematical statements and proofs, verified the cited sources, and reviewed all AI-assisted text and reviewed and tested the AI-assisted code.  The authors take full responsibility for the final content of this work.

\section*{Acknowledgments}
RDM was supported in part by NSF DMS-2603773; and 
JMM was supported in part by
DMS-2318894 and DMS-2309519.

\clearpage
\newpage
\section*{Appendices}
These appendices contain the complete proofs of all theoretical results stated in the main paper, together with the detailed calculation for Example~\ref{example: non Gaussian crit point}.

\section{Proofs of Main Results}\label{app:proofs}

\subsection{Critical Points of the SW Barycenter Functional in Wasserstein Space}

As a technicality, in \eqref{eq: T_theta} we choose the measurable quantile representative, defined $\mu^\theta$-a.e. by
\begin{equation}\label{eq: 1d OT formula}
    T_{j,\theta}^\mu(s)=Q_{\nu_j^\theta}\bigl(F_{\mu^\theta}(s)\bigr),
\end{equation}
where $F_{\mu^\theta}$ is the cumulative distribution function $F_{\mu^\theta}(s):=\mu^\theta((-\infty,s])$, $s\in \mathbb R$, and $Q_{\nu_j^\theta}$ is the quantile function or inverse cumulative distribution function $Q_{\nu_j^\theta}(q):= \inf\{x\in \mathbb R: \,  F_{\nu_j^\theta}(x)\geq q\}$ for $q\in [0,1]$. On the $\mu^\theta$-null set where the endpoint extended values $Q_{\nu_j^\theta}(0)$ or $Q_{\nu_j^\theta}(1)$ may be infinite, we assign the value $0$. This gives a finite jointly measurable representative $(\theta,s)\mapsto T_{j,\theta}^\mu(s)$ without changing the transport map $\mu^\theta$-a.e.

\begin{proof}[Proof of Proposition~\ref{thm: gradient}]
Since the first-variation is linear, it is enough to fix $j\in [N]$, set
\begin{equation*}
    \mathcal F_j(\mu):=\tfrac{1}{2}\,\SW(\mu,\nu_j)^2,
\end{equation*}
compute the Wasserstein gradient of $\mathcal F_j$, and then sum with the coefficients $\lambda_j$. By the definition of the SW distance,
\begin{equation*}
    \mathcal F_j(\mu)=\int_{\sphere} \mathcal J_{j,\theta}(\mu)\,d\sigma(\theta),
    \qquad
    \mathcal{J}_{j,\theta}(\mu):=\tfrac{1}{2}\,\W\bigl(\nu_j^\theta,\mu^\theta\bigr)^2.
\end{equation*}
The proof is divided into three claims: First we project velocities, then we compute the one-dimensional gradient, and finally we justify differentiating under the spherical integral.

\begin{claim}[Projection of velocities]\label{claim:v-theta}
Let $(\mu_t)_{|t|<\varepsilon}\subset \Ptwo(\R^d)$ be an absolutely continuous curve with velocity field $v_t$ which satisfies the continuity equation \eqref{eq: cont eq in Rd} in $\mathbb R^d$
in the sense of distributions. For each $\theta\in \sphere$, define the projected curve $(\mu_t^\theta)_{|t|<\varepsilon}\subset \Ptwo(\R)$ by
$\mu_t^\theta := (P_\theta)_\# \mu_t$. Then $(\mu_t^\theta)_{|t|<\varepsilon}$ satisfies the continuity equation in 1D
\begin{equation}\label{eq: cont eq in R}
    \begin{cases}
        \partial_t \mu_t^\theta + \partial_s(\mu_t^\theta v_t^\theta)=0,\\
        \mu_0^\theta=\mu^\theta,
    \end{cases}
\end{equation}
for the scalar velocity field $v_t^\theta\in  L^2(\R, d\mu_t^\theta)$ characterized by
\begin{equation}\label{eq: v theta}
\int_{\R^d} \psi(x\cdot\theta)\,(v_t(x)\cdot\theta)\, d\mu_t(x)
=
\int_{\R} \psi(s)\, v_t^\theta(s)\, d\mu_t^\theta(s)    
\end{equation}
for every bounded measurable $\psi:\R\to\R$, or equivalently, given by the conditional expectation
$
    v_t^\theta(s)
    :=
    \mathbb{E}\bigl[v_t(X)\cdot\theta \,\big|\, X\cdot\theta=s\bigr]$ with
     $X\sim \mu_t.
$
In particular, at $t=0$, the differential of the projection map $\Phi_\theta(\mu):=\mu^\theta$ is given by $\left[D\Phi_\theta(\mu)\right](v_0)=v_0^\theta$.
Moreover, if $g\in L^2(\mu^\theta)$, then the adjoint satisfies
\begin{align}
\int_{\R^d}\left\langle  \left(\left[D\Phi_\theta(\mu)\right]^*g\right)(x),v_0(x)\right\rangle\,  d\mu(x)&=\int_{\R} g(s) \, v_0^\theta(s) \, d\mu^\theta(s)\notag\\
&=\int_{\R^d} g(x\cdot\theta) \, v_0(x)\cdot\theta \, d\mu(x)\notag \\
&=\int_{\R^d} \left\langle g(x\cdot \theta)\, \theta,v_0(x)\right\rangle \, d\mu(x). \label{eq: adj}
\end{align}
\end{claim}

\begin{claim}[The one-dimensional gradient]\label{claim:1d-grad}
For $\sigma$-a.e. $\theta$ such that $\mu^\theta$ is non-atomic, the 1D Wasserstein gradient at $\mu^\theta$ of
$
    G_{j,\theta}(\rho):=\tfrac{1}{2}\,\W(\nu_j^\theta,\rho)^2
$
is given by
$
\nabla_{W_2} G_{j,\theta}(\mu^\theta)(s)=s-T_{j,\theta}^\mu(s)$,  $\mu^\theta\text{-a.e. }s\in \R.
$
Moreover, if the curve $(\mu_t^\theta)$ has velocity $v_t^\theta$, then
$
\left.\frac{d}{dt}\right|_{t=0} G_{j,\theta}(\mu_t^\theta)
=
\int_{\R}
\bigl(s-T_{j,\theta}^\mu(s)\bigr)\, v_0^\theta(s)\, d\mu^\theta(s).
$
\end{claim}

\begin{claim}[Differentiation under the integration on the sphere]\label{claim:deriv-int}
One can differentiate under the integral sign in the definition of $\mathcal F_j$: for absolutely continuous curves $\mu_t$ this holds for a.e. time $t$, and for the perturbations $\mu_t=(\Id+t v)_\#\mu$ with $v\in C_c^1(\R^d;\R^d)$ it holds at $t=0$.
\end{claim}

We now combine the claims. Fix $\theta\in\sphere$ and write the functional $\mathcal J_{j,\theta}$ as the composition
\begin{align*}
    \mathcal{P}_2(\R^d)&\longrightarrow\mathcal{P}_2(\R)\longrightarrow \R\\
    \quad \mu&\longmapsto \,\mu^\theta\,\longmapsto \frac{1}{2}\,\W(\nu_j^\theta,\mu^\theta)^2.
\end{align*}
Formally by the chain rule,
\begin{equation}\label{eq: chain rule}
    \nabla_{W_2}\mathcal{J}_{j,\theta}(\mu)=\left[D\Phi_\theta(\mu)\right]^*\left[\nabla_{W_2}G_{j,\theta}(\mu^\theta)\right].
\end{equation}
This formal expression is justified by Claims \ref{claim:v-theta} and \ref{claim:1d-grad}. Indeed, using the test function $g(s):=\nabla_{W_2} G_{j,\theta}(\mu^\theta)(s)=s-T_{j,\theta}^\mu(s)$
in \eqref{eq: adj}, we obtain
\begin{align*}
    \frac{d}{dt}\Big|_{t=0} \mathcal{J}_{j,\theta}(\mu_t)
    &=
    \int_{\R} \bigl(s-T_{j,\theta}^\mu(s)\bigr)\,v_0^\theta(s)\,d\mu^\theta(s)=\int_{\R^d} \left\langle\left(x\cdot \theta -T_{j,\theta}^\mu(x\cdot \theta)\right) \theta, v_0(x)\right\rangle \, d\mu(x).
\end{align*}
Essentially, the above calculation can be viewed as pulling the 1D derivative from Claim \ref{claim:1d-grad} back to $\R^d$, getting
\begin{equation}\label{eq: grad J}
    \nabla_{W_2}\mathcal{J}_{j,\theta}(\mu)(x)=\left(x\cdot \theta -T_{j,\theta}^\mu(x\cdot \theta)\right)\theta.
\end{equation}
Before integrating in $\theta$, let us justify that the vector field \eqref{eq: grad J} is well defined in
$L^2(\mu; \R^d)$. For fixed $j$, define
$$
F_j(x,\theta):=\bigl(x\cdot\theta-T_{j,\theta}^\mu(x\cdot\theta)\bigr)\theta,
\qquad (x,\theta)\in \R^d\times \sphere.
$$
By the quantile representative chosen in the standing assumption, $(x,\theta)\mapsto T_{j,\theta}^\mu(x\cdot\theta)$ is measurable. Hence $(x,\theta)\mapsto F_j(x,\theta)$ is measurable, and $\|F_j(x,\theta)\|^2
=
|x\cdot\theta-T_{j,\theta}^\mu(x\cdot\theta)|^2$.
Hence, by Tonelli's theorem,
\begin{align*}
\int_{\R^d}\int_{\sphere}\|F_j(x,\theta)\|^2\,d\sigma(\theta)\,d\mu(x)
&=
\int_{\sphere}\int_{\R^d}|x\cdot\theta-T_{j,\theta}^\mu(x\cdot\theta)|^2\,d\mu(x)\,d\sigma(\theta)\\
&=
\int_{\sphere} W_2(\mu^\theta,\nu_j^\theta)^2\,d\sigma(\theta)\\
&=
SW_2(\mu,\nu_j)^2<\infty.
\end{align*}
Therefore $F_j:\R^d\times \sphere\to\R^d$ is square integrable with respect to the product measure $\mu\otimes\sigma$. In particular, for $\mu$-a.e. $x$, the map
$\theta\mapsto F_j(x,\theta)$ is Bochner integrable in $\R^d$, and
$$
U_j^\mu(x)=\int_{\sphere}F_j(x,\theta)\,d\sigma(\theta)
$$
defines an element of $L^2(\mu;\R^d)$, because Jensen's inequality gives
$$
\int_{\R^d}|U_j^\mu(x)|^2\,d\mu(x)
\leq
\int_{\R^d}\int_{\sphere}\|F_j(x,\theta)\|^2\,d\sigma(\theta)\,d\mu(x)<\infty.
$$

Finally, using Claim \ref{claim:deriv-int} and Fubini's theorem,
\begin{align*}
\frac{d}{dt}\Big|_{t=0} \mathcal{F}_j(\mu_t)
&=
\int_{\sphere}
\int_{\R^d}
\left\langle
\bigl(x\cdot\theta - T_{j,\theta}^\mu(x\cdot\theta)\bigr)\theta,\,
 v_0(x)
\right\rangle
\, d\mu(x)\, d\sigma(\theta)\\
&=
\int_{\R^d}
\left\langle
\int_{\sphere}
\bigl(x\cdot\theta - T_{j,\theta}^\mu(x\cdot\theta)\bigr)\theta
\, d\sigma(\theta),\,
 v_0(x)
\right\rangle
\, d\mu(x).
\end{align*}
Therefore, by the definition of Wasserstein gradient on $\Ptwo(\R^d)$,
$
\nabla_{W_2} \mathcal{F}_j(\mu)(x)
=
\int_{\sphere}
\bigl(x\cdot\theta - T_{j,\theta}^\mu(x\cdot\theta)\bigr)\theta
\, d\sigma(\theta).
$
Equivalently, since $\sigma$ is normalized and uniform,
\begin{equation*}
    \nabla_{W_2} \mathcal{F}_j(\mu)(x)=\frac{1}{d}\, x-\int_{\sphere}
\bigl(T_{j,\theta}^\mu\circ P_\theta (x)\bigr)\theta
\, d\sigma(\theta),
\end{equation*}
since we have the identity \eqref{eq: Id/d}.
The weighted sum over $j$ proves the formula \eqref{eq: grad of sw bary}, that is, since $\mathcal{F}^\lambda(\mu)=\sum_{j=1}^N  \lambda_j \mathcal{F}_j(\mu)$,
we conclude by linearity that
$$
\nabla_{W_2} \mathcal{F}^\lambda(\mu)(x)
=
\sum_{j=1}^N  \lambda_j
\int_{\sphere}
\bigl(x\cdot\theta - T_{j,\theta}^\mu(x\cdot\theta)\bigr)\theta
\, d\sigma(\theta).
$$
Equivalently, since $\sigma$ is normalized and $\sum_j\lambda_j=1$, $\nabla_{W_2} \mathcal{F}^\lambda(\mu)(x)=\frac{1}{d} \, x-\sum_{j=1}^N  \lambda_j
\int_{\sphere}
\bigl(T_{j,\theta}^\mu\circ P_\theta(x)\bigr)\theta
\, d\sigma(\theta).$

The a.e.-in-time formula for absolutely continuous curves follows by applying the same computation at every time where the Wasserstein chain rule holds and where the projected measures are non-atomic.
\end{proof}

\begin{proof}[Proof of Claim~\ref{claim:v-theta}]
We want to show that the projected curve $(\mu_t^\theta)$, with velocity $v_t^\theta$ defined by \eqref{eq: v theta}, satisfies \eqref{eq: cont eq in R} in the sense of distributions. Let $\psi\in C_c^\infty(\R)$ be a test function. Then
\begin{align*}
    \frac{d}{dt}\int_\R \psi(s) \, d\mu_t^\theta(s)
    &=\frac{d}{dt}\int_{\R^d} \psi(x\cdot\theta) \, d\mu_t(x)\\
    &=\int_{\R^d}\psi'(x\cdot\theta)\, \theta\cdot v_t(x) \, d\mu_t(x) \qquad (\text{by } \eqref{eq: cont eq in Rd})\\
    &=\int_{\R}\psi'(s)\, v_t^\theta(s) \, d\mu_t^\theta(s) \qquad (\text{by } \eqref{eq: v theta}).
\end{align*}
This is exactly the weak form of the 1D continuity equation. The adjoint identity follows directly from the defining property \eqref{eq: v theta}.
\end{proof}

\begin{proof}[Proof of Claim~\ref{claim:1d-grad}]
Since $\mu^\theta$ is non-atomic, the one-dimensional monotone rearrangement gives a unique monotone optimal transport map $T_{j,\theta}^\mu$ from $\mu^\theta$ to $\nu_j^\theta$. The derivative formula for $\frac{1}{2} W_2^2$ along absolutely continuous Wasserstein curves gives
$$
\frac{d}{dt}\Big|_{t=0}\frac{1}{2} W_2(\mu_t^\theta,\nu_j^\theta)^2
=
\int_{\R}\bigl(s-T_{j,\theta}^\mu(s)\bigr)v_0^\theta(s)\,d\mu^\theta(s),
$$
because the optimal plan from $\mu^\theta$ to $\nu_j^\theta$ is induced by the map $T_{j,\theta}^\mu$. Thus the 1D Wasserstein gradient of $G_{j,\theta}$ at $\mu^\theta$ is the scalar vector field $\mathrm{Id}-T_{j,\theta}^\mu:\R\to\R$.
\end{proof}

\begin{proof}[Proof of Claim~\ref{claim:deriv-int}]
Consider
$$
\mathcal{F}_j(\mu_t)
=
\int_{\sphere} \frac{1}{2}\, \W(\nu_j^\theta,\mu_t^\theta)^2\, d\sigma(\theta), \qquad 
f_t(\theta):=\frac{1}{2}\, \W(\nu_j^\theta,\mu_t^\theta)^2.
$$
For $\sigma$-a.e. $\theta$, the map $t\mapsto f_t(\theta)$ is absolutely continuous along any absolutely continuous Wasserstein curve, and for a.e. $t$ its derivative is provided by Claim \ref{claim:1d-grad}. By Cauchy--Schwarz,
\begin{align*}
|\partial_t f_t(\theta)|
&=
\left|
\int_{\R}
\bigl(s-T^t_{j,\theta}(s)\bigr)\, v_t^\theta(s)\, d\mu_t^\theta(s)
\right| \le
\left(
\int_{\R} |s-T^{\mu_t}_{j,\theta}(s)|^2\, d\mu_t^\theta(s)
\right)^{1/2}
\left(
\int_{\R} |v_t^\theta(s)|^2\, d\mu_t^\theta(s)
\right)^{1/2}.
\end{align*}
For the first factor, 
$$
\left(
\int_{\R} |s-T^{\mu_t}_{j,\theta}(s)|^2\, d\mu_t^\theta(s)
\right)^{1/2}
=
\W(\mu_t^\theta,\nu_j^\theta)\le \W(\mu_t,\nu_j),
$$
since $P_\theta$ is $1$-Lipschitz. For the second factor, Jensen's inequality for conditional expectation gives
\begin{align*}
\int_{\R} |v_t^\theta(s)|^2\, d\mu_t^\theta(s)
&=
\int_{\R}
\left|
\mathbb{E}\bigl[v_t(X)\cdot\theta \,\big|\, X\cdot\theta=s\bigr]
\right|^2
\, d\mu_t^\theta(s) \\
&\le
\int_{\R^d} |v_t(x)\cdot\theta|^2\, d\mu_t(x)
\le
\int_{\R^d} |v_t(x)|^2\, d\mu_t(x).
\end{align*}
Therefore, $|\partial_t f_t(\theta)|
\le
\W(\mu_t,\nu_j)\,\|v_t\|_{L^2(\mu_t)}$. 
On every compact time interval, $t\mapsto \W(\mu_t,\nu_j)$ is bounded, and $t\mapsto \|v_t\|_{L^2(\mu_t)}$ is integrable for an absolutely continuous Wasserstein curve. The RHS is independent of $\theta$, and it is integrable in time and on $\sphere$. Hence Fubini's theorem and the dominated convergence theorem justify differentiating under the integral sign, giving for a.e. $t$,
$$
\frac{d}{dt}\mathcal{F}_j(\mu_t)
=
\int_{\sphere}
\frac{d}{dt}\left[\frac{1}{2}\W(\nu_j^\theta,\mu_t^\theta)^2\right]
\, d\sigma(\theta).
$$
For the special perturbation $\mu_t=(\Id+t v)_\#\mu$ with $v\in C_c^1(\R^d;\R^d)$, the same estimate holds in a small interval around $0$, so the same argument gives the derivative at $t=0$.
\end{proof}

\begin{proof}[Proof of Corollary~\ref{cor:direct-necessary-stationarity}]
For any $v\in C_c^1(\mathbb R^d;\mathbb R^d)$, define the perturbation
$
\mu_t=(\operatorname{Id}+tv)_\#\bar\mu$.
Then
$W_2(\mu_t,\bar\mu)
\leq
|t|\|v\|_{L^2(\bar\mu;\mathbb R^d)}$,
so $\mu_t$ remains in any prescribed $W_2$-neighborhood of
$\bar\mu$ for $|t|$ small enough. Since $\bar\mu$ is a local minimizer,
the map $t\mapsto \mathcal F^\lambda(\mu_t)$
has a local minimum at $t=0$. Hence its derivative at $t=0$ is $0$.
By Proposition~\ref{thm: gradient},
$$
0
=
\frac{d}{dt}\Big|_{t=0}\mathcal F^\lambda(\mu_t)
=
\int_{\mathbb R^d}
\left\langle \sum_{j=1}^N\lambda_jU_j^{\bar\mu}(x),v(x)\right\rangle\,d\bar\mu(x).
$$
Since this holds for every $v\in C_c^1(\mathbb R^d;\mathbb R^d)$ and
$C_c^1(\mathbb R^d;\mathbb R^d)$ is dense in
$L^2(\bar\mu;\mathbb R^d)$, while
$\sum_{j=1}^N\lambda_jU_j^{\bar\mu}\in L^2(\bar\mu;\mathbb R^d)$, we conclude that $\sum_{j=1}^N\lambda_jU_j^{\bar\mu}=0$ in $L^2(\bar\mu;\mathbb R^d)$.
\end{proof}

\subsection{The Analysis Problem}

\begin{proof}[Proof of Lemma~\ref{lem:norm-of-grad}]
By Proposition \ref{thm: gradient},
$
\nabla_{W_2}\mathcal{F}^\lambda(\mu)=\sum_{j=1}^N\lambda_j U_j^\mu$.
Therefore, by the Gram-matrix identity,
\begin{align*}
\|\nabla_{W_2}\mathcal{F}^\lambda(\mu)\|_{L^2(\mu;\mathbb R^d)}^2
&=\left\langle \sum_{i=1}^N\lambda_i U_i^\mu,\sum_{j=1}^N\lambda_j U_j^\mu\right\rangle_{L^2(\mu;\mathbb R^d)}=\sum_{i,j=1}^N\lambda_i\lambda_j\langle U_i^\mu,U_j^\mu\rangle_{L^2(\mu;\mathbb R^d)}
=\lambda^\top\mathcal A(\mu)\lambda,
\end{align*}
and $\mathcal{A}(\mu)$ is a positive semi-definite matrix in $\mathbb R^{N\times N}$.
\end{proof}

\begin{proof}[Proof of Proposition~\ref{prop:matrix-characterization}]
By using Lemma \ref{lem:norm-of-grad}, $\lambda^\top\mathcal A(\mu)\lambda=0$ if and only if
$\nabla_{W_2}\mathcal F^\lambda(\mu)=0$ in $L^2(\mu;\mathbb R^d)$. Since
$\Delta_{N-1} $ is compact and $\lambda\mapsto
\lambda^\top\mathcal A(\mu)\lambda$ is continuous, the minimum over the
simplex is attained. Therefore, the minimum is zero if and only if there
exists a simplex weight vector $\lambda^*$ for which
$\nabla_{W_2}\mathcal F^{\lambda^*}(\mu)=0$.
\end{proof}

\subsection{Fixed-Point Approach}

\begin{proof}[Proof of Proposition~\ref{prop:sliced-fixed-point-descent}]
First, we check that $B_\mu$ is well defined in $L^2(\mu;\R^d)$. 
Let 
$\overline T_\theta^\mu(x):=\sum_{j=1}^N\lambda_jT_{j,\theta}^\mu(x\cdot\theta)$.
By Jensen's
inequality,
\begin{align*}
\int_{\R^d}\int_{\sphere}
\left|\overline T_\theta^\mu(x)\right|^2\,d\sigma(\theta)\,d\mu(x)
&=
\int_{\R^d}\int_{\sphere}
\left|
\sum_{j=1}^N\lambda_jT_{j,\theta}^\mu(x\cdot\theta)
\right|^2
d\sigma(\theta)\,d\mu(x)\\
&\le
\sum_{j=1}^N\lambda_j
\int_{\sphere}\int_{\R^d}
\left|T_{j,\theta}^\mu(x\cdot\theta)\right|^2
\,d\mu(x)\,d\sigma(\theta)\\
&=
\sum_{j=1}^N\lambda_j
\int_{\sphere}\int_{\R}
|s|^2\,d\nu_j^\theta(s)\,d\sigma(\theta)\\
&=
\frac1d
\sum_{j=1}^N\lambda_j
\int_{\R^d}|y|^2\,d\nu_j(y)
<\infty .
\end{align*}
Therefore, again by Jensen's inequality,
$$
\int_{\R^d}|B_\mu(x)|^2\,d\mu(x)
\le
d^2
\int_{\R^d}\int_{\sphere}
\left|\overline T_\theta^\mu(x)\right|^2
\,d\sigma(\theta)\,d\mu(x)
<\infty .
$$

For fixed $x\in\R^d$ and $\theta\in\sphere$, 
the elementary variance identity gives
\begin{equation}
\label{eq:one-dimensional-variance}
\sum_{j=1}^N\lambda_j
|x\cdot\theta-T_{j,\theta}^\mu(x\cdot\theta)|^2
=
|x\cdot\theta-\overline T_\theta^\mu(x)|^2
+
\sum_{j=1}^N\lambda_j
|\overline T_\theta^\mu(x)-T_{j,\theta}^\mu(x\cdot\theta)|^2 .
\end{equation}
Since $T_{j,\theta}^\mu$ is the optimal map from $\mu^\theta$ to
$\nu_j^\theta$, integrating \eqref{eq:one-dimensional-variance} yields
\begin{align}
2\mathcal F^\lambda(\mu)
&=
\sum_{j=1}^N\lambda_j
\int_{\sphere}\int_{\R^d}
|x\cdot\theta-T_{j,\theta}^\mu(x\cdot\theta)|^2
\,d\mu(x)\,d\sigma(\theta) \notag\\
&=
\int_{\sphere}\int_{\R^d}
|x\cdot\theta-\overline T_\theta^\mu(x)|^2
\,d\mu(x)\,d\sigma(\theta) 
+
\sum_{j=1}^N\lambda_j
\int_{\sphere}\int_{\R^d}
|\overline T_\theta^\mu(x)-T_{j,\theta}^\mu(x\cdot\theta)|^2
\,d\mu(x)\,d\sigma(\theta).
\label{eq:F-mu-decomp-1}
\end{align}

We now decompose the first term. Recall that, since $\sigma$ is the normalized uniform
measure on $\sphere$, we have \eqref{eq: Id/d}. Let us denote by $L^2(\sphere,\sigma)$ the space of scalar functions over the sphere that are square integrable with respect to $\sigma$. 
For $\mu$-a.e. $x$, set
$
f_x(\theta):=\overline T_\theta^\mu(x)$ and
$\mathcal L:=
\{\theta\mapsto z\cdot\theta:z\in\mathbb R^d\}
\subset L^2(\sphere,\sigma).
$
Then $\mathcal L$ is a finite-dimensional closed subspace, and the
function
$
\theta\mapsto B_\mu(x)\cdot\theta
$
is the $L^2(\sphere,\sigma)$-orthogonal projection of $f_x$ onto
$\mathcal L$. Indeed, for every $z\in\mathbb R^d$,
\begin{align*}
\int_{\sphere}
\left(
f_x(\theta)-B_\mu(x)\cdot\theta
\right)
(z\cdot\theta)\,d\sigma(\theta)
&=
z\cdot
\int_{\sphere}f_x(\theta)\theta\,d\sigma(\theta)-
z\cdot
\left(
\int_{\sphere}\theta\theta^\top\,d\sigma(\theta)
\right)B_\mu(x)\\
&=
z\cdot\frac{B_\mu(x)}{d}
-
z\cdot\frac{B_\mu(x)}{d}
=0.
\end{align*}
Since $\sphere\ni\theta\mapsto x\cdot\theta$ is also a linear function,
Pythagoras gives
\begin{align}
\int_{\sphere}
|x\cdot\theta-\overline T_\theta^\mu(x)|^2
\,d\sigma(\theta)
&=
\int_{\sphere}
|(x-B_\mu(x))\cdot\theta|^2
\,d\sigma(\theta) 
+
\int_{\sphere}
|B_\mu(x)\cdot\theta-\overline T_\theta^\mu(x)|^2
\,d\sigma(\theta) \notag\\
&=
\frac1d |x-B_\mu(x)|^2
+
\int_{\sphere}
|B_\mu(x)\cdot\theta-\overline T_\theta^\mu(x)|^2
\,d\sigma(\theta).
\label{eq:projection-pythagoras}
\end{align}
Substituting \eqref{eq:projection-pythagoras} into
\eqref{eq:F-mu-decomp-1}, we obtain
\begin{align}
2\mathcal F^\lambda(\mu)
&=
\frac1d
\int_{\R^d}|x-B_\mu(x)|^2\,d\mu(x) \notag\\
&\qquad
+
\int_{\sphere}\int_{\R^d}
|B_\mu(x)\cdot\theta-\overline T_\theta^\mu(x)|^2
\,d\mu(x)\,d\sigma(\theta) \notag\\
&\qquad
+
\sum_{j=1}^N\lambda_j
\int_{\sphere}\int_{\R^d}
|\overline T_\theta^\mu(x)-T_{j,\theta}^\mu(x\cdot\theta)|^2
\,d\mu(x)\,d\sigma(\theta).
\label{eq:F-mu-decomp-2}
\end{align}

It remains to compare the last two terms with $2\mathcal F^\lambda(\mu^+)$.
Fix $j\in[N]$ and $\theta\in\sphere$. Since $\mu^+=(B_\mu)_\#\mu$, the
random variable $B_\mu(x)\cdot\theta$, with $x$ sampled according to $\mu$ (we write $x\sim\mu$), has law
$(\mu^+)^\theta$. Also, since $T_{j,\theta}^\mu$ pushes $\mu^\theta$ forward
to $\nu_j^\theta$, the random variable
$T_{j,\theta}^\mu(x\cdot\theta)$ has law $\nu_j^\theta$. Hence,
$
\left(B_\mu(\cdot)\cdot\theta,\,
T_{j,\theta}^\mu((\cdot)\,\cdot\theta)\right)_\#\mu
$
is an admissible coupling between $(\mu^+)^\theta$ and $\nu_j^\theta$.
Therefore,
\begin{equation}
\label{eq:coupling-bound}
W_2\bigl((\mu^+)^\theta,\nu_j^\theta\bigr)^2
\le
\int_{\R^d}
|B_\mu(x)\cdot\theta-T_{j,\theta}^\mu(x\cdot\theta)|^2
\,d\mu(x).
\end{equation}
Multiplying \eqref{eq:coupling-bound} by $\lambda_j$, summing over $j$, and
integrating over $\sphere$, we get
\begin{align}
2\mathcal F^\lambda(\mu^+)
&=
\sum_{j=1}^N\lambda_j
\int_{\sphere}
W_2\bigl((\mu^+)^\theta,\nu_j^\theta\bigr)^2
\,d\sigma(\theta) \notag\\
&\le
\sum_{j=1}^N\lambda_j
\int_{\sphere}\int_{\R^d}
|B_\mu(x)\cdot\theta-T_{j,\theta}^\mu(x\cdot\theta)|^2
\,d\mu(x)\,d\sigma(\theta).
\label{eq:F-mu-plus-bound-1}
\end{align}
Using the variance identity \eqref{eq:one-dimensional-variance} again, now with $B_\mu(x)\cdot\theta$ in place of
$x\cdot\theta$, we have
$$
\sum_{j=1}^N\lambda_j
|B_\mu(x)\cdot\theta-T_{j,\theta}^\mu(x\cdot\theta)|^2
=
|B_\mu(x)\cdot\theta-\overline T_\theta^\mu(x)|^2
+
\sum_{j=1}^N\lambda_j
|\overline T_\theta^\mu(x)-T_{j,\theta}^\mu(x\cdot\theta)|^2 .
$$
Thus \eqref{eq:F-mu-plus-bound-1} implies
\begin{align}
2\mathcal F^\lambda(\mu^+)
&\le
\int_{\sphere}\int_{\R^d}
|B_\mu(x)\cdot\theta-\overline T_\theta^\mu(x)|^2
\,d\mu(x)\,d\sigma(\theta) \notag\\
&\quad+
\sum_{j=1}^N\lambda_j
\int_{\sphere}\int_{\R^d}
|\overline T_\theta^\mu(x)-T_{j,\theta}^\mu(x\cdot\theta)|^2
\,d\mu(x)\,d\sigma(\theta).
\label{eq:F-mu-plus-bound-2}
\end{align}
Combining \eqref{eq:F-mu-decomp-2} and \eqref{eq:F-mu-plus-bound-2} gives
the first inequality in \eqref{eq:sliced-descent-strong}. 
Since $(\Id,B_\mu)_\#\mu$ is a coupling of $\mu$ and $\mu^+$,
$
W_2(\mu,\mu^+)^2
\le
\int_{\R^d}|x-B_\mu(x)|^2\,d\mu(x),
$
and therefore we have the second inequality in
\eqref{eq:sliced-descent-strong}.

Finally, if $B_\mu(x)\neq x$ on a set of positive $\mu$-measure, then $\int_{\R^d}|x-B_\mu(x)|^2\,d\mu(x)>0$,
and the stronger estimate \eqref{eq:sliced-descent-strong} gives
$\mathcal F^\lambda(\mu^+)<\mathcal F^\lambda(\mu)$. Thus the decrease is strict unless $B_\mu(x)=x$ for $\mu$-a.e. $x$.
\end{proof}

\begin{remark}Assume that the iteration $\mu_{n+1}:=(B_{\mu_n})_\#\mu_n$
is well defined for every $n\geq0$. Then,
Proposition~\ref{prop:sliced-fixed-point-descent} implies that
$\mathcal F^\lambda(\mu_n)$ is nonincreasing. Moreover,
$
\sum_{n=0}^{\infty}
\int_{\mathbb R^d}
|B_{\mu_n}(x)-x|^2\,d\mu_n(x)
\leq
2d\,\mathcal F^\lambda(\mu_0),
$
and, in particular,
$$
\int_{\mathbb R^d}
|B_{\mu_n}(x)-x|^2\,d\mu_n(x)
\longrightarrow0.
$$
In addition, since
$
\nabla_{W_2}\mathcal F^\lambda(\mu_n)(x)
=
\tfrac1d\bigl(x-B_{\mu_n}(x)\bigr),
$
the preceding estimate is equivalently written as
$
\sum_{n=0}^{\infty}
\left\|
\nabla_{W_2}\mathcal F^\lambda(\mu_n)
\right\|_{L^2(\mu_n;\mathbb R^d)}^2
\leq
\frac{2}{d}\mathcal F^\lambda(\mu_0).
$
Consequently,
$$
\left\|
\nabla_{W_2}\mathcal F^\lambda(\mu_n)
\right\|_{L^2(\mu_n;\mathbb R^d)}
\longrightarrow0.
$$
Thus, the iteration is \emph{asymptotically stationary}.
Furthermore, since
$(\Id,B_{\mu_n})_\#\mu_n$ is a coupling of
$\mu_n$ and $\mu_{n+1}$,
$
W_2(\mu_n,\mu_{n+1})^2
\leq
\int_{\mathbb R^d}
|B_{\mu_n}(x)-x|^2\,d\mu_n(x).
$
Hence, $\sum_{n=0}^{\infty}W_2(\mu_n,\mu_{n+1})^2<\infty$ and
\begin{equation}\label{eq: asymp}
W_2(\mu_n,\mu_{n+1})\longrightarrow0.
\end{equation}
\end{remark}

\begin{remark}
The continuity and compactness strategy developed for the classical
Wasserstein fixed-point iteration in
\cite{alvarez2016fixed} admits a sliced analogue:
\begin{enumerate}
    \item On the one hand, the map $\mu\mapsto (B_\mu)_\#\mu$, defined for probability measures $\mu\in\mathcal P_2(\mathbb R^d)$ such that
$\mu^\theta$  is non-atomic for 
$\sigma$-a.e. $\theta$, is $W_2$-continuous.
For its proof in the present setting, the stability of the full-dimensional optimal
transport maps used in \cite{alvarez2016fixed} needs to be replaced by the
explicit representation of the one-dimensional monotone transport maps
in terms of cumulative distribution and quantile functions.
Indeed, convergence in $W_2$ implies convergence of the projected cumulative
distribution functions, which yields convergence of the corresponding
transported variables for $\sigma$-a.e. direction $\theta$. Start by assuming $W_2(\mu_n,\mu)\to0$ as $n\to \infty$. Then, since $P_\theta$ is $1$-Lipschitz,
$$
W_2(\mu_n^\theta,\mu^\theta)
\leq
W_2(\mu_n,\mu)
\longrightarrow0.
$$
Thus, for $\sigma$-a.e. direction $\theta$, the projected measures
$\mu_n^\theta$ converge weakly to $\mu^\theta$. Since
$\mu^\theta$ is non-atomic, its cumulative distribution function is
continuous, and Pólya's theorem yields uniform convergence of the
corresponding cumulative distribution functions.

More precisely, using the gluing lemma, one may choose
$\mathbb R^d$-valued random variables
$X,X_1,X_2,\ldots$ on a common probability space such that $X_n\sim\mu_n$,
$X\sim\mu$,
\begin{align*}
\mathbb E|X_n-X|^2
&=W_2(\mu_n,\mu)^2\longrightarrow0.
\end{align*}
The representation of the one-dimensional monotone transport maps in
terms of cumulative distribution and quantile functions then implies
that, for every $j\in[N]$ and for $\sigma$-a.e. $\theta$,
$$
T_{j,\theta}^{\mu_n}(X_n\cdot\theta)
\longrightarrow
T_{j,\theta}^{\mu}(X\cdot\theta)
\qquad\text{in probability}.
$$
Note that the transported random variables $T_{j,\theta}^{\mu_n}(X_n\cdot\theta)$, $T_{j,\theta}^{\mu}(X\cdot\theta)$ have the same law $\nu_j^\theta$. Their squared
differences are therefore uniformly integrable, and the preceding
convergence is upgraded to
$$
T_{j,\theta}^{\mu_n}(X_n\cdot\theta)
\longrightarrow
T_{j,\theta}^{\mu}(X\cdot\theta)
\qquad\text{in }L^2.
$$
Finally, integrating over the sphere and applying Jensen's inequality
gives
\begin{align*}
&W_2^2\bigl((B_{\mu_n})_\#\mu_n,(B_\mu)_\#\mu\bigr)\leq \mathbb E|B_{\mu_n}(X_n)-B_\mu(X)|^2
\longrightarrow0.
\end{align*}

\item If the iteration $(\mu_n)_{n\geq 0}$ given by \eqref{eq: ite} is well defined and each iterate is such that $\mu_n\in\mathcal P_2(\mathbb R^d)$ with non-atomic projection
$\mu_n^\theta$, 
$\sigma$-a.e. $\theta$, then the sequence of iterates has uniformly integrable
second moments and is therefore relatively compact in
$(\mathcal P_2(\mathbb R^d),W_2)$. 
The main observation is that, for $X_n\sim\mu_n$, one has
$T_{j,\theta}^{\mu_n}(X_n\cdot\theta)
\sim\nu_j^\theta$.
Thus, although the transport maps depend on $n$, the laws of the
transported projected variables do not. Using this fact together with
Jensen's inequality and a standard uniform-integrability argument based
on the finite second moments of
$\mu_0,\nu_1,\ldots,\nu_N$, one obtains uniform integrability of the
second moments of the iterates. Consequently, the sequence
$(\mu_n)_{n\geq0}$ is relatively compact in
$(\mathcal P_2(\mathbb R^d),W_2)$.
Combined with the asymptotic regularity \eqref{eq: asymp},
the continuity of the map $\mu\mapsto (B_\mu)_\#\mu$ implies that every
$W_2$-cluster point (with non-atomic $\theta$-projection for $\sigma$-a.e. $\theta$), is a fixed
point of $\mu\mapsto (B_\mu)_\#\mu$.
\end{enumerate}

\end{remark}

\subsection{The Gaussian Case}

\begin{proof}[Proof of Theorem~\ref{thm: gaussian-sliced-barycenters}]
Let $\mu\in\mathcal P_2(\mathbb R^d)$ be arbitrary, with mean and covariance denoted by $m_\mu$ and $C_\mu$, i.e.,
$$
    m_\mu:=\int_{\mathbb R^d}x\,d\mu(x),
    \qquad
    C_\mu:=\int_{\mathbb R^d}(x-m_\mu)(x-m_\mu)^\top\,d\mu(x).
$$
Let $\gamma_\mu:=\mathcal N(m_\mu,C_\mu)$
be the Gaussian measure with the same mean and covariance as $\mu$.

For each direction $\theta\in\mathbb S^{d-1}$, the projected measure
$(P_\theta)_\#\mu$ has mean $m_\mu\cdot\theta$ and variance $\theta^\top C_\mu\theta$.
Similarly,
$
    (P_\theta)_\#\nu_i
    =
    \mathcal N(m_i\cdot\theta,\theta^\top S_i\theta)$ for $i=1,\dots, N$, and
    $(P_\theta)_\#\gamma_\mu
    =
    \mathcal N(m_\mu\cdot\theta,\theta^\top C_\mu\theta).
$

Following the strategy in \cite{alvarez2016fixed}, we now use the one-dimensional Gelbrich inequality (see Lemma \ref{lem: Gelbrich}): If $\eta\in\mathcal P_2(\mathbb R)$
has mean $a$ and variance $\alpha^2$, and if
$\zeta=\mathcal N(b,\beta^2)$, then
$
    W_2^2(\eta,\zeta)
    \geq
    (a-b)^2+(\alpha-\beta)^2.
$
In our case,
\begin{equation}\label{eq: Gelbrich}
    W_2^2\bigl((P_\theta)_\#\mu,(P_\theta)_\#\nu_i\bigr)
    \geq
    \bigl((m_\mu-m_i)\cdot\theta\bigr)^2 +
    \left(
        \sqrt{\theta^\top C_\mu\theta}
        -
        \sqrt{\theta^\top S_i\theta}
    \right)^2.
\end{equation}
But the RHS in \eqref{eq: Gelbrich} is exactly the squared Wasserstein distance between the
one-dimensional Gaussian projections
    $(P_\theta)_\#\gamma_\mu
    =
    \mathcal N(m_\mu\cdot\theta,\theta^\top C_\mu\theta)$, 
    $(P_\theta)_\#\nu_i
    =
    \mathcal N(m_i\cdot\theta,\theta^\top S_i\theta)$.
Therefore,
$
    W_2^2\bigl((P_\theta)_\#\mu,(P_\theta)_\#\nu_i\bigr)
    \geq
    W_2^2\bigl((P_\theta)_\#\gamma_\mu,(P_\theta)_\#\nu_i\bigr)$ for every $i\in [N].
$
Integrating over $\theta\in\mathbb S^{d-1}$ gives
    $SW_2^2(\mu,\nu_i)
    \geq
    SW_2^2(\gamma_\mu,\nu_i)$ for every $i\in [N]$.
Multiplying by $\lambda_i$, summing over $i$, and dividing by $2$, we get
$    \mathcal F(\mu)
    \geq
    \mathcal F(\gamma_\mu)$.

In particular, let $\bar\mu$ be a minimizer of $\mathcal F^\lambda$. Applying the previous inequality
to $\mu=\bar\mu$ gives
\begin{equation}\label{eq: <=}
        \mathcal F^\lambda(\bar\mu)
    \geq
    \mathcal F^\lambda(\gamma_{\bar\mu}).
\end{equation}
Since $\bar\mu$ is a minimizer and $\gamma_{\bar\mu}$ is an admissible competitor,
we also have
\begin{equation}\label{eq: >=}
    \mathcal F^\lambda(\bar\mu)
    \leq
    \mathcal F^\lambda(\gamma_{\bar\mu}).
\end{equation}
Hence, \eqref{eq: <=} and \eqref{eq: >=} yield
$
    \mathcal F^\lambda(\bar\mu)
    =
    \mathcal F^\lambda(\gamma_{\bar\mu}).
$

Thus equality must hold in the one-dimensional Gelbrich inequality  \eqref{eq: Gelbrich} for
$\sigma$-almost every direction $\theta$ and for every $i$ with $\lambda_i>0$. In particular, equality holds for the summand with 
$\lambda_{i_0}>0$.
Because $S_{i_0}\in\mathbb S_{++}^d$, the projected target $(P_\theta)_\#\nu_{i_0}
=
\mathcal N(m_{i_0}\cdot\theta,\theta^\top S_{i_0}\theta)$
is nondegenerate for every $\theta\in\mathbb S^{d-1}$. Hence the equality
case in the one-dimensional Gelbrich inequality implies that
$(P_\theta)_\#\bar\mu$ is Gaussian for $\sigma$-a.e. $\theta$.
By continuity of characteristic functions, this a.e.-in-$\theta$ statement
extends to every direction. Therefore, by the Cramér--Wold theorem,
$\bar\mu$ is Gaussian.
Hence every
minimizer of $\mathcal F$ is Gaussian.

\end{proof}

\begin{lemma}\label{lem: another formula for F_G} The functional $\mathcal{F}_G^\lambda$ in \eqref{eq: SW obj for Gaussians}, i.e., $\mathcal{F}_G^\lambda(m,S)=\frac12\sum_{j=1}^N\lambda_j
SW_2\bigl(\mathcal N(m,S),\mathcal N(m_j,S_j)\bigr)^2$, can be re-written as
    \begin{align*}\mathcal F_G^\lambda&(m,S)
=\frac{1}{2d}\sum_{j=1}^N\lambda_j|m-m_j|^2\notag+\frac12\int_{\sphere}\sum_{j=1}^N\lambda_j
\left(\sqrt{\theta^\top S\theta}-\sqrt{\theta^\top S_j\theta}\right)^2
\,d\sigma(\theta).
\end{align*}

\end{lemma}

\begin{proof}[Proof of Lemma~\ref{lem: another formula for F_G}]
    For each $\theta\in\mathbb S^{d-1}$, let
$
(P_\theta)_\#\mathcal N(m,S)
=
\mathcal N(m\cdot\theta,\theta^\top S\theta)$ and $
(P_\theta)_\#\mathcal N(m_j,S_j)
=
\mathcal N(m_j\cdot\theta,\theta^\top S_j\theta)$ $j\in [N]$. 
In one dimension,
$
W_2^2\left(
\mathcal N(a,\alpha^2),
\mathcal N(b,\beta^2)
\right)
=
(a-b)^2+(\alpha-\beta)^2.
$
Therefore,
\begin{align*}
SW_2^2\bigl(\mathcal N(m,S),\mathcal N(m_j,S_j)\bigr)
&=
\int_{\mathbb S^{d-1}}
\left[
((m-m_j)\cdot\theta)^2
+
\left(
\sqrt{\theta^\top S\theta}
-
\sqrt{\theta^\top S_j\theta}
\right)^2
\right]
\,d\sigma(\theta)\\
&=
\frac1d |m-m_j|^2
+
\int_{\mathbb S^{d-1}}
\left(
\sqrt{\theta^\top S\theta}
-
\sqrt{\theta^\top S_j\theta}
\right)^2
\,d\sigma(\theta),
\end{align*}
where we have used \eqref{eq: Id/d}.
Summing with weights $\lambda_j$ gives the displayed formula for
$\mathcal F_G^\lambda$.
\end{proof}

\begin{lemma}\label{lem: gaussian differential}
Let $\nu_j=\mathcal N(m_j,S_j)$, $j\in[N]$
where $m_j\in\mathbb R^d$, $S_j\in\mathbb S_{++}^d$, and let
$\lambda\in\Delta_{N-1}$. 
Consider the SW barycenter functional $\mathcal{F}_G^\lambda$ defined in \eqref{eq: SW obj for Gaussians}
for Gaussian candidates $\mu=\mathcal N(m,S)$, $m\in\mathbb R^d$, $S\in\mathbb S_{++}^d$. Then, the covariance term in the second expression in \eqref{eq: SW obj for Gaussians}, i.e.,
$S\mapsto\frac12\int_{\sphere}\sum_{j=1}^N\lambda_j
\left(\sqrt{\theta^\top S\theta}-\sqrt{\theta^\top S_j\theta}\right)^2
d\sigma(\theta),$
is convex, and for every symmetric matrix $H$, 
the directional derivative of $\mathcal F_G^\lambda$ at $(m,S)$ with respect to the covariance variable $S$, in the direction $H$, is given by 
\begin{align}\label{eq: D_S}
D_S\mathcal F_G^\lambda(m,S)[H]
&=\tfrac12\operatorname{Tr}\left(\int_{\sphere}\left(1-\tfrac{\sum_{j=1}^N\lambda_j\sqrt{\theta^\top S_j\theta}}
{\sqrt{\theta^\top S\theta}}\right)
\theta\theta^\top d\sigma(\theta)H\right)
\end{align} 
\end{lemma}

\begin{proof}[Proof of Lemma~\ref{lem: gaussian differential}]
The covariance part in the second expression in \eqref{eq: SW obj for Gaussians} is the same functional as in the centered case (i.e., when all the Gaussian measures considered are centered at the origin of coordinates). Expanding
the square gives
$$
\frac12
\int_{\mathbb S^{d-1}}
\left[
\theta^\top S\theta
-
2
\left(\sum_{j=1}^N\lambda_j\sqrt{\theta^\top S_j\theta}\right)
\sqrt{\theta^\top S\theta}
+
\sum_{j=1}^N\lambda_j\theta^\top S_j\theta
\right]\,d\sigma(\theta).
$$
For fixed $\theta$, $S\mapsto \theta^\top S\theta$ is linear and
$S\mapsto -\sqrt{\theta^\top S\theta}$ is convex on
$\mathbb S_{++}^d$. Hence the covariance part is convex.

The directional derivative of $\mathcal F_G^\lambda$ at $(m,S)$ with respect to the covariance variable $S$, in the direction of a symmetric $d\times d$ matrix $H$, is defined by 
\begin{equation*}
D_S\mathcal F_G^\lambda(m,S)[H]
:=
\left.\frac{d}{dt}\right|_{t=0}
\mathcal F_G^\lambda(m,S+tH),
\end{equation*}
where $S+tH\in\mathbb S_{++}^d$ for all sufficiently small $|t|$ since $\mathbb S_{++}^d$ is open in the space of symmetric matrices.
For a fixed symmetric matrix $H$, set
$$
S_t:=S+tH,
\qquad
a(\theta):=
\sum_{j=1}^N\lambda_j\sqrt{\theta^\top S_j\theta}.
$$
For each fixed $\theta\in\mathbb S^{d-1}$,
$$
\left.\frac{d}{dt}\right|_{t=0}
\theta^\top S_t\theta
=
\theta^\top H\theta,
$$
and, by the usual derivative of the square-root function,
$$
\left.\frac{d}{dt}\right|_{t=0}
\sqrt{\theta^\top S_t\theta}
=
\frac{\theta^\top H\theta}
{2\sqrt{\theta^\top S\theta}}.
$$
Notice that $a(\theta)$ is independent of $S$. Moreover, since $S\in\mathbb S_{++}^d$, the quantity $\theta^\top S\theta$ is uniformly bounded away from zero on $\mathbb S^{d-1}$, which justifies differentiating under the integral sign. Therefore,
\begin{align*}
D_S\mathcal F_G^\lambda(m,S)[H]&
=
\frac12
\int_{\mathbb S^{d-1}}
\left[
\theta^\top H\theta
-
2a(\theta)
\frac{\theta^\top H\theta}
{2\sqrt{\theta^\top S\theta}}
\right]
\,d\sigma(\theta) \\
&=
\frac12
\int_{\mathbb S^{d-1}}
\left[
1-
\frac{\sum_{j=1}^N\lambda_j\sqrt{\theta^\top S_j\theta}}
{\sqrt{\theta^\top S\theta}}
\right]
\theta^\top H\theta
\,d\sigma(\theta).
\end{align*}
The mean term in the second expression in \eqref{eq: SW obj for Gaussians} does not appear because it is independent of $S$ and $D_S$ keeps $m$ fixed, and hence its directional derivative with respect to $S$ is zero.
Finally, using $\theta^\top H\theta
=
\operatorname{Tr}(\theta\theta^\top H)$,
we obtain the identity \eqref{eq: D_S}.
\end{proof}

\begin{proof}[Proof of Proposition~\ref{prop: gaussian bary}]
The proof is a consequence of Lemma \ref{lem: another formula for F_G} and Lemma \ref{lem: gaussian differential}.

The mean part in the second expression of
\eqref{eq: SW obj for Gaussians} has the unique minimizer
$\overline m$. On the other hand, by \eqref{eq: Gaussian crit point}
and Lemma \ref{lem: gaussian differential}, $D_S\mathcal F_G^\lambda(\overline m,S)[H]
=
0$
for every symmetric matrix $H$.

Let $\widetilde S\in\mathbb S_{++}^d$. For fixed $\overline m$, the map
$
S'\mapsto\mathcal F_G^\lambda(\overline m,S')
$
is convex and differentiable on $\mathbb S_{++}^d$. Hence, the
first-order inequality for differentiable convex functions gives
\begin{align*}
\mathcal F_G^\lambda(\overline m,\widetilde S)
&\geq
\mathcal F_G^\lambda(\overline m,S)
+
D_S\mathcal F_G^\lambda(\overline m,S)
[\widetilde S-S] =
\mathcal F_G^\lambda(\overline m,S).
\end{align*}
Thus, $S$ minimizes the covariance part. Since the mean and covariance
variables decouple in \eqref{eq: SW obj for Gaussians}, it follows that
$
\mathcal F_G^\lambda(\overline m,S)
\leq
\mathcal F_G^\lambda(\widetilde m,\widetilde S)
$
for every $\widetilde m\in\mathbb R^d$ and every
$\widetilde S\in\mathbb S_{++}^d$. Hence
$\mathcal N(\overline m,S)$ minimizes
$\mathcal F_G^\lambda$ over nondegenerate Gaussian candidates.

Moreover, the expression in \eqref{eq: SW obj for Gaussians} extends
continuously to positive semi-definite covariance matrices. Indeed, if
$\widetilde S\in\mathbb S_+^d$, then
$\widetilde S+\varepsilon\Id\in\mathbb S_{++}^d$ for every
$\varepsilon>0$. Therefore, for every
$\widetilde m\in\mathbb R^d$, $\mathcal F_G^\lambda(\overline m,S)
\leq
\mathcal F_G^\lambda
\bigl(\widetilde m,\widetilde S+\varepsilon\Id\bigr)$. 
Letting $\varepsilon\downarrow0$, we conclude that the preceding
minimality also holds over possibly degenerate Gaussian candidates.

Thus, we conclude that $\mathcal N(\overline m,S)$ minimizes the functional 
 $\mathcal F^\lambda$ (given by \eqref{eq: SW bary functional} with Gaussian templates)
over 
Gaussian candidates.

Finally, 
it remains to pass from Gaussian candidates to the full space
$\mathcal P_2(\mathbb R^d)$.
By Theorem \ref{thm: gaussian-sliced-barycenters}, the
infimum of $\mathcal F^\lambda$ over $\mathcal P_2(\mathbb R^d)$
coincides with its infimum over Gaussian candidates. Since
$\mathcal N(\overline m,S)$ attains the latter, it also attains the
former. Hence $\mathcal N(\overline m,S)$ is a global
SW barycenter.

\end{proof}

\begin{proof}[Proof of Corollary \ref{cor:isotropic-gaussian-sliced-condition}]
Set $S_j=\sigma_j^2\Id$ and $S=\sigma^2\Id$. Then, for every
$\theta\in\mathbb S^{d-1}$,
$\sqrt{\theta^\top S_j\theta}=\sigma_j$,
$\sqrt{\theta^\top S\theta}=\sigma$.
Hence the Gaussian criticality condition \eqref{eq: Gaussian crit point} becomes
\[
\left(
1-\frac{\sum_{j=1}^N\lambda_j\sigma_j}{\sigma}
\right)
\int_{\mathbb S^{d-1}}
\theta\theta^\top\,d\sigma(\theta)=0.
\]
Using the identity \eqref{eq: Id/d},
we obtain $\sigma=\sum_{j=1}^N\lambda_j\sigma_j$.
The conclusion follows from
Proposition~\ref{prop: gaussian bary}.
\end{proof}

\paragraph{Detailed Non-Gaussian Critical-Point Construction.}
\label{sec:non-gaussian-details}
\begin{proof}[Details for Example~\ref{example: non Gaussian crit point}]
Let $d\geq 2$, let $\nu=\mathcal N(0,\Id)$
and consider
\begin{equation}
    \mathcal F(\mu):=\frac12 SW_2^2\left(\mu,\mathcal N(0,\Id)\right).
\end{equation}
We show that $\mathcal F$ admits a non-Gaussian critical point.

For $r>0$, let $\mu_r:=\operatorname{Unif}(r\mathbb S^{d-1})$
be the uniform probability measure on the sphere of radius $r$. This measure is
not Gaussian, since it is supported on $r\mathbb S^{d-1}$. 
Let $X\sim\operatorname{Unif}(\mathbb S^{d-1})$, set $Y=rX\sim\mu_r$, and let $Z=X\cdot e_1$ with cumulative distribution function $H$ (i.e., $H(t):=\mathbb P(Z\leq t)$). For $d\geq2$, the law of $Z$ is non-atomic, so every projection of $\mu_r$ satisfies the standing assumption of Proposition \ref{thm: gradient}. 
By rotational invariance of the uniform measure on $\mathbb S^{d-1}$, for every
$\theta\in\mathbb S^{d-1}$,
$$
    X\cdot\theta \stackrel{d}{=} X\cdot e_1=Z, \quad \text{and } \quad P_\theta(Y)=Y\cdot\theta=r(X\cdot\theta)\stackrel{d}{=}rZ. $$
Hence, for each $\theta\in \sphere$, the cumulative distribution function of $(P_\theta)_\#\mu_r$ is given by
$$
    s\mapsto\mathbb P(rZ\leq s)
    =
    \mathbb P\left(Z\leq \frac{s}{r}\right)
    =
    H\left(\frac{s}{r}\right).
$$
On the other hand, since $\nu=\mathcal N(0,\Id)$ is isotropic, $(P_\theta)_\#\nu=\mathcal N(0,1)$
for every $\theta\in\mathbb S^{d-1}$. Let $\Phi$ be the cumulative distribution  function of the standard Gaussian. 

As a consequence, for every $\theta\in\mathbb S^{d-1}$, the one dimensional monotone optimal
transport map from $(P_\theta)_\#\mu_r$ to $(P_\theta)_\#\nu$ is
$
    T_r(s):=
    \Phi^{-1}\left(H\left(\frac{s}{r}\right)\right),
$
where the expression is understood $\mu_r^\theta$-a.e.

Using the first-variation formula for the SW barycenter functional (Proposition \ref{thm: gradient}), we get
\begin{align*}
        \nabla_{W_2}\mathcal F(\mu_r)(x)
    &=
    \int_{\mathbb S^{d-1}}
    \left(
        x\cdot\theta
        -
        T_r(x\cdot\theta)
    \right)\theta
    \,d\sigma(\theta)=
    \int_{\mathbb S^{d-1}}
    \left(
        x\cdot\theta
        -
        \Phi^{-1}\left(H\left(\frac{x\cdot\theta}{r}\right)\right)
    \right)\theta
    \,d\sigma(\theta),
\end{align*}
which is defined in the support of $\mu_r$, that is, for $x\in r\sphere$.
Note that, for every orthogonal matrix $Q\in \mathrm{O}(d)$,
$$
    \nabla_{W_2}\mathcal F(\mu_r)(Qx)=Q\nabla_{W_2}\mathcal F(\mu_r)(x).
$$
Indeed, making the change of variables $\eta=Q^\top\theta$, and using the
invariance of the uniform surface measure under orthogonal
transformations, gives
\begin{align*}
\nabla_{W_2}\mathcal F(\mu_r)(Qx)&=\int_{\mathbb S^{d-1}}
g_r(x\cdot Q^\top\theta)\,\theta\,d\sigma(\theta)
=
\int_{\mathbb S^{d-1}}
g_r(x\cdot\eta)\,Q\eta\,d\sigma(\eta)\\
&=
Q\int_{\mathbb S^{d-1}}
g_r(x\cdot\eta)\,\eta\,d\sigma(\eta)=
Q\nabla_{W_2}\mathcal F(\mu_r)(x),
\end{align*}
where $g_r(t):=t-\Phi^{-1}\left(H\left(\frac{t}{r}\right)\right)$.
In particular, if
$Qe_1=e_1$, then $Q(re_1)=re_1$, and so
$$
    \nabla_{W_2}\mathcal F(\mu_r)(re_1)=\nabla_{W_2}\mathcal F(\mu_r)(Qre_1)=Q\nabla_{W_2}\mathcal F(\mu_r)(re_1).
$$
Thus $\nabla_{W_2}\mathcal F(\mu_r)(re_1)$ is fixed by every orthogonal transformation fixing $e_1$, which implies that
$\nabla_{W_2}\mathcal F(\mu_r)(re_1)$ must be parallel to $e_1$. Hence
$$
    \nabla_{W_2}\mathcal F(\mu_r)(re_1)=0
    \qquad\Longleftrightarrow\qquad
    \nabla_{W_2}\mathcal F(\mu_r)(re_1)\cdot e_1=0.
$$
We now compute this scalar radial component:
\begin{align*}
    \nabla_{W_2}\mathcal F(\mu_r)(re_1)\cdot e_1
    &=
    \int_{\mathbb S^{d-1}}
    \left(
        r\theta_1
        -
        \Phi^{-1}(H(\theta_1))
    \right)\theta_1
    \,d\sigma(\theta)=\frac{r}{d}
    -
    \underbrace{\int_{\mathbb S^{d-1}}
    \Phi^{-1}(H(\theta_1))\theta_1
    \,d\sigma(\theta)}_{A_d},
\end{align*}
where the constant $A_d$ is finite and strictly positive. By choosing $r_\ast:=dA_d$, we obtain
\begin{align*}
\nabla_{W_2}\mathcal F(\mu_{r_\ast})(r_\ast e_1)\cdot e_1&=0.
\end{align*}
and so, $\nabla_{W_2}\mathcal F(\mu_{r_\ast})(r_\ast e_1)=0$.
Finally, every point $x\in r_\ast\mathbb S^{d-1}$ can be written as $x=Q(r_\ast e_1)$
for some $Q\in \mathrm{O}(d)$. Therefore,
$$
    \nabla_{W_2}\mathcal F(\mu_{r_\ast})(x)
    =
    \nabla_{W_2}\mathcal F(\mu_{r_\ast})(Q r_\ast e_1)
    =
    Q\nabla_{W_2}\mathcal F(\mu_{r_\ast})(r_\ast e_1)
    =
    0.
$$
Equivalently,
$
    \nabla_{W_2}\mathcal F(\mu_{r_\ast})(x)=0$ for
    $
    \mu_{r_\ast}\text{-a.e. }x.
$
Thus, $\mu_{r_\ast}
    =
    \operatorname{Unif}(r_\ast\mathbb S^{d-1})$
is a non-Gaussian critical point of $\mathcal F(\mu):=\frac12 SW_2^2\left(\mu,\mathcal N(0,\Id)\right)$.

In this example, we have used that, for each fixed $r>0$, both $\mu_r$ and $\nu=\mathcal N(0,\Id)$ are
rotationally invariant.
Moreover, the isotropy assumption on $\nu$ is essential for this spherical argument. If instead
$\nu=\mathcal N(0,S)$ with a general positive definite covariance matrix $S$,
then $(P_\theta)_\#\nu
    =
    \mathcal N(0,\theta^\top S\theta)$,
and the projected monotone transport map becomes
$
    s\mapsto
    \sqrt{\theta^\top S\theta}\,
    \Phi^{-1}\left(H\left(\frac{s}{r}\right)\right),
$
which generally depends on $\theta$.

\end{proof}

\subsubsection{The Analysis Procedure in the Gaussian Case}\label{app:gauss-analysis}

We now specialize the analysis problem to the case of Gaussian measures. Consider templates $\nu_j=\mathcal N(m_j,S_j)$, $j=1,\ldots,N$, and let $\mu=\mathcal N(m,S)$, 
where $m,m_j\in\mathbb R^d$ and $S,S_j\in\mathbb S_{++}^d$. For each direction
$\theta\in\mathbb S^{d-1}$, the projected measures are one-dimensional Gaussians:
\begin{align*}
\nu_j^\theta
&=\mathcal N(m_j\cdot\theta,\theta^\top S_j\theta), \, j\in[N],
 \qquad 
\mu^\theta
=\mathcal N(m\cdot\theta,\theta^\top S\theta).
\end{align*}
Hence the monotone optimal transport map
$T_{j,\theta}^\mu=T_{\mu^\theta\to\nu_j^\theta}$ is affine and is given by
\begin{equation}\label{eq: gaussian-1d-map}
    T_{j,\theta}^\mu(s)
    =
    m_j\cdot\theta
    +
    \left(
    \frac{\theta^\top S_j\theta}
         {\theta^\top S\theta}
    \right)^{1/2}
    \bigl(s-m\cdot\theta\bigr).
\end{equation}
Define the matrix
\begin{equation}\label{eq: Bj-gaussian}
    B_j( S)
    :=
    \int_{\mathbb S^{d-1}}
    \left(
    \frac{\theta^\top S_j\theta}
         {\theta^\top S\theta}
    \right)^{1/2}
    \theta\theta^\top
    \,d\sigma(\theta).
\end{equation}
Then the sliced criticality condition \eqref{eq: c.p. condition as f.p.} ($\mu$-a.e. $x$),
reduces to the finite-dimensional system, 
\begin{align}\label{eq: gaussian-analysis-system}
m&=\sum_{j=1}^N\lambda_jm_j, \quad 
\sum_{j=1}^N\lambda_jB_j(S)=\frac1d\Id
\end{align}
for $\lambda\in\Delta_{N-1}$.
Indeed, using \eqref{eq: gaussian-1d-map}, we obtain
\begin{align*}
&\sum_{j=1}^N\lambda_j
\int_{\sphere}T_{j,\theta}^\mu(x\cdot\theta)\theta\,d\sigma(\theta)=\frac1d\sum_{j=1}^N\lambda_jm_j
+\left(\sum_{j=1}^N\lambda_jB_j(S)\right)(x-m).
\end{align*}
Then, 
the equality in \eqref{eq: c.p. condition as f.p.} holds for $\mu$-a.e. $x$ if and only if the constant and linear parts agree, which gives
\eqref{eq: gaussian-analysis-system}.

Therefore, in the exact Gaussian analysis problem, the barycentric coordinates can be recovered by solving the linear problem \eqref{eq: gaussian-analysis-system} for $\lambda\in\Delta_{N-1}$.

In practice, one may instead solve the constrained least-squares problem
\begin{align}\label{eq: gaussian-analysis-ls}
\widehat\lambda
&\in\arg\min_{\lambda\in\Delta_{N-1}}
\Bigg\{
\left\|m-\sum_{j=1}^N\lambda_jm_j\right\|^2
+\rho\left\|\frac1d\Id-
\sum_{j=1}^N\lambda_jB_j(S)\right\|_F^2
\Bigg\}.
\end{align}
where $\rho>0$ balances the mean and covariance contributions, and $\|\cdot\|_F$ denotes the Frobenius norm for $d\times d$ matrices.

For numerical implementation, the matrices $B_j( S)$ can be approximated by sampling directions
$\theta_1,\ldots,\theta_L\in\mathbb S^{d-1}$. With uniform weights, we set
\begin{equation}\label{eq: Bj-monte-carlo}
    \widehat B_j( S)
    :=
    \frac1L
    \sum_{\ell=1}^L
    \left(
    \frac{\theta_\ell^\top S_j\theta_\ell}
         {\theta_\ell^\top S\theta_\ell}
    \right)^{1/2}
    \theta_\ell\theta_\ell^\top.
\end{equation}

Since both terms in \eqref{eq: gaussian-analysis-ls} are nonnegative and
$\rho>0$, the minimum value is zero if and only if there exists
$\lambda\in\Delta_{N-1}$ satisfying the Gaussian sliced criticality
equations in \eqref{eq: gaussian-analysis-system}. By the preceding
Proposition \ref{prop: gaussian bary}, every such $\lambda$ is a barycentric coordinate vector of
$\mu=\mathcal N(m,S)$ with respect to the dictionary
$\{\nu_j\}_{j=1}^N$. Hence, if the minimum is zero and the minimizer is
unique, then the barycentric coordinate vector is uniquely identified. If
the minimum is zero but the set of minimizers is not a singleton, then the
same measure $\mu$ admits several barycentric coordinate vectors (i.e., the coordinates $\lambda$ cannot be uniquely recovered). If the minimum is positive, then the exact Gaussian
criticality equations have no solution, and a minimizer of
\eqref{eq: gaussian-analysis-ls} should only be interpreted as a
best-fitting approximate coordinate vector.

\subsection{Global Minimizers}

We first note that the choice \eqref{eq: 1d OT formula} makes $(\theta,s)\mapsto u_{j,\theta}^\mu(s)$, and hence the potentials $\varphi_{j,\theta}^\mu$ in \eqref{eq: Kant pot} are jointly measurable.

\begin{proof}[Proof of Theorem~\ref{thm:SW-supporting-potential}]
We first prove the supporting-potential criterion. Fix an arbitrary
competitor $\rho\in\mathcal P(K)$. For each $j\in[N]$ and for
$\sigma$-a.e. $\theta$, let $\varphi_{j,\theta}^\mu$ be a Kantorovich
potential for the pair $(\mu^\theta,\nu_j^\theta)$ and the cost
$c(s,t)=\frac12|s-t|^2$. By Kantorovich duality, if
$(\varphi_{j,\theta}^\mu)^c$ denotes the corresponding $c$-transform, then
$$
    \frac12 W_2^2(\mu^\theta,\nu_j^\theta)
    =
    \int_{\mathbb R}\varphi_{j,\theta}^\mu\,d\mu^\theta
    +
    \int_{\mathbb R}(\varphi_{j,\theta}^\mu)^c\,d\nu_j^\theta.
$$
Since the same admissible pair
$(\varphi_{j,\theta}^\mu,(\varphi_{j,\theta}^\mu)^c)$ can be tested against
$(\rho^\theta,\nu_j^\theta)$, we obtain
$$
    \frac12 W_2^2(\rho^\theta,\nu_j^\theta)
    \geq
    \int_{\mathbb R}\varphi_{j,\theta}^\mu\,d\rho^\theta
    +
    \int_{\mathbb R}(\varphi_{j,\theta}^\mu)^c\,d\nu_j^\theta.
$$
Subtracting the two identities and using the definition of pushforward gives
$$
    \frac12 W_2^2(\rho^\theta,\nu_j^\theta)
    -
    \frac12 W_2^2(\mu^\theta,\nu_j^\theta)
    \geq
    \int_{\mathbb R}\varphi_{j,\theta}^\mu\,d(\rho^\theta-\mu^\theta)
    =
    \int_{\mathbb R^d}
        \varphi_{j,\theta}^\mu(x\cdot\theta)
    \,d(\rho-\mu)(x).
$$
Multiplying by $\lambda_j$, summing over $j$, and integrating over
$\mathbb S^{d-1}$, we get
\begin{align*}
    \mathcal F^\lambda(\rho)
    -
    \mathcal F^\lambda(\mu)
    &\geq
    \sum_{j=1}^N\lambda_j
    \int_{\mathbb S^{d-1}}
    \int_{\mathbb R^d}
        \varphi_{j,\theta}^\mu(x\cdot\theta)
    \,d(\rho-\mu)(x)
    \,d\sigma(\theta) =
    \int_{\mathbb R^d}\Phi_\mu(x)\,d(\rho-\mu)(x).
\end{align*}
By the supporting-potential condition \eqref{eq:supporting-potential-condition},
$    \int_{\mathbb R^d}\Phi_\mu\,d\mu=c_\mu$ and
  $  \int_{\mathbb R^d}\Phi_\mu\,d\rho\geq c_\mu.
$
Therefore $\int_{\mathbb R^d}\Phi_\mu\,d(\rho-\mu)\geq 0$.
Hence $\mathcal F^\lambda(\rho)
    \geq
    \mathcal F^\lambda(\mu)$
for every $\rho\in\mathcal P(K)$. Thus $\mu$ is a global minimizer.
\end{proof}

\begin{remark}\label{eq: analogous to AC}
The hypothesis \eqref{eq:supporting-potential-condition} is the sliced analogue of the dual optimality condition for classical Wasserstein barycenters in \cite[Prop. 3.8]{agueh2011barycenters}; see also \cite[Sec. 3.1.5]{panaretos2020invitation}. In the classical $W_2$ barycenter problem, let $\mu$ be a candidate barycenter and let $\psi_i$ be convex potentials such that $\nabla\psi_i$ transports $\nu_i$ to $\mu$. The Agueh--Carlier condition can be written, up to an additive constant, as
\begin{equation}\label{eq: A-C}
    \sum_{i=1}^N\lambda_i\psi_i^*(y)
    \leq \tfrac12|y|^2+C
    \qquad\text{for all }y\in\mathbb R^d,
\end{equation}    
with equality for $\mu$-a.e. $y$, where $\psi_i^*$ is the convex conjugate of $\psi_i$. Thus the barycenter is certified by a global supporting inequality whose contact set contains the support of the candidate measure in the almost-everywhere sense. At differentiability points of the contact set, differentiating \eqref{eq: A-C} gives $\sum_{i=1}^N\lambda_i\nabla\psi_i^*(y)=y$
and $T_i=\nabla\psi_i^*$ is the optimal map from $\mu$ to $\nu_i$.
Our sliced condition has the same structure, with the full-dimensional Brenier potentials replaced by averaged one-dimensional optimal transport potentials:
Let
$u_{j,\theta}^\mu$ be a convex potential satisfying $(u_{j,\theta}^\mu)'(s)=T_{j,\theta}^\mu(s)$.
For the quadratic cost, one may choose a Kantorovich potential $    \varphi_{j,\theta}^\mu$ as in \eqref{eq: Kant pot}. Then,
\begin{equation*}
\Phi_\mu(x)
=\tfrac{1}{2d}|x|^2-
\sum_{j=1}^N\lambda_j\int_{\sphere}
 u_{j,\theta}^\mu(x\cdot\theta)\,d\sigma(\theta).
\end{equation*}    
Therefore, the supporting condition \eqref{eq:supporting-potential-condition} is equivalently
\begin{equation*}
\sum_{j=1}^N\lambda_j\int_{\sphere}
 u_{j,\theta}^\mu(x\cdot\theta)\,d\sigma(\theta)
\leq \tfrac{1}{2d}|x|^2-c_\mu,
\quad x\in K
\end{equation*}    
with equality  $\mu$-a.e. $x$. 
At differentiability points in the interior of $K$ where equality holds, differentiating the contact relation yields the corresponding
Karcher equation \eqref{eq: Karcher}.
\end{remark}

\begin{lemma}
\label{lem:gradient-of-sliced-potential}
In addition to the standing assumptions considered in this section,
let $O\supset K$ be an open neighborhood of $K$. Suppose that,
for every $j\in[N]$ and for $\sigma$-a.e.
$\theta\in\mathbb S^{d-1}$, the measurable representative
$T_{j,\theta}^{\mu}$ fixed above is continuous on $P_\theta(O)$.
Then the function $\Phi_\mu$ defined by
\eqref{eq: functional in terms of kant pot} belongs to $C^1(O)$, and
\begin{equation}\label{eq: grad of Phi_mu}
    \nabla\Phi_\mu(x)
    =
    \sum_{j=1}^N\lambda_j U_j^\mu(x)
    \qquad\text{for every }x\in O.
\end{equation}
\end{lemma}

\begin{proof}[Proof of Lemma~\ref{lem:gradient-of-sliced-potential}]
At every point where $u_{j,\theta}^\mu$ is differentiable,
$T_{j,\theta}^\mu=(u_{j,\theta}^\mu)'$, therefore
\begin{equation}\label{eq: potential'}
    (\varphi_{j,\theta}^\mu)'(s)
    =s-T_{j,\theta}^\mu(s).
\end{equation}    
Set $R_K:=\sup_{z\in K}|z|$.
Since $\nu_j^\theta$ is supported in $P_\theta(K)$, then the representative of
$T_{j,\theta}^{\mu}$ chosen as in \eqref{eq: 1d OT formula} satisfies $|T_{j,\theta}^{\mu}(s)|
    \leq R_K$ for all $s\in\mathbb R$,
for every $j\in[N]$ and for $\sigma$-a.e. $\theta$. Moreover,
because $T_{j,\theta}^{\mu}$ is continuous on $P_\theta(O)$, its
normalized primitive $u_{j,\theta}^{\mu}$ is of class $C^1$ there,
with $(u_{j,\theta}^{\mu})'(s)
    =
    T_{j,\theta}^{\mu}(s)$. 
Consequently, by \eqref{eq: potential'},
the function $x\mapsto
    \varphi_{j,\theta}^{\mu}(x\cdot\theta)$
is of class $C^1$ on $O$, and
\begin{equation}\label{eq:gradient-slice-potential}
    \nabla_x
    \bigl[\varphi_{j,\theta}^{\mu}(x\cdot\theta)\bigr]
    =
    \bigl(
        x\cdot\theta
        -
        T_{j,\theta}^{\mu}(x\cdot\theta)
    \bigr)\theta .
\end{equation}

Let $K'\subset O$ be compact and set $R_{K'}:=\sup_{x\in K'}|x|$.
For every $x\in K'$,
$
\left|
    x\cdot\theta
    -
    T_{j,\theta}^{\mu}(x\cdot\theta)
\right|
\leq
R_{K'}+R_K.
$
The RHS is independent of $j$ and $\theta$, and hence
is integrable over $\mathbb S^{d-1}$.
Moreover, using \eqref{eq: Kant pot}
we have, for $x\in K'$,
$
\left|
    \varphi_{j,\theta}^{\mu}(x\cdot\theta)
\right|
\leq
\frac12R_{K'}^2+R_KR_{K'}.
$
Thus the integral over the sphere defining $\Phi_\mu$ is well defined and locally bounded on $O$.

The bound on \eqref{eq:gradient-slice-potential} permits
differentiation under the spherical integral. Therefore,
\begin{align*}
    \nabla\Phi_\mu(x)
    &=
    \sum_{j=1}^N\lambda_j
    \int_{\mathbb S^{d-1}}
    \bigl(
        x\cdot\theta
        -
        T_{j,\theta}^{\mu}(x\cdot\theta)
    \bigr)\theta
    \,d\sigma(\theta) =
    \sum_{j=1}^N\lambda_jU_j^\mu(x).
\end{align*}

Finally, if $x_n\to x$ in $O$, all these points lie in a common
compact subset of $O$. Continuity of $T_{j,\theta}^{\mu}$ gives
pointwise convergence of the integrands for $\sigma$-a.e. $\theta$,
while the preceding uniform bound permits dominated convergence.
Hence $\nabla\Phi_\mu(x_n)\to\nabla\Phi_\mu(x)$, proving that
$\Phi_\mu\in C^1(O)$.
\end{proof}

\begin{proof}[Proof of Corollary~\ref{cor:SW-Karcher-to-barycenter}]

By assumption and Lemma \ref{lem:gradient-of-sliced-potential},
$
    \nabla\Phi_\mu(x)
    =
    \sum_{j=1}^N\lambda_j U_j^\mu(x)
    =
    0$
    for all $x\in K.
$
Since $K$ is convex, for any $x,y\in K$ the segment
$(1-t)x+ty$ lies in $K$. Hence
\begin{align*}
    \Phi_\mu(y)-\Phi_\mu(x)
    &=
    \int_0^1
    \left\langle
        \nabla\Phi_\mu((1-t)x+ty),
        y-x
    \right\rangle
    \,dt =0.
\end{align*}
Therefore $\Phi_\mu$ is constant on $K$. Let this constant be
$c_\mu$. Then
$
    \Phi_\mu(x)=c_\mu
    $ for all $x\in K.
$
In particular, condition \eqref{eq:supporting-potential-condition} holds. Applying
Theorem~\ref{thm:SW-supporting-potential}, we conclude that $\mathcal F^\lambda(\mu)
=\min_{\rho\in\mathcal P(K)}\mathcal F^\lambda(\rho).$
\end{proof}

\begin{lemma}
\label{lem:uniform-Lipschitz-slices}
In addition to the standing assumptions considered in this section, suppose that there exists a constant $\mathrm L<\infty$, independent of $j$
and $\theta$, such that, after choosing measurable representatives, they satisfy \eqref{eq: unif lip cond}, i.e.,
\begin{equation*}    |T_{j,\theta}^\mu(s)-T_{j,\theta}^\mu(t)|
    \leq \mathrm L|s-t|
    \qquad
    \text{for all }s,t\in\mathbb R,
\end{equation*}
for every $j\in[N]$ and for $\sigma$-a.e. $\theta$. Then, the map $\Phi_\mu$ defined by \eqref{eq: functional in terms of kant pot} belongs to $C^{1,1}(\mathbb R^d)$ and  the identity \eqref{eq: grad of Phi_mu}
holds for every $x\in\mathbb R^d$.
\end{lemma}

\begin{proof}[Proof of Lemma~\ref{lem:uniform-Lipschitz-slices}] The proof is direct.
For $j\in[N]$ and $\sigma$-a.e. $\theta\in\mathbb S^{d-1}$, set
$
G_{j,\theta}(x)
:=
\bigl(x\cdot\theta
      -T_{j,\theta}^\mu(x\cdot\theta)\bigr)\theta .
$
Since $T_{j,\theta}^\mu$ is $L$-Lipschitz and we have the expression \eqref{eq: potential'},
the map $x\mapsto\varphi_{j,\theta}^\mu(x\cdot\theta)$ is of class
$C^1$, with gradient $G_{j,\theta}$. Moreover,
\begin{align*}
\|G_{j,\theta}(x)-G_{j,\theta}(y)\|
&\leq |(x-y)\cdot\theta|
 +\bigl|T_{j,\theta}^\mu(x\cdot\theta)
       -T_{j,\theta}^\mu(y\cdot\theta)\bigr| \leq (1+\mathrm L)\|x-y\|.
\end{align*}
Compactness of $K$ and the uniform Lipschitz bound give the local
uniform domination required to differentiate under the integral over the sphere. Therefore,
$
\nabla\Phi_\mu(x)
=
\sum_{j=1}^N\lambda_j
\int_{\mathbb S^{d-1}}G_{j,\theta}(x)\,d\sigma(\theta)
=
\sum_{j=1}^N\lambda_jU_j^\mu(x).
$
Finally,
$$
\|\nabla\Phi_\mu(x)-\nabla\Phi_\mu(y)\|
\leq
\sum_{j=1}^N\lambda_j
\int_{\mathbb S^{d-1}}
\|G_{j,\theta}(x)-G_{j,\theta}(y)\|
\,d\sigma(\theta)
\leq
(1+\mathrm L)\|x-y\|.
$$
Thus $\Phi_\mu\in C^{1,1}(\mathbb R^d)$.
\end{proof}

\begin{remark} The uniform Lipschitz assumption in Lemma \ref{lem:uniform-Lipschitz-slices} is deliberately stronger than the minimal regularity needed for $\Phi_\mu\in C^1$, but it has the advantage of being simple and directly checkable. Its role is analogous to the regularity of Brenier maps \cite{caffarelli1992regularity} in the classical Wasserstein criterion for Karcher means \cite[Thm. 3.1.15]{panaretos2020invitation}: regularity upgrades an identity that is initially valid only $\mu$-a.e. to a pointwise identity on the support. In the present sliced setting, Lemma~\ref{lem:uniform-Lipschitz-slices} gives a continuous representative of $\sum_{j=1}^N\lambda_jU_j^\mu$ and identifies it with the classical gradient $\nabla\Phi_\mu$. \end{remark}

\begin{example}
\label{ex:homothetic-uniform-modulus} As an elementary example where Lemma \ref{lem:uniform-Lipschitz-slices}  applies, consider translations and dilations of a fixed probability measure $\mu$ of the form $\nu_j=(a_j\Id+b_j)_\#\mu$, 
 $a_j>0$, $b_j\in \mathbb R^d$.
Then, $T_{j, \theta}^\mu(s)=a_j s+b_j \cdot \theta$, 
and hence $|T_{j,\theta}^\mu(s)-T_{j,\theta}^\mu(t)|
\leq \left(\max_j a_j\right)|s-t|.$

Similarly, if we fix a probability measure $\rho$ compactly supported and spherically symmetric, by considering invertible $d\times d$ matrices $A_j$, $j=0,1,\dots,N$, and setting $\mu=(m_0+A_0z)_\#\rho$, 
$\nu_j=(m_j+A_jz)_\#\rho$, $j\in[N]$,
then 
\begin{align*}
T_{j,\theta}^\mu(s)
&=m_j\cdot\theta+
\frac{\|A_j^\top\theta\|}{\|A_0^\top\theta\|}
\left(s-m_0\cdot\theta\right).
\end{align*} 
and therefore it follows that the hypotheses of  Lemma \ref{lem:uniform-Lipschitz-slices} are satisfied.

\end{example}

\begin{proof}[Proof of Corollary~\ref{cor:exact-SW-analysis}]
($\Rightarrow$): Suppose first that $\mu$ minimizes $\mathcal F^\lambda$ over $\mathcal P(K)$. Let $v\in C_c^1(\operatorname{int}K;\mathbb R^d)$. For sufficiently small $|t|$, the perturbation $\mu_t=(\Id+tv)_\#\mu$
still belongs to $\mathcal P(K)$. Hence the first-variation vanishes:
$
    \int_K\left\langle
    \sum_{j=1}^N\lambda_jU_j^\mu(x),v(x)
    \right\rangle d\mu(x)=0
$
($C_c^1(\operatorname{int}K;\mathbb R^d)$ is dense in $L^2(\mu;\mathbb R^d)$). Therefore, \eqref{eq: Karcher pp} holds, and so
$\lambda^\top\mathcal A(\mu)\lambda=0$. Since $\mathcal A(\mu)$ is positive semi-definite, this is equivalent to $\mathcal A(\mu)\lambda=0$.

($\Leftarrow$): Conversely, assume that $\mathcal A(\mu)\lambda=0$. Then, we have \eqref{eq: Karcher pp}.
By Lemma \ref{lem:uniform-Lipschitz-slices}, the vector field $x\mapsto\sum_j\lambda_jU_j^\mu(x)$ has a continuous representative and equals $\nabla\Phi_\mu(x)$. Since $\operatorname{supp}\mu=K$, a continuous function that vanishes $\mu$-a.e. must vanish at every point of $K$. Thus
$
    \sum_{j=1}^N\lambda_jU_j^\mu(x)=0
    $ for every $x\in K.
$
Corollary \ref{cor:SW-Karcher-to-barycenter} now implies that $\mu$ is a global minimizer of $\mathcal F^\lambda$ over $\mathcal P(K)$.
\end{proof}

\paragraph{Detailed Failure of Convex-Support Preservation for SW Barycenters.}

\begin{proof}[Proof of Example \ref{prop:SW-convex-support-counterexample}] Here we denote $J(\mu):=\frac12 SW_2^2(\mu,\nu_1)
           +\frac12 SW_2^2(\mu,\nu_2)$.
Both targets are centrally symmetric. Since $W_2^2$ is convex in its
first argument under ordinary linear mixtures, the same is true of
$ SW_2^2$ and $J$. Indeed, mixing couplings with a fixed
second marginal proves the convexity of $W_2^2$, and projection commutes
with mixtures. Moreover, symmetry of the targets gives
$
    J\bigl((-\mathrm{Id})_\#\mu\bigr)=J(\mu).
$
Consequently, the symmetrization
$    \mu^{\mathrm{sym}}
    :=\frac12\bigl(\mu+(-\mathrm{Id})_\#\mu\bigr)
$
satisfies $J(\mu^{\mathrm{sym}})\leq J(\mu)$. Thus the global infimum
may be computed over centrally symmetric measures.

Let $\rho\sim Y$ be a centrally symmetric real random variable with finite
second moment, and let $A\geq0$. Then
\begin{equation}
\label{eq:cs-symmetric-two-point-cost}
    W_2^2\!\left(\rho,\frac12(\delta_{-A}+\delta_A)\right)
    =\mathbb E[Y^2]+A^2-2A\mathbb E|Y|.
\end{equation}
To see this, every coupling of $Y$ with
$Z\sim\frac12(\delta_{-A}+\delta_A)$ satisfies
$YZ\leq A|Y|$. Equality is attained by pairing positive values of $Y$
with $A$ and negative values with $-A$, splitting any mass at $Y=0$
equally between the two target atoms. Symmetry of $Y$ makes this a
coupling with the required target law. The case $A=0$ is immediate.

Now let $X\sim\mu$, where $\mu$ is centrally symmetric. For
$\theta\in\mathbb S^1$, the projected targets are
\[
    \nu_1^\theta
    =\frac12\bigl(\delta_{-|a\cdot\theta|}
                  +\delta_{|a\cdot\theta|}\bigr),
    \qquad
    \nu_2^\theta
    =\frac12\bigl(\delta_{-|b\cdot\theta|}
                  +\delta_{|b\cdot\theta|}\bigr).
\]
Applying \eqref{eq:cs-symmetric-two-point-cost} to $Y=X\cdot\theta$
and integrating over $\theta$ gives
\begin{equation}
\label{eq:cs-pointwise-functional}
    J(\mu)
    =\frac12+\int_{\mathbb R^2}
       \left(\frac{|x|^2}{2}-H(x)\right)\,d\mu(x),
\end{equation}
where
\begin{equation}
\label{eq:cs-H-definition}
    H(x):=\int_{\mathbb S^1}
       \bigl(|a\cdot\theta|+|b\cdot\theta|\bigr)
       |x\cdot\theta|\,d\sigma(\theta).
\end{equation}
Here we used $|a|=|b|=1$ and
$
    \int_{\mathbb S^1}(x\cdot\theta)^2\,d\sigma(\theta)
    =\frac{|x|^2}{2}.
$

The function $H$ is even, nonnegative, and positively homogeneous.
Writing $x=ru$ with $r\geq0$ and $u\in\mathbb S^1$, we have
\[
    \frac{|x|^2}{2}-H(x)
    =\frac{r^2}{2}-rH(u)
    =\frac12\bigl(r-H(u)\bigr)^2-\frac12H(u)^2.
\]
For each fixed $u$, this is uniquely minimized at $r=H(u)$.
Thus the global minimizers of the pointwise integrand in
\eqref{eq:cs-pointwise-functional} occur in directions maximizing
$H$ on $\mathbb S^1$.

\noindent\emph{Step 1: Computation of the angular kernel.}
For unit vectors $u,v\in\mathbb S^1$, let
$\delta\in[0,\pi/2]$ be the smaller angle between the unoriented lines
they span, and set
\[
    \kappa(\delta):=\int_{\mathbb S^1}
        |u\cdot\theta|\,|v\cdot\theta|\,d\sigma(\theta).
\]
By rotation invariance and the invariance under replacing either
vector by its negative,
\[
    \kappa(\delta)
    =\frac{1}{2\pi}\int_0^{2\pi}
        |\cos\varphi|\,|\cos(\varphi-\delta)|\,d\varphi.
\]
The integrand is $\pi$-periodic. For $0\leq\delta\leq\pi/2$,
removing the absolute values therefore gives
\begin{align*}
    \pi\kappa(\delta)
    &=\int_0^\pi\cos\varphi\cos(\varphi-\delta)\,d\varphi
      -2\int_{\pi/2}^{\pi/2+\delta}
         \cos\varphi\cos(\varphi-\delta)\,d\varphi=\frac\pi2\cos\delta
      -\bigl(\delta\cos\delta-\sin\delta\bigr).
\end{align*}
Consequently,
\begin{equation}
\label{eq:cs-kappa-formula}
    \kappa(\delta)
    =\frac1\pi\left[
        \sin\delta+\left(\frac\pi2-\delta\right)\cos\delta
      \right],
    \qquad 0\leq\delta\leq\frac\pi2.
\end{equation}
In particular,
\begin{equation}
\label{eq:cs-kappa-derivative}
    \kappa'(\delta)
    =-\frac1\pi\left(\frac\pi2-\delta\right)\sin\delta<0,
    \qquad 0<\delta<\frac\pi2.
\end{equation}
Thus $\kappa$ is strictly decreasing on $[0,\pi/2]$.

\noindent\emph{Step 2: Computation of $H(e_1)$.}
Both $a$ and $b$ make angle $\pi/6$ with $e_1$. Hence
\begin{align*}
    H(e_1)
    &=2\kappa\!\left(\frac\pi6\right)=\frac2\pi\left[
        \frac12+\frac\pi3\frac{\sqrt3}{2}
      \right]
     =\frac1\pi+\frac{\sqrt3}{3}
     =M.
\end{align*}

\noindent\emph{Step 3: $H$ is maximized only at $\pm e_1$.}
Write $u_\alpha=(\cos\alpha,\sin\alpha)$ and set $\beta:=\pi/6$.
Since $H$ is even and invariant under reflection across the $x$-axis,
it suffices to consider $\alpha\in[0,\pi/2]$.

If $0\leq\alpha\leq\beta$, the relevant line angles are
$\beta-\alpha$ and $\beta+\alpha$, so
$
    H(u_\alpha)=\kappa(\beta-\alpha)+\kappa(\beta+\alpha).
$
Using \eqref{eq:cs-kappa-formula} and simplifying gives
\begin{equation}
\label{eq:cs-angular-comparison}
    \pi\bigl(H(u_\alpha)-H(e_1)\bigr)
    =\alpha\sin\alpha
       -\left(1+\frac\pi{\sqrt3}\right)(1-\cos\alpha).
\end{equation}
For $0<\alpha\leq\beta$,
$
    \frac{\alpha\sin\alpha}{1-\cos\alpha}
    =\alpha\cot\frac\alpha2<2,
$
because $\tan t>t$ for $0<t<\pi/2$. Since
$1+\pi/\sqrt3>2$, equation \eqref{eq:cs-angular-comparison} implies
$
    H(u_\alpha)<H(e_1)$,
    $0<\alpha\leq\frac\pi6.$ 

If $\beta\leq\alpha\leq\pi/3$, the relevant line angles are
$\alpha-\beta$ and $\alpha+\beta$. Therefore
$
    H(u_\alpha)=\kappa(\alpha-\beta)+\kappa(\alpha+\beta),
$ 
and, for $\beta<\alpha<\pi/3$,
$$
    \frac{d}{d\alpha}H(u_\alpha)
    =\kappa'(\alpha-\beta)+\kappa'(\alpha+\beta)<0.
$$
Thus $H(u_\alpha)$ is strictly decreasing on $[\beta,\pi/3]$
and remains strictly below $H(e_1)$.

Finally, if $\pi/3\leq\alpha\leq\pi/2$, the relevant line angles are
$
    \alpha-\beta\geq\beta,$ 
    $\pi-\alpha-\beta>\beta.
$ 
Both belong to $[0,\pi/2]$, so strict monotonicity of $\kappa$ yields
\[
    H(u_\alpha)
    =\kappa(\alpha-\beta)+\kappa(\pi-\alpha-\beta)
    <2\kappa(\beta)=H(e_1).
\]
It follows that
$
    \operatorname*{arg\,max}_{u\in\mathbb S^1}H(u)
    =\{e_1,-e_1\}.
$

Consequently, the pointwise function $|x|^2/2-H(x)$ is minimized
precisely at $x=\pm Me_1$, where its value is $-M^2/2$.
By \eqref{eq:cs-pointwise-functional}, the unique centrally symmetric
minimizer is
\[
    \mu_* =\frac12\delta_{Me_1}+\frac12\delta_{-Me_1}.
\]
The symmetrization argument shows that this is also a global
minimizer over all of $\mathcal P_2(\mathbb R^2)$.

Every vertex of $K$ has first coordinate $\pm\sqrt3/2$, so
\[
    x\in K\quad\Longrightarrow\quad |x_1|\leq\frac{\sqrt3}{2}.
\]
On the other hand,
$
    M-\frac{\sqrt3}{2}
    =\frac1\pi-\frac{\sqrt3}{6}>0,
$
where we used $\pi<2\sqrt3$. Therefore $\pm Me_1\notin K$, and
$\mu_*\notin\mathcal P(K)$.

Suppose, finally, that an unrestricted minimizer $\widetilde\mu$
were supported on $K$. Since $K=-K$, its symmetrization
$\widetilde\mu^{\mathrm{sym}}$ would also be supported on $K$.
Moreover,
\[
    J(\widetilde\mu^{\mathrm{sym}})
    \leq J(\widetilde\mu)
    =\inf_{\eta\in\mathcal P_2(\mathbb R^2)}J(\eta),
\]
so $\widetilde\mu^{\mathrm{sym}}$ would be a centrally symmetric
global minimizer. It would therefore equal $\mu_*$, contradicting
$\mu_*\notin\mathcal P(K)$. Hence no unrestricted SW barycenter
is supported on $K$.
\end{proof}

\section{Well-Known Properties and Identities}

\begin{lemma}
    The quantity $SW_2(\cdot,\cdot)$ defines a true metric on $\mathcal{P}_2(\mathbb R^d)$.
\end{lemma}
\begin{proof}
    It is symmetric by being defined as averages of 1D Wasserstein distances, which are symmetric. The triangle inequality follows from the triangle inequality for the 1D Wasserstein distances together with Minkowski’s inequality. Finally, the identity of indiscernibles follows from the corresponding property of the 1D Wasserstein distances involved in definition \eqref{eq: SW}, together with either the Fourier Slice Theorem or the Cramér-Wold Theorem.
\end{proof}

\begin{lemma}[Flat convexity of the SW distance]
Let $\nu_1,\nu_2,\mu \in \mathcal{P}_2(\mathbb{R}^d)$, and define the convex combination
$$
\nu_t=t\nu_1+(1-t)\nu_2,\qquad t\in[0,1].
$$
Then, the squared SW distance is convex in each argument: for $0\leq t\leq 1$,
\begin{align*}
SW_2(\nu_t,\mu)^2
&\leq t\,SW_2(\nu_1,\mu)^2+(1-t)\,SW_2(\nu_2,\mu)^2, \, \text{ and}
\end{align*}
\begin{align*}
SW_2(\mu,\nu_t)^2
&\leq t\,SW_2(\mu,\nu_1)^2+(1-t)\,SW_2(\mu,\nu_2)^2.
\end{align*}
\end{lemma}

\begin{proof} We refer the reader to \cite{nguyen2025introduction} and we reproduce the proof here for completeness. 

For each direction $\theta \in \mathbb{S}^{d-1}$, we have the projected linear combination
$
\nu_t^\theta
=
t\,\nu^\theta_1
+
(1-t)\,\nu^\theta_2$,
and one can use that $W_2^2$ is convex in one of its arguments, that is,
$$
W_2(\nu^\theta_t,\mu^\theta)^2
\leq
t\,W_2(\nu^\theta_1,\mu^\theta)^2
+
(1-t)\,W_2(\nu^\theta_2,\mu^\theta)^2.
$$
Finally,
\begin{align*}
SW_2(\nu_t,\mu)^2
&=\int_{\sphere} W_2(\nu_t^\theta,\mu^\theta)^2\,d\sigma(\theta)\\
&\leq t\int_{\sphere}W_2(\nu_1^\theta,\mu^\theta)^2\,d\sigma(\theta) +(1-t)\int_{\sphere}W_2(\nu_2^\theta,\mu^\theta)^2\,d\sigma(\theta)\\
&=t\,SW_2(\nu_1,\mu)^2+(1-t)\,SW_2(\nu_2,\mu)^2.
\end{align*}
So, $\nu\mapsto \SW(\nu,\mu)^2$ is convex under ordinary linear mixtures of measures . By symmetry, the same statement holds in the second argument. 
\end{proof}

\begin{lemma}
Let $\sigma$ be the uniform measure on the $\sphere$. Then
$$\int_{\sphere}(x\cdot\theta)\theta\,d\sigma(\theta)=\frac{x}{d}, \qquad x\in \mathbb R^d.$$
\end{lemma}

\begin{proof}
    Given $x\in \mathbb R^d$, choose a rotation $R$ such that $Rx=\|x\|e_1$, where $e_1=(1,0,\dots, 0)$ is the first canonical vector in $\mathbb R^d$. Then $ \theta \cdot x=(R\theta) \cdot (Rx)$ and, since $\sigma$ is rotation-invariant, for any $1\leq p<\infty$
\begin{align*}
\int_{\mathbb S^{d-1}}|\theta\cdot x|^p\,d\sigma(\theta)
&=\|x\|^p\int_{\mathbb S^{d-1}}|(R\theta)\cdot e_1|^p\,d\sigma(\theta)=\|x\|^p\int_{\mathbb S^{d-1}}|\theta\cdot e_1|^p\,d\sigma(\theta).
\end{align*}
    In particular, if $p=2$, then $1=
    \int_{\mathbb{S}^{d-1}}\sum_{j=1}^d|\theta\cdot e_j|^2 d\sigma(\theta)$, so $\frac{1}{d}=
    \int_{\mathbb{S}^{d-1}}|\theta \cdot e_1|^2 d\sigma(\theta)$. 
Thus, for every $x\in\mathbb R^d$,
$
\int_{\sphere}|x\cdot\theta|^2\,d\sigma(\theta)
=\frac{\|x\|^2}{d}.
$
By polarization, for every $y\in\mathbb R^d$,
\begin{align*}
\left(\int_{\sphere}(x\cdot\theta)\theta\,d\sigma(\theta)\right)
\cdot y
&=\int_{\sphere}(x\cdot\theta)(y\cdot\theta)\,d\sigma(\theta)=\frac14\int_{\sphere}
\left(|(x+y)\cdot\theta|^2-|(x-y)\cdot\theta|^2\right)
\,d\sigma(\theta)\\
&=\frac{\|x+y\|^2-\|x-y\|^2}{4d}
=\frac{x\cdot y}{d}.
\end{align*}
Since this holds for every $y\in\mathbb R^d$, we conclude that
$
\int_{\sphere}(x\cdot\theta)\theta\,d\sigma(\theta)=\frac{x}{d}.
$
\end{proof}

\begin{lemma}[1D Gelbrich Inequality \cite{gelbrich1990formula, dowson1982frechet, olkin1982distance}]\label{lem: Gelbrich}
Let $\eta\in\mathcal P_2(\mathbb R)$ have mean $a$ and variance $\alpha^2$, and let $\zeta=\mathcal N(b,\beta^2)$ with $\alpha,\beta\geq0$. Then
$$
    W_2^2(\eta,\zeta)
    \geq
    (a-b)^2+(\alpha-\beta)^2.
$$
If $\beta>0$, equality holds if and only if $\eta=\mathcal N(a,\alpha^2)$. If $\beta=0$, then $\zeta=\delta_b$ and equality holds for every $\eta\in\mathcal P_2(\mathbb R)$ with mean $a$ and variance $\alpha^2$.
\end{lemma}

\begin{proof}
If $\beta=0$, then $\zeta=\delta_b$ and
$
W_2^2(\eta,\delta_b)=\int_{\mathbb R}|x-b|^2\,d\eta(x)
=(a-b)^2+\alpha^2=(a-b)^2+(\alpha-\beta)^2,
$
so the asserted statement is immediate. We may therefore assume $\beta>0$.

Let $Q_\eta$ and $Q_\zeta$ denote the quantile functions of $\eta$ and $\zeta$.
In one dimension,
$$
    W_2^2(\eta,\zeta)
    =
    \int_0^1
    \bigl|Q_\eta(t)-Q_\zeta(t)\bigr|^2\,dt.
$$
Let $U\sim \operatorname{Unif}(0,1)$, and define $X:=Q_\eta(U)$, $Y:=Q_\zeta(U).$
Then $X\sim\eta$, $Y\sim\zeta$, and
$W_2^2(\eta,\zeta)
    =
    \mathbb E|X-Y|^2$. Since $\zeta=\mathcal N(b,\beta^2)$, we may write $Y=b+\beta \Phi^{-1}(U)$,
where $\Phi$ is the standard Gaussian cumulative distribution function.
Expanding the square gives
\begin{align*}
    \mathbb E|X-Y|^2
    &=
    \mathbb E\bigl|(X-a)-(Y-b)+(a-b)\bigr|^2 \\
    &=
    (a-b)^2
    +
    \mathbb E|X-a|^2
    +
    \mathbb E|Y-b|^2
    -
    2\mathbb E[(X-a)(Y-b)] \\
    &=
    (a-b)^2+\alpha^2+\beta^2
    -
    2\operatorname{Cov}(X,Y).
\end{align*}
By Cauchy--Schwarz,
$
    \operatorname{Cov}(X,Y)
    =
    \mathbb E[(X-a)(Y-b)]
    \leq
    \alpha\beta.
$
Therefore,
$
    W_2^2(\eta,\zeta)
    \geq
    (a-b)^2+\alpha^2+\beta^2-2\alpha\beta
    =
    (a-b)^2+(\alpha-\beta)^2.
$

Now assume equality holds. Then equality must hold in Cauchy--Schwarz. Since
$\beta>0$, this implies that $X-a$ is a scalar multiple of $Y-b$, namely
$
    X-a=\frac{\alpha}{\beta}(Y-b)
    \text{ a.s.}
$
Thus $ X
    =
    a+\frac{\alpha}{\beta}(Y-b)$, and since $Y$ is Gaussian, $X$ is Gaussian as well. Hence, the law of $X$ is $\eta=\mathcal N(a,\alpha^2)$.
    
Conversely, if $\eta=\mathcal N(a,\alpha^2)$, then the monotone optimal coupling
between $\eta$ and $\zeta$ is affine, i.e., 
$
    X=a+\frac{\alpha}{\beta}(Y-b),
$
and direct computation gives
$ W_2^2(\eta,\zeta)
    =
    (a-b)^2+(\alpha-\beta)^2$. The case $\alpha=0$ is included: then $\eta=\mathcal N(a,0)=\delta_a$.
\end{proof}

\section{Algorithms and Experiments}\label{sec: synt}

\paragraph{Code availability.}
An anonymized implementation is available at:

\noindent  \url{https://anonymous.4open.science/r/sw-barycenters-analysis}.  

\paragraph{Computing resources.}
For all the figures and tables in this work, all experiments used Python with NumPy and SciPy on an
Intel Core Ultra 7 165U CPU under Windows~11, without GPU
acceleration.

\paragraph{Dataset licenses.}
ModelNet-C \citep{ren2022benchmark} is distributed under
PointCloud-C's non-commercial-use terms, which require
attribution.\footnote{\url{https://pointcloud-c.github.io/download.html}}
The underlying ModelNet40 collection \citep{wu20153d} is
provided for academic research, with copyright retained by
the original CAD-model authors.\footnote{\url{https://modelnet.cs.princeton.edu/download.html}}

We also utilize Point Cloud MNIST \citep{garcia2021pointcloudmnist},
a point-cloud representation of the MNIST handwritten-digit dataset
\citep{lecun1998gradient}, comprising ten digit classes.
The Point Cloud MNIST conversion repository
\citep{garcia2021pointcloudmnist} carries an MIT
license.\footnote{\url{https://huggingface.co/datasets/cgarciae/point-cloud-mnist/blob/97cc82c6aa94495ad2ba378d8716f274ae04c092/LICENSE}}

\subsection{Experimental Specifications in Figure \ref{fig:pointcloud-analysis-diagnostics}}\label{app:classif}

For classification in Figure \ref{fig:pointcloud-analysis-diagnostics}, both datasets use nested dictionaries with three
or five references per class and $L=32$ directions.
ModelNet-C uses ten repeated classwise 80/20 splits, with references
selected from the training pools and the same 100 test objects
evaluated clean and under five corruptions, yielding 600
query--condition pairs per trial. Test objects and directions are
shared across dictionary sizes within each trial.
Point Cloud MNIST uses ten training-dictionary draws and a fixed
balanced subset of 1,000 official-test digits; both this subset and
the directions remain fixed across draws and dictionary sizes.

In the classification experiments, the global coordinate quadratic programming (QP)
and the class-restricted QPs used to compute the rejection score
are solved by an active-set method with stationarity-gap
tolerance $10^{-9}$ and a feasible Frank--Wolfe fallback.
Residual rankings are recomputed for each dictionary size, while
both classifiers retain the same queries within each panel.
Pointwise 95\% Student-$t$ intervals are computed over the ten trials;
MNIST intervals are conditional on its fixed test subset and
directions. 

\paragraph{Wall-clock runtime (Figure \ref{fig:pointcloud-analysis-diagnostics}).} For the classification experiments, encoding and diagnostics
averaged $56.1$ and $50.1$\,ms per query on ModelNet-C with
15 and 25 references, respectively. These timings were obtained
in separate runs. 
For each MNIST query, we computed the Gram matrix and SW distances
using all 50 references once, and reused the corresponding entries
for the 30-reference experiment. The shared computation
of the 50-reference Gram matrix and SW distances took
$22.4$\,ms per query, followed by $6.3$ or $8.3$\,ms for
coordinate estimation and class-residual diagnostics with
30 or 50 references, respectively. These measurements exclude
data loading and dictionary preparation.

\subsection{Synthesis Algorithm}\label{app:synth}

We emphasize that the synthesis part of our approach is closely related to the sliced Wasserstein barycenter framework in \cite{boneel}. Both approaches rely on the same fundamental observation: after projecting the measures onto one-dimensional directions, the corresponding OT problems become explicit through monotone rearrangements. Thus, from a computational viewpoint, both methods use the ingredients
    $\theta\in\mathbb S^{d-1},\, 
    \mu^\theta=(P_\theta)_\#\mu, \, 
    \nu_j^\theta=(P_\theta)_\#\nu_j$,
together with the one-dimensional OT maps
$ T_{j,\theta}^\mu=T_{\mu^\theta\to\nu_j^\theta}.
$
In particular, any particle implementation of our synthesis procedure involves sampling directions, projecting the particles, sorting along each direction, computing one-dimensional transport maps, and averaging the resulting displacement fields over directions. This is the main computational overlap with \cite{boneel}.
In particular, after a discretization,
our synthesis scheme becomes a particle descent for the same sliced barycenter
energy considered in~\cite{boneel}. Thus the difference is not the use of
one-dimensional optimal transports or the projection machinery. Rather, the approaches differ on the points of view used to perform the corresponding  derivations, as noted in the main text. 

\paragraph{Particle Discretization and Discrete Gradient.}
We represent the candidate measure $\mu$ by an empirical measure
\begin{equation*}
\mu_X = \frac1m \sum_{k=1}^m \delta_{X_k},
\quad
X=(X_1,\dots,X_m)\in (\R^d)^m.
\end{equation*}
Likewise, each target $\nu_j$ is approximated by a point cloud
$Y^{(j)}=(Y^{(j)}_1,\dots,Y^{(j)}_{m_j})$.

For a finite set of directions $\Theta=\{\theta_1,\dots,\theta_L\}\subset \sphere$, we approximate
the integral over the sphere by an average:
\begin{equation}\label{eq: empirical direction}
U_j^\mu(x)
\approx U_j^\Theta(x)
:=
\frac1L \sum_{\theta\in\Theta}
\bigl(x\cdot \theta - T_{j,\theta}^\mu(x\cdot\theta)\bigr)\theta,
\end{equation}
where, in one dimension, the transport maps $ T_{j,\theta}^\mu$ are approximated by sorting the projected samples and matching their empirical quantiles or monotone rearrangements. We abbreviate such monotone quantile matching by $\text{MQM}$.

We note that the empirical measures are atomic, whereas Proposition \ref{thm: gradient} is stated under non-atomicity of the projected source. The formulas in this section should therefore be understood as a numerical discretization. 
We also note that, for generic point clouds and randomly sampled directions, projected ties occur with probability zero. 

Evaluating formula \eqref{eq: grad of sw bary} at the particles $X_k$ gives the discrete gradient field
\begin{equation}\label{eq: avg discrete}
\sum_{j=1}^N  \lambda_j U_j^\Theta(X_k),
\qquad k=1,\dots,m.
\end{equation}
The resulting Euler step is
\begin{equation}\label{eq: Euler ite}
X_k^{(\ell+1)}
=
X_k^{(\ell)} - \tau_\ell\, \sum_{j=1}^N  \lambda_j U_j^\Theta(X_k^{(\ell)}).
\end{equation}
In practice, we use either a fixed step size $\tau_\ell=\tau$, or a mild stochastic variant in which the set of directions $\Theta_\ell$ is resampled
at each iteration. 
The Synthesis procedure, which consists of performing the iterations in \eqref{eq: Euler ite}, is explicitly given in the Appendix (see Algorithm \ref{alg: synth}). 

For Analysis, see Algorithm \ref{alg: analysis}:  For predefined barycentric templates $\{Y^{(j)}\}_{j=1}^N$, given a query point cloud $X=(X_1,\ldots,X_m)$ representing a probability measure $\mu$,  one needs to build, in the discrete particle setting, the matrix \eqref{eq: A}, that is,
\begin{equation}\label{eq: empirical matrix A}
\mathcal A_{ij}(\mu)
\approx
\frac1m \sum_{k=1}^m
\langle U_i^\Theta(X_k), U_j^\Theta(X_k)\rangle,
\end{equation}
in order to then retrieve estimated coordinates $\widehat\lambda$ of $X$ in terms of the templates through the simplex-constrained quadratic problem
\begin{equation}\label{eq: min}
\min_{\lambda\in\Delta_{N-1} } \lambda^\top \mathcal A(\mu) \lambda.
\end{equation} 
One also can retrieve a residual $\mathrm{res}_\mathrm{an}$ representing the magnitude of the empirical gradient $\left({\widehat{\lambda}^\top\mathcal{A}(\mu)\widehat{\lambda}}\right)^{\frac12}$. 

\begin{algorithm}[ht!]
\caption{Synthesis - SW Barycenter}\label{alg: synth}
\small
\begin{algorithmic}[1]
\Require Template point clouds $\{Y^{(j)}\}_{j=1}^N$, weights
$\lambda\in\Delta_{N-1}$, initial point cloud
$X^{(0)}=(X_k^{(0)})_{k=1}^m$, number of directions $L$,
step-size rule $\tau_\ell$, tolerance $\varepsilon$, and maximum
number of iterations $K$.
\Ensure A particle cloud approximating a critical point of
$\mathcal F^\lambda$.
\Statex \State $\ell\gets 0$
\Repeat
    \State Sample $\Theta_\ell=\{\theta_q\}_{q=1}^L\subset\sphere$.
    \For{$j=1,\ldots,N$}
        \For{$k=1,\ldots,m$}
            \State $U_{j,k}\gets 0\in\R^d$
        \EndFor
        \For{$q=1,\ldots,L$}
            \State $a_k\gets X_k^{(\ell)}\cdot\theta_q$ for
            $k=1,\ldots,m$
            \State $b_r\gets Y_r^{(j)}\cdot\theta_q$ for
            $r=1,\ldots,m_j$
            \State $(t_k)_{k=1}^m\gets
            \operatorname{MQM}((a_k)_{k=1}^m,(b_r)_{r=1}^{m_j})$
            \For{$k=1,\ldots,m$}
                \State $U_{j,k}\gets
                U_{j,k}+L^{-1}(a_k-t_k)\theta_q$
            \EndFor
        \EndFor
    \EndFor
    \For{$k=1,\ldots,m$}
        \State $G_k\gets\sum_{j=1}^N\lambda_jU_{j,k}$
        \State $X_k^+\gets X_k^{(\ell)}-\tau_\ell G_k$
    \EndFor
    \State $r_\ell \leftarrow
\left(\frac{1}{m}\sum_{k=1}^{m}
\|X_k^{+}-X_k^{(\ell)}\|^2\right)^{1/2}$
                \State $X^{(\ell+1)}\gets X^+$; $\ell\gets\ell+1$
\Until{$r_{\ell-1}\leq\varepsilon$ or  $\ell=K$}
\State \Return $X^{(\ell)}$
\end{algorithmic}
\end{algorithm}

\subsection{Synthesis-Analysis Experiments}\label{app:f-b}

This section provides controlled experiments on the synthesis and
analysis of barycentric models in two dimensions. The template measures
follow the examples in \cite{boneel}. 
Indeed, we use three animal silhouettes (ostrich, duck, and rabbit)
from the toolbox accompanying \citet{boneel}, and three
locally generated Gaussian-mixture templates inspired by
their examples. These examples allow us to compare how the different
barycentric constructions interpolate between the same templates and to
evaluate whether the proposed SW analysis recovers the coefficients used
to generate each distribution.

In each experiment, we use three equally weighted empirical template
measures,
  $  \nu_j=\frac{1}{m}\sum_{k=1}^{m}\delta_{x_{j,k}}$,
    $j=1,2,3$,
with $m=2000$ particles for the animal silhouettes and $m=1500$ for the
Gaussian mixtures. Each Gaussian-mixture template has two components.
We evaluate barycentric coefficients on the triangular grid
\[
        \left\{(a,b,c)/4:\ a,b,c\in\{0,1,2,3,4\},\ a+b+c=4\right\}
    \subset\Delta_2.
\]
The three vertices display the original templates, while the remaining
$12$ grid points correspond to synthesized measures. The same templates
and coefficient vectors are used for every method.

We compare the proposed SW fixed-point synthesis
(Algorithm~\ref{alg: synth}) with the
sliced-Wasserstein method with Newton updates and the Radon Wasserstein
construction of \cite{boneel}, as well as with other transport geometries, namely, classical,  linearized, and
entropy-regularized Wasserstein barycenters. A review of these different transport-based barycenter constructions is provided below.
Figures~\ref{fig:animals_synthesis} and
\ref{fig:multimod_gauss_synthesis} show the resulting distributions.
Table~\ref{tab:synthesis-runtime-summary} reports synthesis runtimes.

For each synthesized measure $\mu$, we then apply the same proposed SW
analysis procedure (Algorithm~\ref{alg: analysis}) to estimate its
coefficients $\widehat\lambda$.
We compare these estimates with the generating coefficients $\lambda$
using the $\ell^1$-error $\|\lambda-\widehat\lambda\|_1$ and evaluate the empirical
SW stationarity residual
$(\widehat\lambda^\top\widehat{\mathcal A}(\mu)\widehat\lambda)^{\frac{1}{2}}$,
together with the analysis runtime. Table~\ref{tab:synthesis-analysis-summary}
summarizes these quantities over the $12$ non-vertex coefficient vectors.
They assess coefficient
recovery and compatibility with the SW barycentric model. For measures
synthesized using a different transport geometry, a larger SW residual
or coefficient discrepancy does not, by itself, indicate an inaccurate
solution of that method's own barycenter problem. 

\paragraph{Experimental specifications.}
For the synthetic comparisons, we use 2,000 particles per cloud for
animals and 1,500 for Gaussian mixtures and the single-Gaussian 3D
example. Our SW synthesis starts from the first reference cloud and
uses $L=40$ directions with at most 35 iterations for animals, and
$L=48$ with at most 40 iterations for both Gaussian examples.
Directions are sampled uniformly from the sphere and resampled at
every iteration. The step size multiplying the averaged displacement
field is $\tau=d$, where $d$ is the ambient dimension. Iterations stop
when the root-mean-square particle displacement falls below $10^{-6}$
or the iteration limit is reached. 

To recover
coordinates from the resulting barycenters, analysis uses a
separate set of $L=256$ directions per query,
held fixed throughout the quadratic programming (QP) optimization.
Its simplex-constrained QP is solved by projected gradient with step size 
$\eta=[2\max\{\|\widehat{\mathcal A}\|_2,10^{-15}\}]^{-1}$,
where $\|\widehat{\mathcal A}\|_2$ denotes the largest eigenvalue
of the positive-semidefinite Gram matrix,
an absolute successive-iterate tolerance of $10^{-12}$ in Euclidean
norm, and a limit of $10^4$ iterations.

\paragraph{Forward-Backward.}
For visualization of the forward-backward methodology, that is, synthesis of SW barycenters (forward map) and retrieving coordinates through analysis (backward or inverse map), we include Figure \ref{fig:forward-backward}.

In  Figure \ref{fig:forward-backward}, bottom panel, we include a pointcloud experiment in 3D by utilizing  Gaussian templates.
We consider three centered, anisotropic Gaussian distributions in
$\mathbb{R}^3$ with differently oriented  covariance
matrices, each represented by $1{,}500$ independently sampled points. In Figure \ref{fig:forward-backward} bottom panel, a 2D projection is used for display.

Figure \ref{fig:forward-backward} shows synthesis and analysis for animal shapes (top),
Gaussian mixtures (middle), and  3D Gaussians (bottom). Left:
 SW fixed-point synthesis (Algorithm \ref{alg: synth}). Right:  SW synthesis from \cite{boneel}. At each simplex point we see the true 
generating weights $\lambda$, and the estimated weights
$\widehat{\lambda}$.

\begin{figure}[ht!]
    \centering
    \includegraphics[width=1\linewidth]{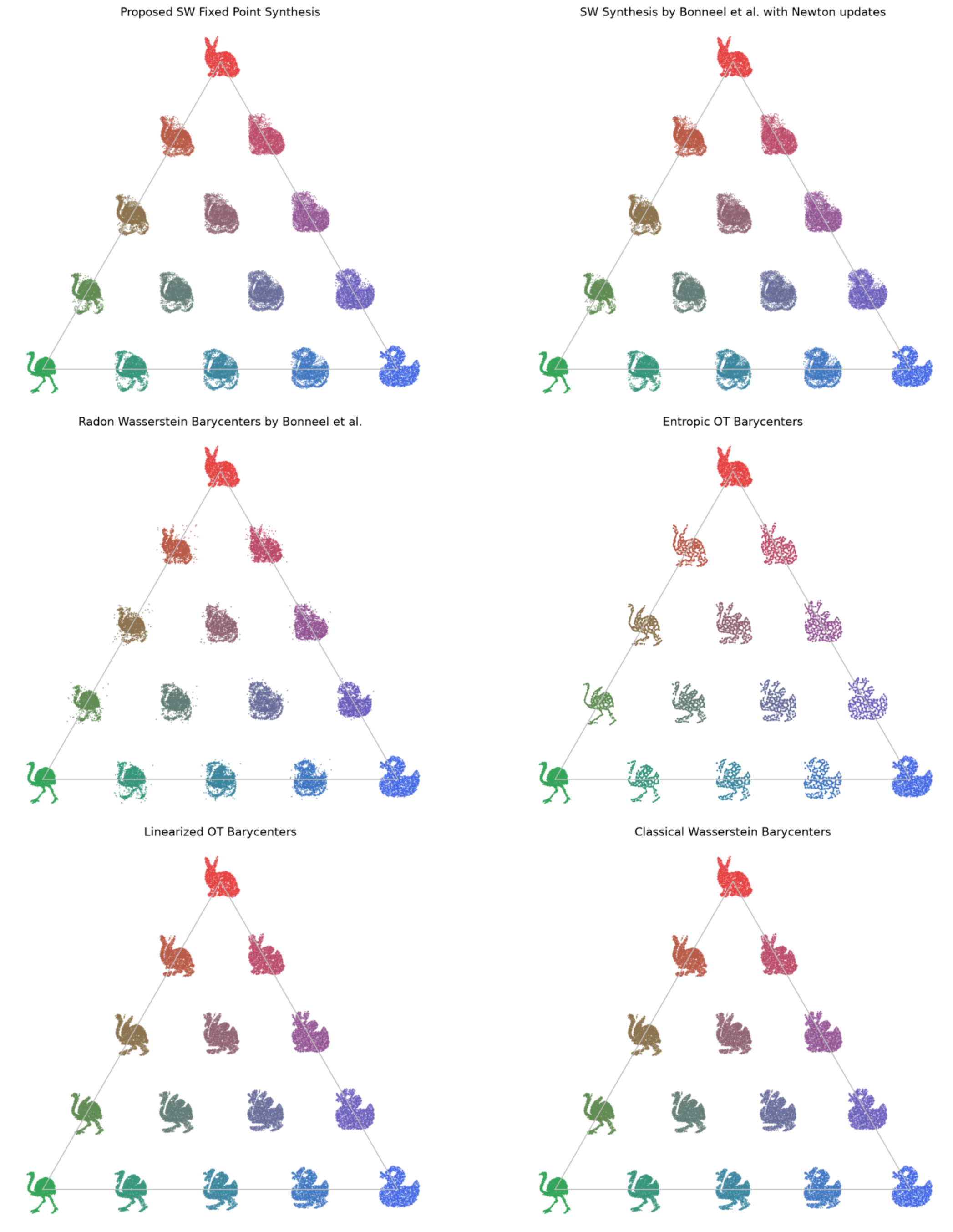}
    \caption{Animal-shape barycenters over a triangular grid of mixture coefficients.
The columns of each row show, from top to bottom, the proposed sliced
Wasserstein fixed-point synthesis and the sliced Wasserstein method of
Bonneel et al. with Newton updates; the Radon Wasserstein and entropic OT
barycenters; and the linearized OT and classical Wasserstein barycenters.
Each distribution contains $2000$ samples. The three vertices show the same
input measures for every method, while the remaining $12$ locations are
synthesized barycenters.}
    \label{fig:animals_synthesis}
\end{figure}

\begin{figure}[ht!]
    \centering
    \includegraphics[width=1\linewidth]{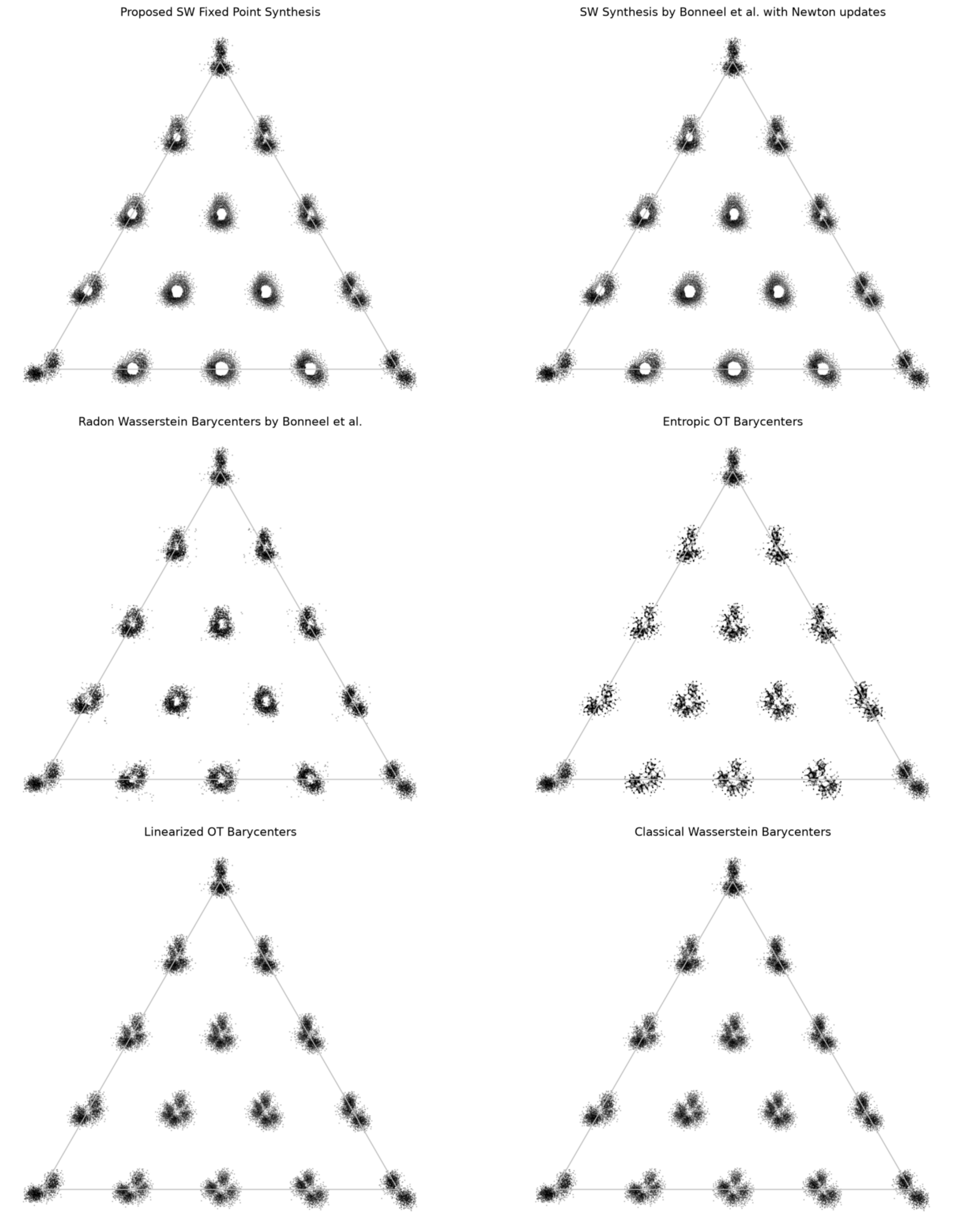}
    \caption{Multimodal Gaussian-mixture barycenters over a triangular grid of mixture
coefficients. The columns of each row show, from top to bottom, the proposed
sliced Wasserstein fixed-point synthesis and the sliced Wasserstein method
of Bonneel et al. with Newton updates; the Radon Wasserstein and entropic OT
barycenters; and the linearized OT and classical Wasserstein barycenters.
Each distribution contains $1500$ samples. The three vertices show the same
input Gaussian-mixture measures for every method, while the remaining $12$
locations are synthesized barycenters.}
    \label{fig:multimod_gauss_synthesis}
\end{figure}

\begin{figure}[ht!]
  \centering
  \includegraphics[width=\linewidth]{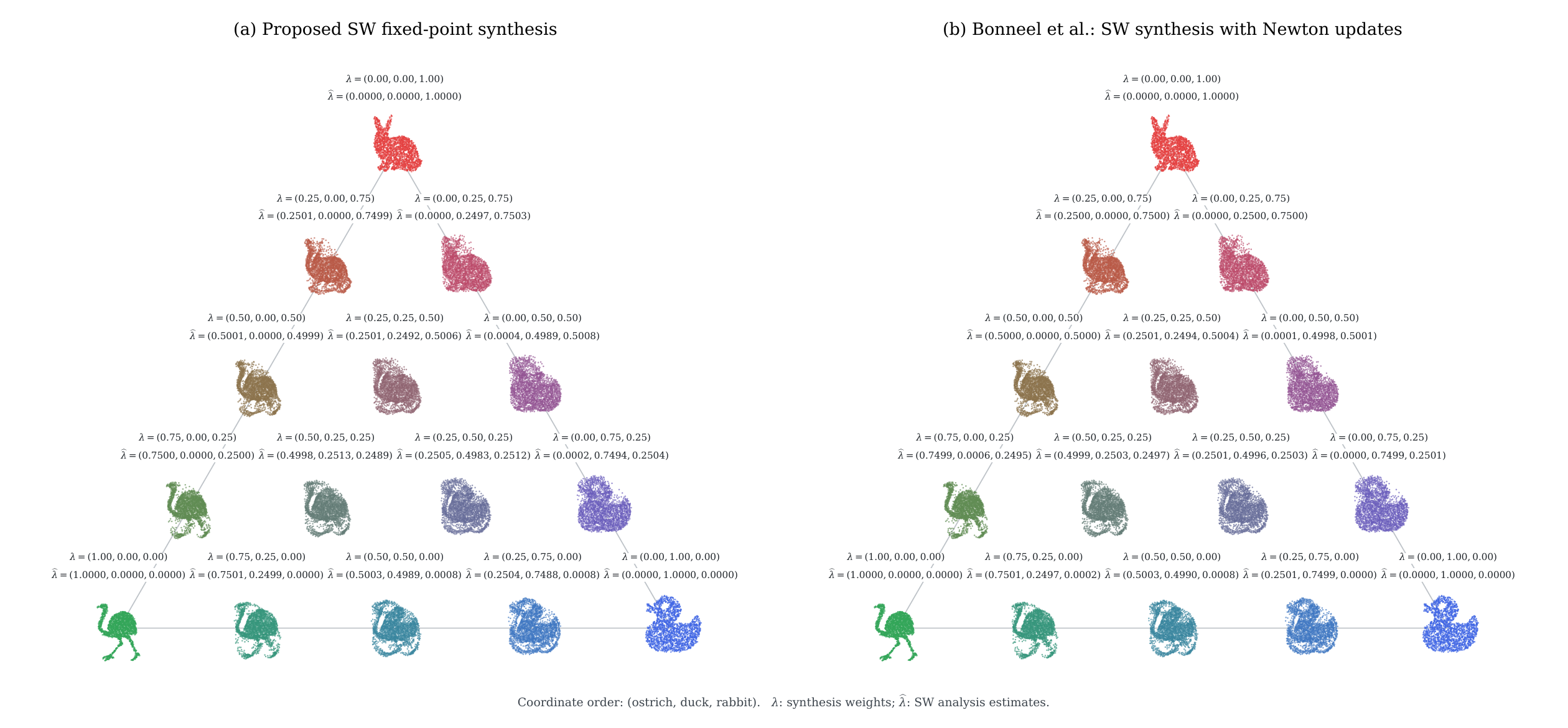}
  \includegraphics[width=\linewidth]{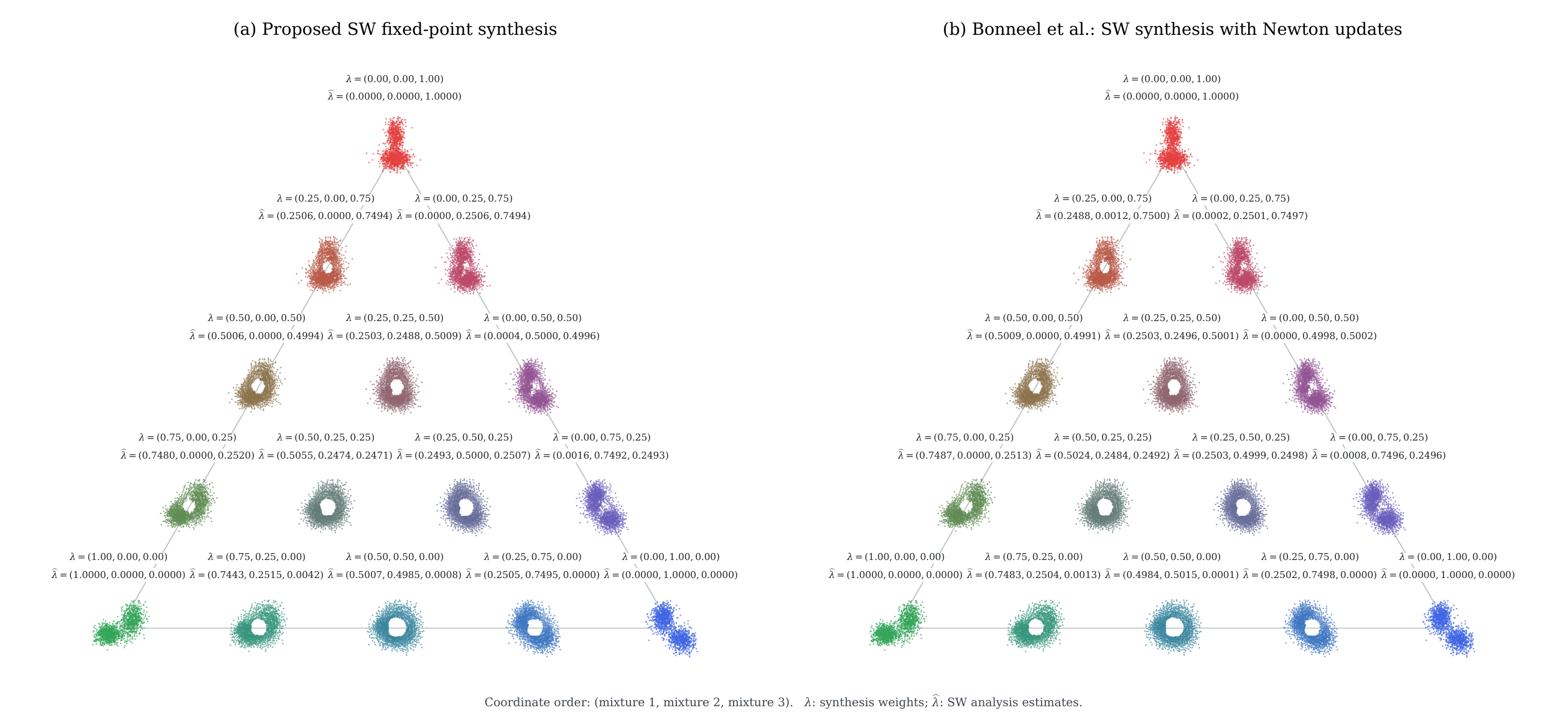}
  \includegraphics[width=0.\linewidth]{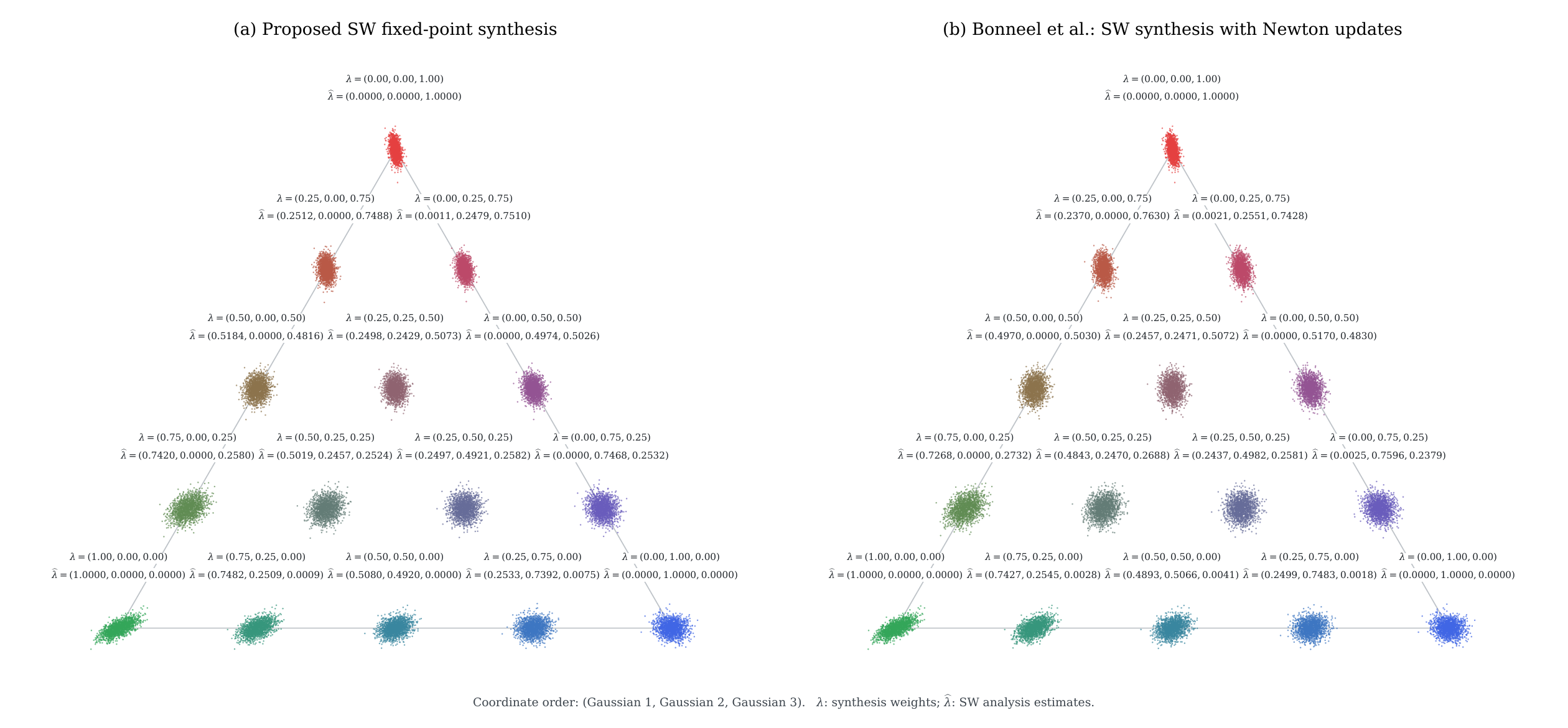}
  \caption{Synthesis and Analysis.}
  \label{fig:forward-backward}
\end{figure}

The following
paragraphs specify the constructions and numerical settings used in
our comparisons (Figures \ref{fig:animals_synthesis} and \ref{fig:multimod_gauss_synthesis}).

\paragraph{Classical Wasserstein Barycenters:}

The classical quadratic Wasserstein barycenter is defined by
\begin{equation}
    \overline{\nu}^{W_2}_{\lambda}
    \in
    \operatorname*{argmin}_{\mu\in\mathcal{P}_2(\mathbb{R}^d)}
    \frac{1}{2}
    \sum_{j=1}^{N}
    \lambda_j W_2^2(\mu,\nu_j).
    \label{eq:classical-w2-barycenter}
\end{equation}
This is the standard Wasserstein barycenter introduced in
\cite{agueh2011barycenters}.
We approximate~\eqref{eq:classical-w2-barycenter} using an equally weighted,
free-support measure $    \nu^{(n)}
    =
    \frac{1}{m}
    \sum_{i=1}^{m}\delta_{y_i^{(n)}}$.
At iteration \(n\), the optimal transport between \(\nu^{(n)}\) and each
reference measure \(\nu_j\) is computed as the assignment
\begin{equation}\label{eq:classical-assignment}
    \sigma_j^{(n)}
    \in
    \operatorname*{argmin}_{\sigma\in {S}_m}
    \sum_{i=1}^{m}
    \left\|
        y_i^{(n)}-x_{j,\sigma(i)}
    \right\|^2,
\end{equation}
where \({S}_m\) is the set of permutations of
\(\{1,\ldots,m\}\). We solve~\eqref{eq:classical-assignment} using SciPy's linear-assignment solver. The support points are then updated
by averaging their matched locations:
\begin{equation}
    y_i^{(n+1)}
    =
    \sum_{j=1}^{N}
    \lambda_j x_{j,\sigma_j^{(n)}(i)}.
    \label{eq:classical-fixed-point}
\end{equation}
Thus, each iteration alternates between computing $N$ 
optimal assignments and updating the barycenter support. 

To produce Figures \ref{fig:animals_synthesis} and \ref{fig:multimod_gauss_synthesis} together with Table \ref{tab:synthesis-runtime-summary}, we initialize the
iteration using the first reference cloud and perform at most \(K=18\)
iterations, stopping earlier if
\[
    \left(
        \frac{1}{m}
        \sum_{i=1}^{m}
        \left\|y_i^{(n+1)}-y_i^{(n)}\right\|^2
    \right)^{1/2}
    <10^{-6}.
\]
This assignment-based computation stores an \(m\times m\) cost matrix and
solves several assignment problems at every barycenter iteration. The computational complexity, in general, is $\mathcal{O}(KNm^3)
$.

\paragraph{Linearized Wasserstein Barycenters:}
We use the approach in 
\cite{werenski2025linearized}. First, we select a common base measure $    \mu_0
    =
    \frac{1}{m}
    \sum_{i=1}^{m}\delta_{z_i}$.

For each reference \(\nu_j\), we compute an optimal assignment map $T_j:\{z_1,\ldots,z_m\}
    \to
    \{x_{j,1},\ldots,x_{j,m}\}$
by solving
\begin{equation}
    T_j
    \in
    \operatorname*{argmin}_{T:\,T_{\#}\mu_0=\nu_j}
    \frac{1}{m}
    \sum_{i=1}^{m}
    \left\|z_i-T(z_i)\right\|^2.
    \label{eq:lot-reference-map}
\end{equation}
These maps provide the linear optimal transport representation of the
references around the common base measure.

The linearized barycenter has the closed form $    \overline{\nu}^{\mathrm{LOT}}_{\lambda}
    =
    \left(
        \sum_{j=1}^{N}\lambda_j T_j
    \right)_{\#}\mu_0.
$
For empirical measures, this means that its support points are simply
\begin{equation}
    y_i
    =
    \sum_{j=1}^{N}
    \lambda_j T_j(z_i),
    \qquad i=1,\ldots,m,
    \label{eq:linearized-particle-update}
\end{equation}
and hence $\overline{\mu}^{\mathrm{LOT}}_{\lambda}
    =
    \frac{1}{m}
    \sum_{i=1}^{m}\delta_{y_i}.$

In our experiments visualized in Figures \ref{fig:animals_synthesis} and \ref{fig:multimod_gauss_synthesis}, and Table \ref{tab:synthesis-runtime-summary}, the base or reference measure  \(\mu_0\) is an approximate equal-weight classical
Wasserstein barycenter of the reference measures (i.e., with $\lambda=(1/N,\dots, 1/N)$), computed using \(12\)
iterations of~\eqref{eq:classical-fixed-point}.
Computing the base measure and the maps \(T_j\) is treated as a one-time
setup cost. Once these maps have been cached, every new barycenter requires
only the weighted average in~\eqref{eq:linearized-particle-update}. This
explains the extremely small per-barycenter synthesis time in our
experiments. The linearized barycenter agrees with the classical
Wasserstein barycenter when the corresponding optimal maps satisfy the
compatibility conditions described in
\cite[Definition 4 and Proposition 2]{werenski2025linearized}; in general, however, the two
barycentric models are distinct.

\paragraph{Entropy-Regularized Wasserstein Barycenters:}

We implement the entropy-regularized barycenter fixed point developed in
\cite{mallery2025synthesis}. Given the current candidate $
    \nu^{(n)}
    =
    \frac{1}{m}
    \sum_{i=1}^{m}\delta_{y_i^{(n)}}$,
we compute an entropy-regularized coupling between \(\nu^{(n)}\) and every
reference \(\nu_j\) given by 
\begin{equation}
    \pi_{j,\varepsilon}^{(n)}
    =
    \operatorname*{argmin}_{\pi\in\Pi(\nu^{(n)},\nu_j)}
    \left\{
        \sum_{i,k}C_{j,ik}^{(n)}\pi_{ik}
        +
        \varepsilon
        \sum_{i,k}
        \pi_{ik}\bigl(\log\pi_{ik}-1\bigr)
    \right\}\, \text{ where } C_{j,ik}^{(n)}
    =
    \frac{1}{2}
    \left\|y_i^{(n)}-x_{j,k}\right\|^2.
    \label{eq:entropic-coupling}
\end{equation}
We compute~\eqref{eq:entropic-coupling} with Sinkhorn scaling. In
particular, we form $K_{j,ik}^{(n)}
    =
    \exp\left(
        -\frac{C_{j,ik}^{(n)}}{\varepsilon}
    \right)$
and alternately rescale its rows and columns to satisfy the uniform
marginals.

The entropy-regularized transport map is approximated by the barycentric
projection of the coupling:
\begin{equation}
    T_{j,\varepsilon}^{(n)}(y_i^{(n)})
    =
    \frac{
        \displaystyle
        \sum_{k=1}^{m}
        \pi_{j,\varepsilon,ik}^{(n)}x_{j,k}
    }{
        \displaystyle
        \sum_{k=1}^{m}
        \pi_{j,\varepsilon,ik}^{(n)}
    }.
    \label{eq:entropic-barycentric-projection}
\end{equation}
The support is updated using the fixed-point characterization
\begin{equation}
    y_i^{(n+1)}
    =
    \sum_{j=1}^{N}
    \lambda_j
    T_{j,\varepsilon}^{(n)}(y_i^{(n)}).
    \label{eq:entropic-fixed-point}
\end{equation}
Thus, the entropic maps are recomputed at every outer iteration. 

In Figures \ref{fig:animals_synthesis} and \ref{fig:multimod_gauss_synthesis}, and Table \ref{tab:synthesis-runtime-summary}, we use
\(80\) Sinkhorn iterations for each coupling and at most \(18\) outer
fixed-point iterations. The regularization parameters are $\varepsilon=0.002$ for the animal experiment, and $\varepsilon=0.005$ for the multimodal Gaussian-mixture experiment.
These relatively small values keep the entropic construction close to
unregularized transport.  Our implementation uses the regularized
transport cost itself, rather than the debiased Sinkhorn-divergence
variant also considered in \cite{mallery2025synthesis}.
Entropic
maps are barycentric projections of diffuse couplings, so several source
particles can be sent to almost the same conditional mean. Therefore, as reflected in our figures, the markers then
overlap. This can also be explained by the \emph{``contraction''} property stated in \cite[Lemma 9.2]{yang2023estimating}.

\paragraph{Sliced Wasserstein Barycenters in \cite{boneel}:}

Following \cite{boneel}, we approximate the spherical integral in \eqref{eq: SW} using fixed directions
\(\theta_1,\ldots,\theta_L\). In two dimensions, for instance, these are equispaced
directions on \([0,\pi)\). For a candidate point cloud
\(Y^{(n)}=(y_1^{(n)},\ldots,y_m^{(n)})\), the one-dimensional optimal
transport along \(\theta_\ell\) is found by sorting the projected candidate
points
$\langle y_i^{(n)},\theta_\ell\rangle$
and matching them with the sorted projected reference points $\langle x_{j,k},\theta_\ell\rangle$.
Let \(m_{j,\ell}^{(n)}(i)\) denote the resulting monotone match. The
finite-direction gradient field at \(y_i^{(n)}\) is
\begin{equation}
    g_i^{(n)}
    =
    \sum_{j=1}^{N}\lambda_j
    \frac{1}{L}
    \sum_{\ell=1}^{L}
    \left(
        \langle y_i^{(n)},\theta_\ell\rangle
        -
        \langle
            x_{j,m_{j,\ell}^{(n)}(i)},\theta_\ell
        \rangle
    \right)\theta_\ell.
    \label{eq:bonneel-sliced-gradient}
\end{equation}
The corresponding finite-direction Hessian is
\begin{equation}
    H_L
    =
    \frac{1}{L}
    \sum_{\ell=1}^{L}
    \theta_\ell\theta_\ell^{\mathsf T}.
    \label{eq:bonneel-hessian}
\end{equation}

We then use the Newton update
\begin{equation}
    y_i^{(n+1)}
    =
    y_i^{(n)}
    -
    H_L^{-1}g_i^{(n)}.
    \label{eq:bonneel-newton-update}
\end{equation}

For equispaced directions in two dimensions,
\(H_L=\frac{1}{2}I_2\) up to numerical precision. The animal experiment (Figure \ref{fig:animals_synthesis} and Table \ref{tab:synthesis-runtime-summary})
uses \(L=40\) directions and at most \(35\) Newton iterations. The
multimodal Gaussian-mixture experiment (Figure
\ref{fig:multimod_gauss_synthesis} and Table \ref{tab:synthesis-runtime-summary}) uses \(L=48\) directions and at
most \(45\) iterations.

\paragraph{Radon Wasserstein Barycenters in \cite{boneel}:}
The Radon construction of \cite{boneel} first treats each projection
direction independently. Let
\[
    \mathcal{R}\mu(t,\theta)
    =
    \int \delta\bigl(t-\langle x,\theta\rangle\bigr)\,d\mu(x)
\]
denote the Radon transform of $\mu$, interpreted in the distributional
sense. For fixed $\theta$, it represents the one-dimensional projected
measure $(P_\theta)_{\#}\mu$, or its density when one exists.

For each direction $\theta$, we compute the one-dimensional quadratic
Wasserstein barycenter of the projected reference measures. If
$
    Q_{j,\theta}(u)
    =
    F^{-1}_{(P_\theta)_{\#}\nu_j}(u)$,
 $u\in(0,1)$,
is the corresponding quantile function, then the barycenter quantile is
\begin{equation}
    \overline{Q}_{\lambda,\theta}(u)
    =
    \sum_{j=1}^{N}\lambda_j Q_{j,\theta}(u).
    \label{eq:radon-average-quantile}
\end{equation}
This quantile function determines the one-dimensional barycenter
   $ \overline{\rho}_{\lambda,\theta}
    =
    (\overline{Q}_{\lambda,\theta})_{\#}
    \operatorname{Unif}(0,1).
$

To reconstruct a two-dimensional image, we first represent these
one-dimensional barycenters on the detector grid. Collecting the
resulting mass profiles over all angles gives an array
$\overline{g}_{\lambda}(t,\theta)$, called the barycenter
\emph{sinogram}. We then apply a numerical inverse Radon transform.
This step reconstructs a two-dimensional image whose projections
approximately match the one-dimensional barycenters.
The independently computed barycenters need not be projections of
a common nonnegative image, so the reconstruction may contain negative
values. We clip these values to zero and normalize the result to have
unit mass, obtaining our numerical probability approximation
$\overline{\nu}^{\mathrm{Radon}}_{\lambda}$.

For the experiments in Figures~\ref{fig:animals_synthesis}
and~\ref{fig:multimod_gauss_synthesis} and
Table~\ref{tab:synthesis-runtime-summary}, each template measure is
first rasterized as a $96\times96$ discrete density. We then compute
its Radon transform at $96$ equispaced projection angles. Each column
of the resulting template sinogram represents the mass distribution
along one projection direction.

The template sinograms depend only on the reference measures and not
on the barycentric coefficients $\lambda$. We therefore compute and
store them once before constructing any barycenters, and report this
preprocessing as the one-time setup cost. For every new coefficient
vector $\lambda$, we (1) compute the barycenter quantile
\eqref{eq:radon-average-quantile} at each angle; (2) convert these
quantiles into one-dimensional mass distributions on the detector
grid, forming the barycenter sinogram; (3) numerically invert the
Radon transform to reconstruct a two-dimensional image; (4) clip
negative values to zero and normalize the image to have unit mass;
and (5) sample $m$ particles from this density for visualization.

\begin{table}[ht!]
  \centering
  \small
  \caption{Synthesis runtime.  ``Mean'' is
  the mean time for one of the 12 non-vertex barycenters, ``setup'' is reusable
  preprocessing, and ``total'' is setup plus synthesis of all 12 barycenters.
  The methods solve different barycentric problems, so the measurements
  document computational tradeoffs rather than a universal speed ranking.}
  \label{tab:synthesis-runtime-summary}
  \resizebox{\textwidth}{!}{  \begin{tabular}{lrrr@{\qquad}rrr}
    \toprule
    & \multicolumn{3}{c}{Animals ($m=2000$)}
    & \multicolumn{3}{c}{Gaussian mixtures ($m=1500$)} \\
    \cmidrule(lr){2-4}\cmidrule(lr){5-7}
    Method
      & Mean (s) & Setup (s) & Total (s)
      & Mean (s) & Setup (s) & Total (s) \\
    \midrule
    Proposed SW fixed point
      & 0.8297 & 0 & 9.9558
      & 1.0701 & 0 & 12.8417 \\
    Bonneel SW with Newton updates
      & 0.7543 & 0 & 9.0517
      & 1.1063 & 0 & 13.2760 \\
    Radon Wasserstein
      & 0.0770 & 0.2063 & 1.1301
      & 0.0920 & 0.2094 & 1.3131 \\
    Entropic OT
      & 20.6850 & 0 & 248.2203
      & 12.0453 & 0 & 144.5437 \\
    Linearized OT
      & $5.6042\!\times\!10^{-5}$ & 54.9266 & 54.9273
      & $1.2556\!\times\!10^{-4}$ & 43.4526 & 43.4541 \\
    Classical Wasserstein
      & 30.9526 & 0 & 371.4313
      & 21.6746 & 0 & 260.0952 \\
    \bottomrule
  \end{tabular}}
\end{table}

\begin{table}[ht!]
  \centering
  \small
  \caption{Summary of SW analysis using Algorithm \ref{alg: analysis} applied to barycenters synthesized by six
  methods (means over the same 12 non-vertex coefficient vectors).  Weight
  error is measured against the coefficients supplied to the corresponding
  synthesis method.  The residual is the empirical SW stationarity residual
  evaluated at the coefficients returned by the analysis algorithm.}
  \label{tab:synthesis-analysis-summary}
  \resizebox{\textwidth}{!}{  \begin{tabular}{llrrr}
    \toprule
    Dataset & Synthesis method
      & Mean $\ell^1$ weight error
      & Mean residual
      & Mean analysis time (s) \\
    \midrule
    Animals & Proposed SW fixed point
      & $1.4158\times10^{-3}$ & $3.5933\times10^{-4}$ & 0.1291 \\
    & Bonneel SW with Newton updates
      & $6.2187\times10^{-4}$ & $3.3234\times10^{-4}$ & 0.1181 \\
    & Radon Wasserstein
      & $9.3748\times10^{-2}$ & $1.1203\times10^{-2}$ & 0.1073 \\
    & Entropic OT
      & $8.5506\times10^{-2}$ & $8.5137\times10^{-3}$ & 0.1044 \\
    & Linearized OT
      & $2.2627\times10^{-2}$ & $3.1012\times10^{-3}$ & 0.1032 \\
    & Classical Wasserstein
      & $2.1671\times10^{-2}$ & $3.0865\times10^{-3}$ & 0.1141 \\
    \midrule
    Gaussian mixtures & Proposed SW fixed point
      & $3.4900\times10^{-3}$ & $1.0522\times10^{-3}$ & 0.1593 \\
    & Bonneel SW with Newton updates
      & $1.8737\times10^{-3}$ & $5.1431\times10^{-4}$ & 0.1323 \\
    & Radon Wasserstein
      & $1.0529\times10^{-1}$ & $1.5438\times10^{-2}$ & 0.1313 \\
    & Entropic OT
      & $2.9572\times10^{-2}$ & $1.4433\times10^{-2}$ & 0.1353 \\
    & Linearized OT
      & $3.1230\times10^{-2}$ & $1.0220\times10^{-2}$ & 0.1369 \\
    & Classical Wasserstein
      & $2.5721\times10^{-2}$ & $9.8864\times10^{-3}$ & 0.1478 \\
    \bottomrule
  \end{tabular}}
\end{table}

\paragraph{Interpretation of the Synthesis Comparisons.}
Both SW synthesis procedures, Algorithm \ref{alg: synth} and SW synthesis with Newton updates proposed in \cite{boneel}, use projected monotone matching and yield
similar computational scales and small SW stationarity residuals.
For outputs of the two SW synthesis schemes, mean $\ell^1$ weight errors range
from $6.22\times10^{-4}$ to $3.49\times10^{-3}$; for barycenters generated in
the other tested transport geometries, they range from
$2.17\times10^{-2}$ to $1.05\times10^{-1}$.  This separation indicates
compatibility with the SW model.  Larger discrepancies under other geometries
do not imply that their own synthesis objectives were poorly optimized, rather that the geometries and barycentric spaces are different.
Proposed synthesis takes $0.83$--$1.07$ seconds per
barycenter, comparable to the tested SW/Newton method.  Radon and amortized
linearized synthesis are faster per output, whereas the tested entropic and
classical Wasserstein implementations are slower.  
The analysis timings in Table~\ref{tab:synthesis-analysis-summary} are
$0.10$--$0.16$ seconds per query.

\subsection{Stability Experiment of the Analysis Algorithm under Point-Cloud Corruptions}\label{app:stab}

We next use the ModelNet-C component of PointCloud-C
\citep{ren2022benchmark}, derived from ModelNet40 \citep{wu20153d},
and the five categories used by \citet{mallery2025synthesis}:
airplane, bed, guitar, monitor, and vase. Each of ten randomized trials
uses three clean references per class, $L=32$ directions, and 20 test
objects per class, disjoint from the references. 

The experiment in this section investigates whether the estimated coordinates remain stable under point-cloud corruptions (see Figure \ref{fig:corrup}).

Every cloud is centered and analyzed in its clean form ($X$) and under five
corruptions ($\widetilde{X}$): two levels of local
point removal on concentrated neighborhoods, global point dropout throughout the cloud, jitter (perturb point coordinates), and local
addition (insert a concentrated group of points).

These objects are not assumed to be exact
SW barycenters of a dictionary; their estimated
coordinates are understood as a way of representing data through the SW BCM, and 
the purpose is to test the \emph{analysis map} that retrieves coordinates $\widehat{\lambda}(X)=(\widehat{\lambda_1}(X),\dots,\widehat{\lambda_N}(X))$, where, in this experiment, $N=15$ corresponding to 5 classes and 3 randomly selected clean templates per class. 

As explained in the main text, summing within a class $c$ gives its \emph{class mass}, $\widehat p_c(X)=\sum_{j\in\mathcal I_c}\widehat\lambda_j(X),$
where $\mathcal I_c$ indexes that class's references. 

To assess stability, we compare each object's class masses before  and
after corruption using $\|\widehat p(\widetilde X)-\widehat p(X)\|_1$.
Figure~\ref{fig:corrup} shows smaller mean
changes under jitter ($0.067$) and global dropout ($0.132$) than under
local removal or addition ($0.275$--$0.402$). The class-level
representation is therefore more stable under jitter and global dropout
at the tested severities. %

\begin{figure}[H]
    \centering
    \includegraphics[width=0.75\linewidth]{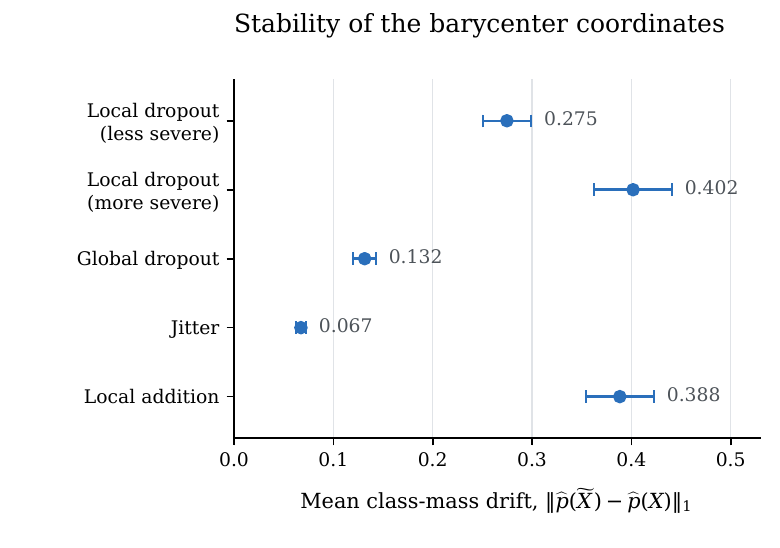}
    \vspace{-0.25in}
    \caption{Stability of the barycenter coordinates from the SW BCM classes on ModelNet-C with three clean
  templates per class (5 classes in total). We visualize how the assigned class-mass changes/drifts under different corruptions. Precisely, each point is the mean paired drift
  $\|\widehat p(\widetilde X)-\widehat p(X)\|_1$, averaged over the same
  100 objects in each trial; smaller values indicate greater stability.
  Error bars show 95\% Student-$t$ confidence intervals over ten randomized
  trials.
  The less and more severe local dropout settings remove 200 and
300 of the original 1,024 points, respectively. The other corruptions
use the highest severity setting of ModelNet-C~\citep{ren2022benchmark},
corresponding to index 4 in the official implementation's zero-based
indexing.
  }
    \label{fig:corrup}
\end{figure}

A plausible explanation for these results is that the corruptions affect the spatial
distribution of mass differently.  Global dropout samples points throughout the
object: although only $256$ of the original $1{,}024$ points remain, uniform
thinning preserves the empirical distribution in expectation before centering.
Jitter instead introduces small, independent, zero-mean displacements%
~\citep{ren2022benchmark}.
Both perturbations can therefore leave the overall projected distributions
comparatively close to those of the clean object.  Local dropout removes
spatially concentrated neighborhoods and may erase distinctive shape parts;
local addition places extra mass near a few locations.  In the latter case,
the $500$ added points account for approximately $32.8\%$ of the resulting
uniform empirical measure.  These localized changes can alter projected
quantiles across many directions and hence the transport fields used to
recover the coordinates.

\clearpage
\newpage
\bibliographystyle{apalike}
\bibliography{references}

\end{document}